\documentclass{siamltex}
\def\Bbb#1{\mbox{\bb #1}}
\def\reals{ {\Bbb R}}
\def\complex{ {\Bbb C}}
\def\naturals{ {\Bbb N}}

\usepackage{amsfonts}
\usepackage{mathrsfs}
\usepackage{amsmath}
\usepackage{amssymb}
\usepackage{graphicx}
\usepackage{subfigure}
\usepackage{epsfig}

\newcommand{\bend}{\hspace*{0ex} \hfill \hbox{\vrule height
    1.5ex\vbox{\hrule width 1.4ex \vskip 1.4ex\hrule  width 1.4ex}\vrule
    height 1.5ex}}

\newtheorem{example}[theorem]{Example}
\newtheorem{remark}[theorem]{Remark}
\newtheorem{prop}[theorem]{Proposition}
\newtheorem{cor}[theorem]{Corollary}

\newcommand{\iu}{{\rm i}}

\def\bea{\begin{eqnarray*}}
\def\eea{\end{eqnarray*}}
\def\bean{\begin{eqnarray}}
\def\eean {\end{eqnarray}}

\begin{document}

\thispagestyle{empty}
\bibliographystyle{siam}

\title{Decay properties of spectral projectors  \\ 
with applications to electronic structure\thanks{Work supported by 
National Science Foundation grants 
DMS-0810862 and DMS-1115692
and by a grant of the University Research Committee
of Emory University}.}

\author {
Michele Benzi\thanks{Department of Mathematics and Computer
Science, Emory University, Atlanta, Georgia 30322, USA
(benzi@mathcs.emory.edu).} \and
Paola Boito\thanks{Department of Mathematics and Computer
Science, Emory University, Atlanta, Georgia 30322, USA. Current
address: DMI-XLIM UMR 7252 Universit\'e de Limoges - CNRS, 
123 avenue Albert Thomas, 87060 Limoges Cedex, France
(paola.boito@unilim.fr).} \and
Nader Razouk\thanks{Department of Mathematics and Computer Science,
Emory University, Atlanta, Georgia 30322, USA. Current
address: Ernst \& Young GmbH Wirtschaftspr\"ufungsgesellschaft,
Arnulfstra\ss e 126, 80636 M\"unchen, Germany (nrazouk@gmail.com).}}

\maketitle

\markboth{{\sc M.~Benzi, P.~Boito, and N.~Razouk}}{{Decay properties
of spectral projectors}}

\begin{abstract}
Motivated by applications in quantum chemistry and solid state physics,
we apply general results from approximation theory and matrix analysis
to the study of the decay
properties of spectral projectors associated with large and
sparse Hermitian matrices.
 Our theory leads to a rigorous proof 
of the exponential off-diagonal decay 
(``nearsightedness'') for the density matrix of gapped systems
at zero electronic temperature in both orthogonal and
non-orthogonal representations, thus providing a firm theoretical
basis for the possibility of linear scaling methods
in electronic structure calculations for non-metallic systems.
We further discuss
the case of density matrices for metallic systems
at positive electronic temperature. 
A few other possible applications are also discussed.
\end{abstract}



\section{Introduction}\label{sec:Introduction}
The physical and chemical properties of materials are largely
determined by the electronic structure of the atoms and 
molecules found in them. In all but the simplest cases,
the electronic structure can only
be determined approximately, and since the late 1920's
a huge amount of work
has been devoted to finding suitable 
approximations and numerical methods for solving this
fundamental problem. 
Traditional methods for electronic structure computations
are based on the solution of generalized eigenvalue problems 
(\lq\lq diagonalization\rq\rq) for a sequence of large
Hermitian matrices, known as one-particle Hamiltonians. The 
computational cost of this
approach scales
cubically in the size $n$ of the problem, which is in
turn determined by the number of electrons in the system.
For large systems, the costs become prohibitive; this
is often referred to as \lq\lq the $O(n^3)$ bottleneck\rq\rq\
in the literature.

In the last two decades, a number of researchers have been
developing approaches that are capable in many 
cases to achieve \lq\lq optimal\rq\rq\
computational complexity: the computational
effort scales linearly in the number of electrons, 
leading to better performance for sufficiently large 
systems and making the 
electronic structure problem tractable for
large-scale systems. These methods, 
often referred
to as \lq\lq $O(n)$ methods,\rq\rq\  apply mostly
to insulators. 
They avoid diagonalization by computing instead
the {\em density matrix},
a matrix which encodes all the important physical properties of
the system.
For insulators at zero temperature, this is the spectral
projector onto the invariant subspace associated with
the eigenvalues of the Hamiltonian falling below a certain
value. For systems at positive temperatures, the density matrix 
can be expressed as a smooth function of the Hamiltonian.

The possibility of developing such methods
rests on a deep property of electronic matter, called
\lq\lq nearsightedness\rq\rq\ by W.~Kohn \cite{kohn}.
Kohn's \lq\lq Nearsightedness Principle\rq\rq\ 
expresses the fact that for a large class of systems 
the effects of disturbances, or 
perturbations, 
remain localized and thus do not propagate beyond
a certain (finite) range; in other words,
far away parts
of the system do not \lq\lq see\rq\rq\ each other. 
Mathematically, this property translates into the 
rapid off-diagonal decay in the density matrix. 
The fast fall-off in the density matrix entries has been often
assumed without proof, or proved only in special cases.
Moreover,
the precise dependence of the rate of decay on properties of
the system (such as the band gap in insulators
or the  temperature in metallic systems) 
has been the subject of much discussion.

The main goal of this paper is to provide a rigorous mathematical
foundation for linear scaling methods in electronic structure
computations. We do this by
deriving estimates,
in the form of decay bounds, for the entries of 
general density
matrices for insulators, and for metallic systems at positive
electronic temperatures. We also address the question of the
dependence of the rate of decay on the band gap and on the 
temperature. Although immediately susceptible of physical interpretation,
our treatment is purely mathematical. By stripping
the problem down to its essential features and working
at the discrete level, we are 
able to develop an abstract
theory covering nearly all types of systems
and discretizations encountered in actual electronic
structure problems.

Our results are based on a general theory
of decay for the entries in analytic functions of sparse matrices,
initially proposed in \cite{benzigolub,benzirazouk,nadersthesis} 
and further 
developed here. The theory is based on classical approximation
theory and matrix analysis. A bit of functional analysis will
be used when considering a simple model of \lq\lq metallic
behavior,\rq\rq\ for which the decay in the density matrix
is very slow.

The approach described in this paper has a number of potential
applications beyond electronic structure computations, and
can be applied to any problem involving functions of large
matrices where \lq\lq locality of interaction\rq\rq\ plays
a role.  
Towards the end of the paper we
briefly review the possible use of decay bounds
in the study of correlations in quantum statistical mechanics and
information theory, 
in the analysis of complex networks, and in some problems in
classical numerical linear algebra, like the computation of invariant
subspaces of symmetric tridiagonal matrices. The discussion of
these topics will be
necessarily brief, but we hope it will stimulate further work
in these areas.

In this paper we are mostly concerned with the theory behind
$O(n)$ methods rather than with specific algorithms.
Readers who are interested in the computational aspects
should consult any of the many recent
surveys on algorithms for electronic structure computations;
among these, \cite{bowler2,niklasson,RRS11,saad} are especially 
recommended.

The remainder of the paper is organized as follows.
Section \ref{materials} provides some background on
electronic structure theory. The formulation of the
electronic structure problem in terms of spectral projectors
is reviewed in section \ref{sec:DM}.
A survey of previous, related work on decay estimates
for density matrices is given in
section \ref{related}.
In section \ref{norm} we formulate our basic assumptions
on the matrices (discrete Hamiltonians) considered in this paper,
particularly their normalization and asymptotic behavior
for increasing system size ($n\to \infty$). 
The approximation (truncation) of matrices with decay properties
is discussed in section \ref{trunc}.
A few general properties of orthogonal projectors are
established in section \ref{proj}. 
The core of the paper
is represented by section \ref{sec_main}, where various
types of decay bounds for spectral projectors are stated and proved.
In section \ref{overlap} we discuss the transformation
to an orthonormal basis set.
The case of vanishing gap is discussed in section \ref{metals}.
Other applications of our results and methods
are mentioned in section \ref{other}.
Finally, concluding remarks and some open problems
are given in section \ref{concl}.

\section{Background on electronic structure theories}\label{materials}
In this section we briefly discuss the basic principles underlying
electronic structure theory. For additional details the reader is
referred to, e.g., \cite{burke,LeBris,march,martin,saad,SO}. 

Consider a physical system formed by a number
of nuclei and $n_e$ electrons in three-dimensional (3D) space.
The
time-independent Schr\"odinger equation for the system is the eigenvalue problem
\begin{equation}\label{schrodinger1}
{\mathcal H}_{\rm tot} \Psi_{\rm tot}=E_{\rm tot}\Psi_{\rm tot},
\end{equation}
where $\mathcal{H}_{\rm tot} $ is the many-body Hamiltonian operator,  
$E_{\rm tot} $ is the total energy and
the functions $\Psi_{\rm tot}$ are the eigenstates of the system.

The Born--Oppenheimer approximation allows us to separate the 
nuclear and electronic coordinates.
As a consequence, we only seek to solve the quantum mechanical 
problem for the electrons, considering the nuclei as sources of 
external potential. Then the electronic part of equation 
\eqref{schrodinger1} can be written as
\begin{equation}\label{schrodinger2}
{\mathcal H} \Psi=E\Psi,
\end{equation}
where $E$ is the electronic energy and the eigenstates $\Psi$ 
are functions of $3n_e$ spatial coordinates and $n_e$ (discrete) spin coordinates.

We denote spatial coordinates as ${\bf r}$ and the spin 
coordinate as $\sigma$; each electron is then defined by 
$3+1$ coordinates ${\bf x}_i=\left(\begin{array}{c}{\bf r}_i\\ 
\sigma_i\end{array}\right)$, and wavefunctions  
are denoted as  $\Psi({\bf x}_1,\dots,{\bf x}_{n_e})$.
Then the electronic Hamiltonian operator in \eqref{schrodinger2} 
can be written as
$$
{\mathcal H} = T + V_{\rm ext} + V_{\rm ee},
$$
where
$T=-\frac{1}{2}\nabla^2$ is the kinetic energy, $V_{\rm ext}$ is 
the external potential (i.e., the potential due to the nuclei) 
and $V_{\rm ee}=\frac{1}{2}\sum_{i\neq j}^{n_e}\frac{1}{|{\bf r}_i-{\bf
r}_j|}$ is the potential due to the electron-electron 
repulsion.\footnote{As is customary in physics, we use here 
atomic units, that is, $e^2=\hbar=m=1$, with $e=$ electronic 
charge, $\hbar=$ reduced Planck's constant and $m=$ electronic mass. }
Moreover, the ground-state energy  is given by
$$
E_0=\min_\Psi \langle\Psi|{\mathcal H}|\Psi\rangle,
$$
where the minimum is taken over all the normalized antisymmetric 
wavefunctions (electrons being Fermions, their wavefunction 
is antisymmetric). 
The electronic density is defined as
$$\rho({\bf r})=n_e\sum_\sigma \int d{\bf x}_2\dots\int d{\bf x}_{n_e}|
\Psi({\bf r},\sigma, {\bf x}_2,\dots,{\bf x}_{n_e})|^2.
$$
In this expression, the sum over $\sigma$ is the sum over the
spin values of the first electron, 
while integration with respect to ${\bf x}_i$, 
with $2\le i\le n_e$, denotes the
integral over $\reals^3$ and sum over 
both possible spin values for the $i$th electron.

Observe that \eqref{schrodinger2} is a many-particle equation 
that cannot be separated into several one-particle equation 
because of the term $V_{\rm ee}$. Of course, being able to turn 
\eqref{schrodinger2} into a separable equation would simplify 
the problem considerably, since the number of
unknowns per equation would drop from $3n_e+n_e$ to $3+1$. 
This is the motivation for one-electron methods.

 For non-interacting particles, the many-body eigenstates 
$\Psi({\bf x}_1,\dots,{\bf x}_{n_e})$ can be written as
 Slater determinants of occupied orbitals $\phi_1({\bf x}_1),
\dots,\phi_{n_e}({\bf x}_{n_e})$, 
$$\Psi({\bf x}_1,\dots,{\bf x}_{n_e}) = \frac{1}{\sqrt{n_e!}}
\left |\begin{matrix} \phi_1 ({\bf x}_1) & \ldots & \phi_{n_e} ({\bf x}_1) \cr
                      \vdots             & \ddots &  \vdots \cr
                      \phi_1({\bf x}_{n_e}) & \ldots & \phi_{n_e}({\bf x}_{n_e})\cr
       \end{matrix} \right |,
$$
where each orbital satisfies a single-particle eigenstate 
equation ${\mathcal H}_i \phi_i=E_i\phi_i$. 
In general, the name \lq\lq one-particle method\rq\rq\  is used also 
when self-consistent terms (e.g., involving the density) 
are present in ${\cal H}_i$; in this case, the equations are solved 
iteratively, computing at each step the solution to a single-particle 
problem and then filling the lowest eigenstates with one 
electron each, to form a Slater determinant. However, some of 
the properties of a true non-interacting system (such as the 
fact that the energy is the sum of the eigenvalues of 
occupied states) are lost.

\hspace{-0.05in} A fundamental example of one-particle method is density 
functional theory (DFT). The main idea 
behind DFT consists in rewriting the ground-state energy as 
a density functional rather than a wavefunction functional. 
Indeed, the first Hohenberg--Kohn theorem \cite{HK64} 
states that the potential is uniquely
(up to a constant) determined by the ground-state density 
$\rho({\bf r})$. In other words, the system can be seen as 
characterized by the density rather than by the potential. 
Moreover, the ground-state density of a system with given 
external potential can be computed by minimizing a suitable 
energy functional of $\rho$ (second Hohenberg--Kohn theorem).

While of crucial theoretical importance, though, these results 
do not give a recipe for computing
electronic structures. The next important step comes with 
the Kohn--Sham construction \cite{KS65}: roughly speaking, one replaces 
the original, non-separable system with a fictitious system 
of non-interacting
electrons that have exactly the same density as the original 
system. The single-particle equations for the 
Kohn--Sham system are (neglecting spin):
$$
\left(-\frac12\nabla^2+V({\bf r})\right)\psi_i({\bf r})=\varepsilon_i\psi_i({\bf r}),
$$
where the $\psi_i$'s are the Kohn--Sham orbitals and $V({\bf r})$ 
is the single-electron potential. The associated density is
$$
\rho ({\bf r})=\sum_{i=1}^{n_e}|\psi_i({\bf r})|^2.
$$
The single-particle potential $V({\bf r})$ can be written as
$$
V({\bf r})=V_{\rm ext}({\bf r})+\int_{\reals^3} \frac{\rho({\bf r})}
{|{\bf r}-{\bf r'}|}\,d{\bf r}' +V_{xc}[\rho]({\bf r}),$$
where the term $V_{xc}[\rho]({\bf r})$ is called {\em exchange-correlation 
potential} and depends on the density.
It is important to point out that the Kohn--Sham construction 
is not an approximation, in that the
Kohn--Sham equations are exact and yield the exact density.

On the other hand, the exchange-correlation energy is not known 
in practice and needs to be approximated. In the local density 
approximation (LDA) framework, for instance, the exchange energy 
is based on the energy of a uniform electron gas. Introducing 
spin allows for a more refined approximation (LSDA, or local 
spin-density approximation). One may also include gradient corrections, thus
obtaining the so-called generalized gradient approximation (GGA).

The solution of Kohn--Sham equations is usually computed 
via self-consistent iterations. The iterative process begins 
with an approximation for the density; the associated approximate 
exchange-correlation potential is injected in the Kohn--Sham 
equations. The output density is then used to form
a new approximation of the potential. The process continues 
until the update term for the density
or the potential becomes negligible. Observe that the basic 
building block of this computational technique is the solution 
of an eigenvalue problem for non-interacting particles.

Electrons at the lowest atomic-like levels (`core' electrons) do not change much 
their state within chemical processes. For this reason, many 
computational techniques do not consider them explicitly, and replace 
instead the  Coulomb attraction of the nucleus with a potential 
(called pseudopotential) that includes the effect of the core electrons 
on the valence electrons. This approach is always employed when 
using plane waves as a basis for wavefunctions, since the number 
of plane waves required to represent core electrons is prohibitive.

\section{Density matrices}\label{sec:DM} 
As mentioned earlier, 
conventional methods for electronic structure 
calculations 
require the repeated solution of linear eigenvalue problems 
for a one-electron Hamiltonian operator
of the form $\mathcal{H}=-\frac{1}{2}\nabla^2+V({\bf r})$. 
In practice, operators are discretized by grid methods 
or via Galerkin projection
onto the finite-dimensional subspace spanned by a set of 
basis functions $\{\phi_i\}_{i=1}^n$. When linear combinations
of atom-centered Slater or Gaussian-type functions (see below)
are employed,
the total number of basis functions is
$n\approx n_b\cdot n_e$, where
$n_e$ is the number of (valence) electrons in the system and $n_b$ 
is a small or moderate integer related to the 
number of basis functions per atom. 
Traditional electronic structure algorithms diagonalize 
the discrete Hamiltonian 
resulting in algorithms with $O(n_e^3)$ (equivalently,
$O(n^3)$)  operation count
\cite{LeBris,martin,saad}. In these approaches, a sequence of
generalized eigenproblems of the form
\begin{equation}\label{gen_eig}
H\psi_i = \varepsilon_i S\psi_i, \quad 1\le i\le n_e,
\end{equation}
is solved, where $H$ and $S$ are, respectively,
the discrete Hamiltonian and the overlap matrix relative
to the basis set $\{\phi_i\}_{i=1}^n$. 
 The eigenvectors $\psi_i$ in \eqref{gen_eig}
are known as the {\em occupied states}, and correspond
to the $n_e$ lowest generalized eigenvalues   
$\varepsilon_1\leq \cdots \leq \varepsilon_{n_e}$, the
{\em occupied levels}. 
 The overlap matrix $S$ is
just the Gram matrix associated with the basis set: $S_{ij} = \langle \phi_j,
\phi_i\rangle$ for all $i,j$, where 
$\langle \cdot,\cdot \rangle$
denotes the standard $L^2$-inner product. In Dirac's {\em bra-ket} notation,
which is the preferred one in the physics and
chemistry literature, one writes $S_{ij} = \langle \phi_i |
\phi_j \rangle$.  
For an
orthonormal basis set, $S=I_n$ (the $n\times n$ identity matrix)
and the eigenvalue problem \eqref{gen_eig} is a standard one.

Instead of explicitly 
diagonalizing the discretized Hamiltonian $H$, one may reformulate
the problem in terms of the density operator $P$, which 
is the $S$-orthogonal projector\footnote{That is, orthogonal with respect
to the inner product associated with $S$.} onto the $H$-invariant subspace 
corresponding
to the occupied states, that is, the subspace
spanned by the $n_e$ eigenvectors $\psi_i$ in 
\eqref{gen_eig}. Virtually all quantities of interest in 
electronic structure theory can be computed as functionals of
the density matrix $P$; see, e.g., \cite{challacombe,NS00,niklasson}.
It is this reformulation of the problem that allows for the 
development
of potentially more efficient algorithms for electronic structure, including
algorithms that asymptotically require only $O(n_e)$ 
(equivalently, $O(n)$) arithmetic operations
and storage. Most current methodologies, 
including Hartree--Fock, Density Functional Theory (e.g., Kohn--Sham),
and hybrid schemes (like BLYP) involve  
self-consistent field (SCF) iterations, in which
the density matrix $P$ must be computed at each SCF step, 
typically with increasing accuracy as the outer iteration
converges; see, e.g., \cite{LeBris,yang}.

As stated in section \ref{sec:Introduction},
in this paper we use some classical results from
polynomial approximation theory and matrix analysis to
provide a mathematical foundation for linear scaling
electronic structure calculations for a very broad class
of systems.
We assume that the basis functions $\phi_i$ are localized, i.e.,
decay rapidly outside of a small region. Many of the most popular
basis sets used in quantum chemistry, such as 
Gaussian-type orbitals (GTO), which are functions of the form 
$$\phi\, (x,y,x) = C\,x^{n_x} y^{n_y} z^{n_z}{\rm e}^{-\alpha r^2}\,,$$
where $C$ is a normalization constant,
satisfy this requirement \cite{LeBris}.
For systems with sufficient separation between atoms, 
this property implies a fast off-diagonal decay of the entries 
of the Hamiltonian matrix; moreover, a larger distance 
between atoms corresponds to a faster decay of matrix entries 
\cite[page 381]{LeBris}. 
If the entries that fall below a given (small) truncation 
tolerance are set to zero, the 
Hamiltonian turns out to be a sparse matrix.

Decay results are especially easy to
state in the banded case,\footnote{A square matrix
$A=\left (A_{ij}\right )$ is said to be $m$-banded if
$A_{ij}=0$ whenever $|i-j|>m$; for instance, a tridiagonal
matrix is $1$-banded according to this definition.} 
but more general sparsity 
patterns will be
taken into account as well. 

We can also assume from the outset that the basis functions 
form an orthonormal set. If this is not the case, 
we perform a congruence
transformation to an orthogonal basis and replace the original 
Hamiltonian $H$ with $\tilde{H}=Z^THZ$, where $S^{-1}=ZZ^T$ is 
either the L\"owdin ($Z=S^{-1/2}$, see \cite{lowdin}) or the 
inverse Cholesky ($Z=L^{-T}$, with $S=LL^T$)
factor of the overlap matrix $S$; see, e.g., \cite{challacombe}. 
Here $Z^T$ denotes the transpose of $Z$; for 
the L\"owdin factorization, $Z$ is symmetric ($Z=Z^T$).
Up to truncation, the transformed matrix
$\tilde{H}$ is still a banded (or sparse) matrix, albeit 
not as sparse as $H$. Hence, in our decay results
we can replace $H$ with $\tilde{H}$. 
The entries in $S^{-1}$, and therefore those in $Z$, decay
at a rate which depends on the conditioning of $S$. This,
in turns, will depend on the particular basis set used, on
the total number of basis functions, and on the inter-atomic
distances, with larger separations leading to faster decay.   
This is further
discussed in section \ref{overlap} below.
We note that the case of tight-binding
Hamiltonians is covered by our theory. 
Indeed, the tight-binding method consists in expanding the states
of the physical system (e.g., a crystal) in linear combinations of
atomic orbitals of the composing atoms; such an approximation is
successful if the atomic orbitals have
little overlap, which translates to 
a sparse Hamiltonian.
The same applies
to `real space' finite difference (or finite element) 
approximations \cite{saad}. 

For a given sparse discrete Hamiltonian $H$ in
an orthonormal basis, we consider the problem of  
approximating
the zero-temperature density matrix associated with $H$, that is,
the spectral 
projector $P$ onto the occupied subspace spanned by the eigenvectors
corresponding to the smallest $n_e$  
eigenvalues of $H$:
$$
P=\psi_1\otimes\psi_1+\dots+\psi_{n_e}\otimes\psi_{n_e}
\equiv 
|\psi_1\rangle \langle \psi_1| +\dots + |\psi_{n_e}\rangle \langle \psi_{n_e}|,
$$
where $H\psi_i=\varepsilon_i\psi_i$ for $i=1,\dots,n_e$.
Clearly, $P$ is Hermitian and idempotent: $P=P^* = P^2$.
Consider now the Heaviside (step) function
\begin{displaymath}
h(x)=\left\{
\begin{array}{c}
1\qquad {\rm if}\quad x<\mu\\
\frac{1}{2}\qquad {\rm if}\quad x=\mu \\
0\qquad {\rm if}\quad x>\mu
\end{array}\right.
\end{displaymath}
where the number $\mu$ (sometimes called the {\em Fermi level} or {\em chemical
potential}, see \cite{goedecker1}), is such that 
$\varepsilon_{n_e}<\mu<\varepsilon_{n_e+1}$. 
If the {\em spectral gap} 
$\gamma=\varepsilon_{n_e+1}-\varepsilon_{n_e}$, also known as the
HOMO-LUMO gap,\footnote{HOMO = Highest Occupied Molecular 
Orbital; LUMO = Lowest Unoccupied Molecular Orbital.} is not too small, the step function  
$h$ is well approximated by the Fermi--Dirac function\footnote{Several other 
analytic approximations to the step function are known, some of which
are preferable to the Fermi--Dirac function from the computational
point of view; see, e.g., \cite{baerhg3} for a comparative
study. For theoretical analysis,
however, we find it convenient to work with the
Fermi--Dirac function.}
$f_{FD}(x)=1/(1+{\rm e}^{\beta (x-\mu)})$
for suitable values of $\beta >0$: 
$$P=h(H)\approx f_{FD}(H)= \left [I_n + {\rm exp}(\beta (H - \mu I_n))\right ]^{-1}.$$ 
The smaller $\gamma$,
the larger $\beta$ must be taken in order to have a
good approximation: see Fig.~\ref{betagamma}. The parameter $\beta$ can be 
interpreted as an (artificial) inverse temperature; the
zero-temperature limit is quickly approached as $\beta \to \infty$. 
A major advantage of the Fermi--Dirac
function is that it is analytic; hence, we can replace
$h$ with $f_{FD}$ and apply to it a wealth of results from 
approximation theory
for analytic functions. 

We emphasize that the study of the zero-temperature limit -- that is, 
the ground state of the system --  is of fundamental importance in 
electronic structure theory. In the words of \cite[Chapter 2, pp.~11-12]{martin}: 

\medskip

\begin{quote}
...the lowest energy ground state of the electrons determines 
the structure and low-energy motions of the nuclei. The vast
 array of forms of matter -- from the hardest material known, 
diamond carbon, to the soft lubricant, graphite carbon, 
to the many complex crystals and molecules formed by the elements 
of the periodic table -- are largely manifestations of the 
ground state of the electrons.
\end{quote}

\medskip

The  Fermi--Dirac distribution
is also used when dealing with systems at positive
electronic temperatures ($T>0$) with a small or 
null gap (e.g., metallic systems); 
in this
case $\beta = (k_B T)^{-1}$, where
$k_B $ is Boltzmann's constant. In particular, 
use of the Fermi--Dirac function
allows one to compute thermodynamical properties 
(such as the specific heat) and the $T$-dependence of 
quantities from first principles. In this case, of course,
the matrix $P=f_{FD}(H)$ is no longer an orthogonal
projector, not even approximately.

We mention in passing that it is sometimes advantageous to
impose the normalization condition ${\rm Tr}(P) = 1$ on the
density matrix; indeed, such a condition is standard and part
of the definition of density matrix in the quantum mechanics
literature, beginning with von Neumann \cite{vn,VonN}.
At zero temperature we have ${\rm Tr}(P)={\rm rank}(P) = n_e$,
and $P$ is replaced by
$\frac{1}{n_e}P$. With this normalization $P$ is no
longer idempotent, except when $n_e = 1$. In this paper 
we do not make use of such normalization.

The localization (\lq\lq pseudo-sparsity\rq\rq) of the 
density matrix for insulators has been long known to
physicists and chemists; see the literature review in the
following section. A number of authors have exploited
this property to develop a host of linear scaling algorithms 
for electronic structure computations; see, e.g.,
\cite{baerhg,baerhg2,bowler2,CKP08,challacombe,goedecker1,goedeckercolombo,kohn,LNV,baerhg3,martin,niklasson,ordejon,ordejon2,RRS11,scuseria,wu}. 
In this paper we derive explicitly computable decay bounds which 
can be used, at least in principle, to 
determine {\em a priori} the bandwidth or sparsity pattern of the
truncation of the density matrix corresponding to a prescribed error. 
As we shall see, however, our decay estimates tend to be  
conservative and may be pessimistic in practice. Hence,
we regard our results primarily as a theoretical
contribution, providing a rigorous (yet elementary) mathematical
justification for some important localization phenomena 
observed by physicists.
An important aspect of our work is that
our bounds are universal, in the sense that they 
only depend on the bandwidth (or sparsity pattern) of the
discrete Hamiltonian $H$, on the smallest and largest eigenvalues
of $H$, on the gap $\gamma$ and, when relevant, 
on the temperature $T$. 
In particular, our results are valid for a wide range
of basis sets and indeed for different discretizations
and representations
of the Hamiltonian.

\section{Related work}\label{related}
The localization properties of spectral projectors 
(more generally, density matrices) 
associated with electronic structure computations
in quantum chemistry and solid state
physics have
been the subject of a large number of papers. 
Roughly speaking, the results found in the literature fall into three 
broad categories:

\begin{enumerate}
\item Fully rigorous mathematical results
for model systems (some quite general);
\item \lq\lq Semi-rigorous\rq\rq\ results for specific 
systems; these results are often characterized as \lq\lq exact\rq\rq, or
\lq\lq analytical\rq\rq\ by the authors (usually phsyicists), but would not
be recognized as mathematically rigorous by mathematicians;
\item Non-rigorous results based on a mixture of heuristics,
physical reasoning, and numerics.
\end{enumerate}

Contributions in the first group are typically due to 
researchers working in solid state and mathematical physics. 
These include
the pioneering works of Kohn \cite{kohn59} and des Cloizeaux \cite{desC},
and the more recent papers by Nenciu \cite{Nenciu}, Brouder 
et al.~\cite{brouder}, and a group of papers by Prodan, Kohn, and 
collaborators 
\cite{prodan1,prodan2,PK05}. 

Before summarizing the content of these contributions, we
should mention that nearly all the results found in the literature 
are expressed at the continuous level, that is, in terms of
decay in {\em functions} rather than decay in {\em matrices}. 
The functions
are typically functions of (real) space; results are
often formulated in terms of the {\em density
kernel}, but sometimes in terms of the {\em Wannier functions}.
The latter form an orthonormal basis set associated with a
broad class of Hamiltonians, and are widely used in
solid state physics. 
Since the Wannier functions
span the occupied subspace, localization results for the  
Wannier functions immediately imply similar localization results
for the corresponding spectral projector. 
Note, however, that the spectral projector may be exponentially localized
even when the Wannier functions are not.

At the continuous level, the density matrix 
$\rho:\reals^d \times \reals^d \longrightarrow \complex$ is the
kernel of the density operator $\cal P$ defined by 
$$({\cal P}\psi)({\bf r}) = \int_{\reals^d}\rho({\bf r}, {\bf r}')\psi ({\bf r}')d{\bf r}'\,,$$
regarded as an integral
operator on $L^2(\reals^d)$; here $d=1,2,3$.
The vectors ${\bf r}$ and
${\bf r}'$ represent any two points in $\reals^d$, and
$|{\bf r} - {\bf r}'|$ is their (Euclidean) distance.
The density kernel can be expressed as
$$\rho ({\bf r}, {\bf r}')= \sum_{i=1}^{n_e} \psi_i({\bf r})\psi_i({\bf r}')^*\,,$$
where now $\psi_i$ is the (normalized) eigenfunction of the 
Hamiltonian operator
$\cal H$ corresponding to the $i$th lowest eigenvalue, $i=1,\ldots ,n_e$, 
and the asterisk denotes complex conjugation; see. e.g., \cite{march}.
The density operator $\cal P$ admits the Dunford integral representation
\begin{equation}\label{cauchy}
{\cal P} = \frac{1}{2\pi \iu}\int_\Gamma (zI - {\cal H})^{-1} \, dz\,,
\end{equation}
where $\Gamma$ is a simple
closed contour in $\complex$ surrounding the eigenvalues of 
$\cal H$ corresponding
to the occupied states, with the remaining eigenvalues on the outside. 

In \cite{kohn59}, Kohn proved the rapid decay of
the Wannier functions
for one-dimensional, one-particle Schr\"odinger operators
with periodic and symmetric potentials with non-intersecting 
energy bands. 
This type of Hamiltonian describes one-dimensional, 
centrosymmetric crystals.
Kohn's main result takes the following form:
\begin{equation}\label{as_decay}
\lim_{x\to \infty} w(x)\,{\rm e}^{qx} = 0\,,
\end{equation}
where $w(x)$ denotes a Wannier function (here $x$ is the
distance from the center of symmetry)
and $q$ is a suitable
positive constant. 
In the same paper (page 820) Kohn also points out that for free
electrons (not covered by his theory, which
deals only with insulators) the decay is very
slow, being like $x^{-1}$.

A few observations are in order: first, the
decay result (\ref{as_decay}) is asymptotic, that is,
it implies fast decay at sufficiently large distances $|x|$ only. 
Second, (\ref{as_decay}) is consistent not only with strict 
exponential decay, but
also with decay of the form $x^p{\rm e}^{-q'x}$ where $p$ is arbitrary 
(positive or negative) and $q'>q$. Hence, the actual decay could
be faster, but also slower, than exponential. Since the result
in (\ref{as_decay}) provides only an estimate (rather than an 
upper bound) for the density matrix in real space, it is
not easy to use in actual calculations. To be fair, 
such practical aspects were
not discussed by Kohn until much later (see, e.g., \cite{kohn}). 
Also, later work showed that the asymptotic regime is achieved 
already for distances of the order of 1-2 lattice constants, and
helped clarify the form of the power-law prefactor, as
discussed below. 

The techniques used by Kohn, mostly the theory of analytic functions 
in one complex variable
and some classical asymptotics for linear second-order differential
operators with variable coefficients, did not lend themselves
naturally to the treatment of 
higher dimensionl cases or more complicated potentials. 
The problem of the validity of Kohn's results in two and three dimensions
has remained open for a very long time, and has been long regarded
as one of the last outstanding problems of one-particle condensed-matter
physics. Partial results were obtained by des Cloizeaux \cite{desC}  
and much later by Nenciu \cite{Nenciu}. Des Cloizeaux, who studied
both the decay of the Wannier functions and that of the associated
spectral projectors, extended Kohn's localization results
to 3D insulators
with a center of inversion (a specific symmetry requirement)
in the special case of simple, isolated (i.e., nondegenerate) 
energy bands; 
he also treated the tight-binding limit for arbitrary crystals. Nenciu 
further generalized Kohn's results
to arbitrary $d$-dimensional insulators, again limited
to the case of simple bands.

The next breakthrough came much more recently, when Brouder 
et al.~\cite{brouder} managed to prove localization of the 
Wannier functions for a broad class of insulators in arbitrary
dimensions. The potentials considered by these authors are
sufficiently general for the results to be directly applicable
to DFT, both within the LDA and the GGA frameworks.
 The results in \cite{brouder}, however, also prove that
for {\em Chern insulators} (i.e., 
insulators for which the {\em Chern invariants}, which 
characterize the band structure, are non-vanishing) the Wannier 
functions do {\em not} decay exponentially, therefore leaving open the
question of proving the decay of the density matrix in this case \cite{TV06}. 
It should be mentioned that the mathematics in \cite{brouder}
is fairly sophisticated, and requires some knowledge
of modern differential geometry and topology.

Further papers of interest include the work by Prodan, Kohn, and
collaborators \cite{prodan1,prodan2,PK05}. 
From the mathematical standpoint, the most satisfactory
results are perhaps those presented in \cite{prodan2}. In this
paper, the authors use norm estimates for complex symmetric
operators in Hilbert space to obtain sharp
exponential decay estimates for the 
resolvents of rather general
Hamiltonians with spectral gap. 
Using the contour integral
representation formula (\ref{cauchy}), these estimates yield 
(for sufficiently large separations)
exponential spatial decay bounds of the form
\begin{equation}\label{density_decay_ins}
|\rho ({\bf r}, {\bf r}')| \le C\,{\rm e}^{-\alpha |{\bf r} - {\bf r}'|} \quad
 (C>0, \,\, \alpha > 0, \,\,\,{\rm const.})
\end{equation}
for a broad class of insulators.
A lower bound on the decay rate  $\alpha$ (also known as the {\em decay length}
or {\em inverse correlation length}) is derived,
and the behavior of $\alpha$
as a function of the spectral gap $\gamma$ is examined.

Among the papers in the second group, we mention
\cite{goe,hevan,IBA,JK,maslen,tar1,tar2}. These papers
provide quantitative decay estimates for the density
matrix, either based on fairly rigorous analyses of special
cases, or on not fully rigorous discussions of general
situations. Large use is made of approximations, 
asymptotics, heuristics and physically motivated assumptions, and the
results are often validated by numerical calculations. 
Also, it is occasionally stated that while the results were
derived in the case of simplified models, the conclusions
should be valid in general. Several of these authors emphasize
the difficulty of obtaining rigorous results for general
systems in arbitrary dimension. In spite of not being fully
rigorous from a mathematical point of view, these
contributions are certainly very valuable and seem to have been
broadly accepted by physicists and chemists. We note, however,
that the results in these papers usually take the form of
order-of-magnitude estimates for the density
matrix $\rho ({\bf r}, {\bf r}')$ in real space, valid for sufficiently
large separations $|{\bf r} - {\bf r}'|$, rather than
strict upper bounds. As said before of Kohn's results, this type
of estimates may be difficult to use for computational
purposes.  

In the case of insulators,
the asymptotic decay estimates in these papers take the form
\begin{equation}\label{dec_insul}
 \rho ({\bf r},{\bf r}') = C\, \frac{{\rm e}^{-\alpha |{\bf r}-{\bf r}'|}}
{|{\bf r}-{\bf r}'|^{\sigma}} \,, \quad |{\bf r} - {\bf r}'| \to \infty
\quad (\alpha >0 \,\,, \sigma > 0, \,\,\,{\rm const.})\,,
\end{equation}
where higher order terms have been neglected. 
Many of these papers concern the precise form of the power-law
factor (i.e., the value of $\sigma$) in both insulators 
and metallic systems. 
The actual functional
dependence of $\alpha$ on the gap 
and of $\sigma$ on the dimensionality of the problem
have been the subject of
intense discussion, with some authors claiming that $\alpha$
is proportional to $\gamma$, and others finding it to be
proportional to $\sqrt{\gamma}$; see, e.g.,
\cite{goedecker1,IBA,JK,maslen,tar1,tar2} and
section \ref{dep_gap} below. It appears that
both types of behavior can occur in practice. For instance,
in \cite{JK} the authors provide a tight-binding model of an insulator
for which the density falls off exponentially with decay
length $\alpha= O(\gamma)$
in the diagonal direction of the lattice, and
$\alpha = O(\sqrt{\gamma})$ in non-diagonal directions, as $\gamma \to 0+$.
We also note that in \cite{JK}, the decay behavior of the density
matrix for an insulator is found to be given (up to higher
order terms) by 
$$\rho ({\bf r},{\bf r}') =C\,\frac{{\rm e}^{-\alpha |{\bf r}-{\bf r}'|}}
{|{\bf r}-{\bf r}'|^{d/2}}, \quad |{\bf r}-{\bf r}'|\to \infty \,,$$
where $d$ is the dimensionality of the problem.
In practice, the power-law factor in the denominator
is often ignored, since the exponential decay dominates.

In \cite{goe},
Goedecker argued that the density matrix for 
$d$-dimensional ($d=1,2,3$) metallic
systems at electronic temperature $T>0$ behaves 
to leading order like
\begin{equation}\label{density_decay_met}
\rho ({\bf r}, {\bf r}')=C\,\frac{\cos\,(|{\bf r} - {\bf r}'|)}
{|{\bf r} - {\bf r}'|^{(d+1)/2}} 
\, {\rm e}^{-k_BT|{\bf r} - {\bf r}'|}, 
\quad |{\bf r} - {\bf r}' | \to \infty \,.
\end{equation}
Note that in the zero-temperature limit, a power-law
decay (with oscillations) is observed. An analogous result was
also obtained in \cite{IBA}. Note that
the decay length in the exponential goes to zero like 
the temperature $T$ rather than like $\sqrt{T}$, as
claimed for instance in \cite{baer}. We will return on this
topic in section \ref{dep_T}.

Finally, as representatives of the third group of papers
we select \cite{baer} and \cite{zhang}. The authors of \cite{baer}
use the Fermi--Dirac approximation of the density matrix and
consider its expansion in the Chebyshev basis. From an estimate of
the rate of decay of the coefficients of the Chebyshev expansion
of $f_{FD}(x)$, they obtain estimates for the number of terms
needed to satisfy a prescribed error in the
approximation of the density matrix. In turn, this yields
estimates for the rate of decay as a function of the extreme
eigenvalues and spectral gap of the discrete Hamiltonian. Because
of some {\em ad hoc} assumptions and the many approximations used the 
arguments in this paper cannot be considered  mathematically
rigorous, and the estimates thus obtained are not
always accurate. Nevertheless, the idea of using a polynomial 
approximation for the Fermi--Dirac function and the observation that
exponential decay of the expansion coefficients implies exponential
decay in the (approximate) density matrix is quite valuable and,
as we show in this paper, can be made fully rigorous. 

Finally, in \cite{zhang} the authors present the results of
numerical calculations for various insulators in order to gain
some insight on the dependence of the decay length on the gap.
Their experiments confirm that the decay behavior of 
$\rho({\bf r}, {\bf r}')$
can be strongly anisotropic, and that different rates of decay
may occur in different directions; this is consistent
with the analytical results in \cite{JK}.  


Despite this considerable body of work, the localization
question for density matrices cannot be regarded as
completely settled from the mathematical standpoint. 
We are not aware of any completely general and
rigorous mathematical treatment of the decay properties  in
density matrices associated with general (localized) Hamiltonians,
covering all systems with gap as well as metallic systems at
positive temperature.
Moreover, rather than order-of-magnitude estimates, actual upper bounds
would be more satisfactory.

Also, almost all the above-mentioned results concern the continuous,
infinite-dimensional case. In practice, of course, calculations
are performed on discrete, $n$-dimensional approximations $H$
and $P$ to the operators $\cal H$ and $\cal P$.
The replacement of density operators with finite density matrices can be obtained 
via the introduction of a system of $n$
basis functions $\{\phi_i\}_{i=1}^n$, leading to the density
matrix $P=(P_{ij})$ with 
\begin{equation}\label{density_matrix}
P_{ij}=\langle \phi_j, {\cal P}\phi_i\rangle = \langle \phi_i|{\cal P}|\phi_j\rangle = 
\int_{\reals^d}\int_{\reals^d}
\rho ({\bf r},{\bf r}')\phi_i({\bf r})^*\phi_j({\bf r}')d{\bf r}d{\bf r}'\,. 
\end{equation}
As long as the basis functions are localized in space, the decay
behavior of the density function $\rho ({\bf r},{\bf r}')$
for increasing spatial separation $| {\bf r} - {\bf r}'| $
is reflected in the decay behavior of the matrix elements $P_{ij}$ away from the main
diagonal (i.e., for $|i-j|$ increasing) or, more generally, for increasing distance
$d(i,j)$ in the graph associated with the discrete Hamiltonian;
see section \ref{trunc} for details. 

In developing and analyzing $O(n)$ methods for electronic structure computations, 
it is important to rigorously establish decay bounds for the
entries of the density matrices that take into account properties
of the discrete Hamiltonians. 
It is in principle possible to obtain decay estimates for
finite-dimensional approximations using localized
basis functions from the spatial decay estimates for
the density kernel. Note, however,
that any estimates obtained inserting (\ref{density_decay_ins}) 
or (\ref{density_decay_met})
into (\ref{density_matrix}) would depend on the particular
set of basis functions used.

In this paper we take a different approach. Instead of starting
with the continuous problem and discretizing it, we establish our
estimates directly for sequences of matrices of finite, but increasing
order.
We believe that this approach is closer to the practice of
electronic structure calculations, where matrices are the primary
computational objects. 

 We impose a minimal set of assumptions on our matrix
sequences so as to reproduce the main features of problems encountered
in actual electronic structure computations, while at the 
same time ensuring a high degree of generality. 
Since our aim is to provide a rigorous and general mathematical
justification to the possibility of $O(n)$ methods, this approach
seems to be quite natural.\footnote{We refer the historically-minded reader
to the interesting discussion given by John von Neumann in \cite{vn1} 
on the benefits that can be expected from a study of the 
asymptotic properties of large matrices, in alternative to the
study of the infinite-dimensional (Hilbert space) case.} 

To put our work further into perspective, we quote from
two prominent researchers in the field of electronic structure, one a
mathematician, the other a physicist.
In his excellent survey \cite{LeBris} Claude Le Bris, discussing 
the basis for linear scaling algorithms, i.e., the assumed sparsity
of the density matrix, wrote (pages 402 and 404):

\medskip

\begin{quote}
The latter assumption is in some sense an {\em a posteriori} assumption,
and not easy to analyse [...] It is to be emphasized  that the 
numerical analyis of the linear scaling
methods overviewed above that would account for cut-off rules and locality
assumptions, is not yet available.
\end{quote}

\medskip

It is interesting to compare these statements with two earlier ones
by Stefan Goedecker. In \cite{goe_jcp} he wrote (page 261):

\medskip

\begin{quote}
To obtain a linear scaling, the extended orbitals [i.e., the
eigenfunctions of the one-particle Hamiltonian corresponding to
occupied states]
have to be replaced by the density matrix,
whose physical behavior can be exploited to obtain a fast algorithm. This last
point is essential. Mathematical and numerical analyses alone are not
sufficient to construct a linear algorithm. They have to be combined
with physical intuition.
\end{quote}

\medskip

A similar statement can be found in \cite{goedecker1}, page 1086:

\medskip

\begin{quote}
Even though $O(N)$ algorithms contain many aspects of mathematics and
computer science they have, nevertheless, deep roots in physics. Linear
scaling is not obtainable by purely mathematical tricks, but it is based
on an understanding of the concept of locality in quantum mechanics.
\end{quote}

\medskip

In the following we provide a general treatment of the
question of decay in spectral projectors that is as {\em a priori}
as possible, in the sense that it relies on a minimal set of assumptions
on the discrete Hamiltonians; furthermore, our theory is purely
mathematical, and therefore completely independent of any physical
interpretation. Nevertheless, our theory allows us to
shed light on questions like the dependence of the decay length
on the temperature in the density matrix for metals at $T>0$;
see section \ref{dep_T}. We do this using
for the most part fairly simple mathematical tools  
from classical approximation theory and
linear algebra. 

Of course, in the development of practical linear
scaling algorithms a deep knowledge of the physics involved
is extremely important; we think, however, that locality is as
much a {\em mathematical phenomenon} as a physical one. 

We hope that the increased level of generality attained in this paper
(relative to previous treatments in the physics literature) will also 
help in the
development of $O(n)$ methods for other types of problems where
spectral projectors and related matrix functions play a central
role. A few examples are discussed in section \ref{other}.

\section{Normalizations and scalings}\label{norm}

We will be dealing with sequences of matrices $\{H_n\}$ of increasing size. 
We assume that
each matrix $H_n$ is an Hermitian $n\times n$ matrix, where $n=n_b\cdot n_e$;
here $n_b$ is fixed, while $n_e$ is increasing. As explained in 
section \ref{sec:DM},
the motivation for this assumption is that in most electronic structure codes,
once a basis set has been selected
the number $n_b$ of basis functions per particle is fixed, and one is
interested in the scaling as $n_e$, the number of particles, increases.
Hence, the parameter that controls the system size is $n_e$. 
We also assume that the system
is contained in a $d$-dimensional box of volume $V=L^d$ and that
$L\to \infty$ as $n_e \to \infty$ in such a way that the average 
density $n_e/L^d$ remains constant ({\em thermodynamic limit}).
This is very
different from the case of finite element or finite difference approximations
to partial differential equations (PDEs), where the system (or domain)
size is considered
fixed while the number of basis functions increases or,
equivalently, the mesh size $h$ goes to zero. 

Our scaling assumption has very important consequences on the structural
and spectral properties
of the matrix sequence $\{H_n\}$; namely, the following properties hold:\\


\begin{enumerate}
\item The bandwidth of $H_n$, which reflects
the {\em interaction range} of the discrete Hamiltonians, remains bounded as
the system size increases \cite[page 454]{martin}. More generally, the 
entries of $H_n$ decay away from the main diagonal at a rate
independent of $n_e$ (hence, of $n$).
See section \ref{trunc} for precise definitions and generalizations.
\item The eigenvalue spectra
$\sigma (H_n)$ are also uniformly bounded as $n_e \to \infty$. 
In view of the previous property, this is equivalent to 
saying that the entries in $H_n$ are uniformly
bounded in magnitude: this is just a consequence of Ger\v sgorin's Theorem
(see, e.g., \cite[page 344]{HJ}).
\item For the case of Hamiltonians modeling insulators 
or semiconductors, the spectral (HOMO-LUMO) gap does not
vanish as $n_e\to \infty$. More precisely: if $\varepsilon_i^{(n)}$ denotes the $i$th
eigenvalue of $H_n$, and $\gamma_n:=\varepsilon_{n_e+1}^{(n)} - \varepsilon_{n_e}^{(n)}$,
then $\inf_n \gamma_n > 0$. This assumption does not hold for Hamiltonians
modelling metallic systems; in this case, $\inf_n \gamma_n = 0$, i.e.,
the spectral gap goes to zero as $n_e\to \infty$.
\end{enumerate}

\vspace{0.06in}

We emphasize that these properties hold for very general classes of
physical systems and discretization methods for electronic
structure, with few exceptions (i.e., non-localized basis functions,
such as plane waves).
It is instructive to contrast these properties with those of matrix sequences 
arising in finite element or finite difference approximations of PDEs, where
the matrix size increases as $h\to 0$, with $h$ a discretization parameter.
Considering the case of a scalar,
second-order elliptic PDE, we see that the first property only holds in the
one-dimensional case, or in higher-dimensional cases when the discretization
is refined in only one dimension. (As we will see, this condition
is rather restrictive and can be relaxed.) 
Furthermore,
it is generally impossible to satisfy the second assumption and the one on
the non-vanishing gap ($\inf_n \gamma_n > 0$) 
simultaneously. Indeed, normalizing the matrices so that their spectra remain
uniformly bounded will generally cause the eigenvalues to completely fill
the spectral interval as $n\to \infty$. That is, in general, given any two
points inside this interval, for $n$ large enough 
at least one eigenvalue of the corresponding $n\times n$
matrix falls between these two points. 

Our assumptions allow us to refer to the spectral gap of the matrix
sequence $\{H_n\}$ without having to specify whether we are talking
about an absolute or a relative gap. As we shall see, it is 
convenient to assume that all the matrices
in the sequence $\{H_n\}$ have spectrum contained in the interval
$[-1,1]$; therefore, the absolute gap and the relative gap of any matrix $H_n$
are the same, up to the factor 2. The spectral gap 
(more precisely, its reciprocal) is a natural measure of
the conditioning of the problem of computing the spectral projector onto
the occupied subspace, i.e.,
the subspace spanned by the eigenvectors of $H_n$ corresponding to eigenvalues
$\varepsilon_i^{(n)} < \mu$; see, e.g., \cite[page B4]{ruben} for a recent
discussion. The assumption $\inf_n \gamma_n > 0$
then simply means that the electronic structure
problem is {\em uniformly well-conditioned}; note that this assumption
is also very important for the convergence of the outer SCF
iteration \cite{LeBris,yang}. This hypothesis is satisfied
for insulators and semiconductors, but not in the case of metals.

\section{Approximation of matrices by numerical truncation}\label{trunc}
Discretization of $\cal H$, the Hamiltonian operator, by means of basis sets
consisting of 
linear combinations of Slater or Gaussian-type orbitals 
leads to matrix representations that are,
strictly speaking, full. Indeed, since these basis functions 
are globally supported, almost all matrix elements
$H_{ij} = \langle \phi_j,{\cal H}\phi_i\rangle 
\equiv \langle \phi_i |{\cal H} | \phi_j \rangle$ are non-zero. 
The same is true for the entries of the overlap matrix
$S_{ij} = \langle \phi_j, \phi_i\rangle$. However, owing
to the rapid decay of the basis functions outside of a
localized region, and due to the local nature of the
interactions encoded by the Hamiltonian operator, the entries
of $H$ decay exponentially fast with the
spatial separation of the basis functions. (For the overlap
matrix corresponding to Gaussian-type orbitals,
the decay is actually even faster than exponential.) 

More formally, we say that a sequence of
$n\times n$ matrices ${A_n} =\left (\,[A_n]_{ij}\right )$ has the
{\em exponential
off-diagonal decay property} if there are constants $c > 0$ and
$\alpha >0$ independent of $n$ such that
\begin{equation}
  \left  |[A_n]_{ij}\right | \leq c\, {\rm e}^{-\alpha |i-j|}, \quad {\rm for \,\, all}\quad
    i,j=1,\ldots ,n. \label{truncerr}
\end{equation}
Corresponding to each matrix $A_n$ we then define for
a nonnegative integer $m$ the
matrix $A_n^{(m)}=\left ([A_n^{(m)}]_{ij}\right )$ defined as follows:
$$[A_n^{(m)}]_{ij} = 
\begin{cases}
    \begin{array}{cc}
    [A_n]_{ij} & {\rm if}\quad  |i-j| \leq m;\\
    0 &{\rm otherwise}.
    \end{array}
\end{cases} 
$$
Clearly, each matrix $A_n^{(m)}$ is $m$-banded and can be
thought of as an approximation, or truncation, of $A_n$. Note that
the set of $m$-banded matrices forms a vector subspace
${\cal V}_m \subseteq {\complex}^{n\times n}$ and that $A_n^{(m)}$
is just the orthogonal projection of $A_n$ onto ${\cal V}_m$ with
respect to the Frobenius inner product $\langle A, B\rangle_F
:= {\rm Tr}(B^*A)$. Hence, $A_n^{(m)}$ is the best approximation 
of $A_n$ in ${\cal V}_m$ with respect to the Frobenius norm.

Note that we do not require the
matrices to be Hermitian or symmetric here;
we only assume (for simplicity) that the same pattern of non-zero
off-diagonals is present on either side of
the  main diagonal.
The following simple result from \cite{benzirazouk} provides 
an estimate of the rate at which
the truncation error decreases as the bandwidth $m$ of the
approximation increases. In addition, it establishes $n$-independence
of the truncation error for $n\to \infty$ for matrix sequences
satisfying (\ref{truncerr}).

\vspace{0.06in}

\begin{proposition}\cite{benzirazouk} \label{ter}
Let $A$ be a matrix with entries $A_{ij}$
satisfying (\ref{truncerr}) and let $A^{(m)}$ be the
corresponding $m$-banded approximation.
Then for any $\epsilon > 0$ there is an $\bar m$ such that
$\|A - A^{(m)}\|_1 \leq \epsilon$ for $m\geq \bar m$.
\end{proposition}

\vspace{0.06in}

The integer $\bar m$ in the foregoing proposition is easily
found to be given by
$$\bar m = 
\left \lfloor \frac{1}{\alpha}{\ln }\left (\frac{2c}{1-{\rm e}^{-\alpha}}\epsilon^{-1} \right )
\right \rfloor .$$
Clearly, this result is of interest only for $\bar m < n$ (in fact,
for $\bar m \ll n$).

\begin{example}
Let us consider a tridiagonal matrix $H$ of size $200\times 200$, with 
eigenvalues randomly chosen in $[-1,-0.5]\cup [0.5,1]$, and let $P$ be
the associated density matrix with $\mu=0$. Numerical computation shows that 
$P$ satisfies the bound \eqref{truncerr} with $\alpha=0.6$ and $c=10$ (as long
as its entries are larger than the machine precision). 
Fig.~\ref{mbarexample} depicts the absolute value of the entries in the first row 
of $P$ and the bound \eqref{truncerr}, in a logarithmic scale. Choose, for instance,
a tolerance $\epsilon=10^{-6}$; then it follows from the previous formula that 
the truncated matrix $P^{(m)}$ satisfies $\|P-P^{(m)}\|_1\leq\epsilon$ for any 
bandwidth $m\geq29$.
\end{example}

\begin{figure}[t!]
\begin{center}
\includegraphics[width=0.7\textwidth]{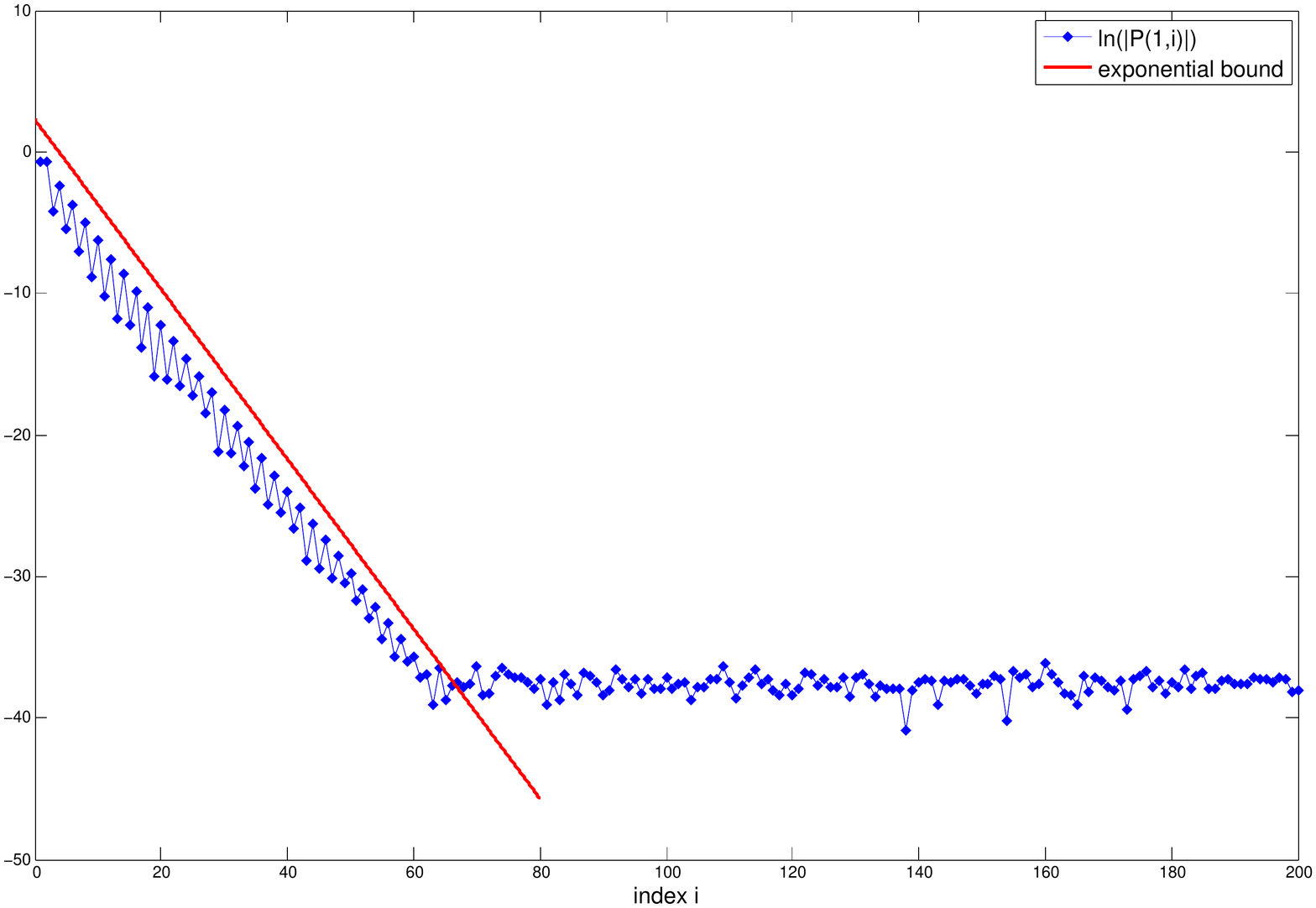}
\vspace{-0.2in}
\caption{Logarithmic plot of the first row of a density matrix and
an exponential bound.}\label{mbarexample}
\end{center}
\end{figure}

What is important about this simple result is
that when applied to a sequence $\{A_n\}=\left ([A_n]_{ij}\right )$ 
of $n\times n$ matrices having the
off-diagonal decay property (\ref{truncerr}) with $c$ and
$\alpha$ independent of $n$, the bandwidth $\bar m$ is itself
independent of $n$.
For convenience, we have stated Proposition \ref{ter} in the
$1$-norm; when $A=A^*$ the same conclusion holds for the
$2$-norm, owing to the
inequality
\begin{equation}\label{norm_ineq}
\|A\|_2 \leq \sqrt{\|A\|_1\|A\|_{\infty}} 
\end{equation}
(see \cite[Corollary 2.3.2]{GVL}). 
Moreover, a similar result also applies to other types
of decay, such as algebraic (power-law) decay of the form
$$\left |[A_n]_{ij}\right | \leq \frac{c}{|i-j|^p + 1}, \quad {\rm for \,\, all}
\quad i,j = 1,\ldots ,n$$
with $c$ and $p$ independent of $n$, as long as $p>1$.

\begin{remark}
It is worth emphasizing that the above considerations 
do not require that the matrix entries $[A_n]_{ij}$ 
themselves actually decay exponentially away from the main
diagonal, but only that they are bounded above in an exponentially 
decaying manner. In particular, the decay behavior of the matrix
entries need not be monotonic.
\end{remark}

Although we have limited ourselves to absolute approximation errors 
in various norms, it is easy to accommodate relative errors  
by normalizing the matrices. Indeed, upon normalization all the
Hamiltonians satisfy $\|H_n\|_2 = 1$; furthermore, for density
matrices this property is automatically satisfied, since they are orthogonal
projectors. In the next section we also consider using the
Frobenius norm for projectors.

The foregoing considerations can be extended to matrices with
more general decay patterns, i.e., with 
exponential decay away from a subset of selected positions
$(i,j)$ in the matrix; see, e.g., \cite{benzirazouk} as well as
\cite{cramer}.
In order to formalize this notion, we first recall the 
definition of {\em geodetic distance} $d(i,j)$ in a graph \cite{diestel}:
it is the number of edges
in the shortest path connecting
two nodes $i$ and $j$, possibly infinite if there is no
such path. 
Next, given a (sparse) matrix sequence $\{A_n\}$
we associate with each
matrix $A_n$ a graph $G_n$ with $n$ nodes and
$m=O(n)$ edges. 
In order to obtain meaningful results, however, we need 
to impose some restrictions on the
types of sparsity allowed. Recall that the {\em degree}
of node $i$ in a graph is just the number of neighbors
of $i$, i.e., the number of nodes at distance 1 from $i$. We
denote by ${\rm deg}_n(i)$ the degree of node $i$ in
the graph $G_n$. We shall assume that the maximum degree
of any node in $G_n$ remains bounded as $n\to \infty$;
that is, there exists a positive integer $D$ independent
of $n$ such that $\max_{1\le i \le n}{\rm deg}_n(i) \le D$ 
for all $n$. 
Note that
when $A_n = H_n$ (discretized Hamiltonian),
this property is a mathematical restatement of the  
physical notion of locality, or finite range, of interactions. 

Now let us assume that we have a sequence of 
$n\times n$ matrices $A_n =\left ([A_n]_{ij}\right )$ 
with associated graphs $G_n$ and graph distances $d_n(i,j)$. 
We will say that ${A_n}$ has the {\em exponential decay
property relative to the graph $G_n$} if there are 
constants $c > 0$ and
$\alpha >0$ independent of $n$ such that 
\begin{equation}
  \left  |[A_n]_{ij}\right | \leq c\, {\rm e}^{-\alpha d_n(i,j)}, 
   \quad {\rm for \,\, all}\quad
    i,j=1,\ldots ,n. \label{truncerr1}
\end{equation}

We have the following simple result.

%

\vspace{0.05in}

\begin{proposition}\label{ter1}
Let $\{A_n\}$ be a sequence of $n\times n$ matrices 
satisfying the exponential decay property (\ref{truncerr1})
relative to a sequence
of graphs $\{G_n\}$ having uniformly bounded maximal degree.
Then, for any given $0 < \epsilon < c$,
each $A_n$ contains at most $O(n)$ entries greater than $\epsilon$
in magnitude.
\end{proposition}
\begin{proof}
For a fixed node $i$, the condition $\left |[A_n]_{ij}\right |
> \epsilon$ together with (\ref{truncerr1}) immediately implies
\begin{equation}\label{pip}
d_n(i,j) < \frac{1}{\alpha}\ln \left (\frac{c}{\epsilon}\right ).
\end{equation}
Since $c$ and $\alpha$ are independent of $n$, inequality
(\ref{pip}) together
with the assumption that the graphs $G_n$ have bounded
maximal degree implies that 
for any row of the matrix (indexed by $i$), there is at most
a constant number of entries that have magnitude greater
than $\epsilon$. Hence, only $O(n)$ entries in $A_n$ can satisfy 
$\left |[A_n]_{ij}\right | > \epsilon$.
\end{proof} 

\begin{remark}
Note that the hypothesis of uniformly bounded maximal degrees
is certainly satisfied if the graphs $G_n$ have uniformly
bounded bandwidths (recall that the bandwidth of a graph is
just the bandwidth of the corresponding adjacency matrix). 
This special case corresponds to the matrix sequence $\{A_n\}$
having the off-diagonal exponential decay property. 
\end{remark}

Under the same assumptions of Proposition \ref{ter1},
we can show that it is possible to approximate each $A_n$ to within an
arbitrarily small error $\epsilon >0$ in norm with a sparse
matrix $A_n^{(m)}$ (i.e., a matrix containing only $O(n)$ non-zero entries).

\vspace{0.05in}

\begin{proposition}\label{ter2}
Assume the hypotheses of Proposition \ref{ter1} are satisfied.
Define the matrix $A_n^{(m)} = \left ([A_n^{(m)}]_{ij}\right)$,
where
$$[A_n^{(m)}]_{ij} = 
\begin{cases}
    \begin{array}{cc}
    [A_n]_{ij} & {\rm if}\quad  d_n(i,j) \leq m;\\
    0 &{\rm otherwise}.
    \end{array}
\end{cases} 
$$
Then for any given $\epsilon >0$,
there exists $\bar m$ independent of $n$
such that $\|A_n-A_n^{(m)}\|_1 < \epsilon$, for
all $m\ge \bar m$. Moreover, if $A=A^*$ then it is also 
$\|A_n-A_n^{(m)}\|_2 < \epsilon$ for all $m\ge \bar m$. 
Furthermore,
each $A_n^{(m)}$ contains only $O(n)$ non-zeros.
\end{proposition}
\begin{proof}
For each $n$ and $m$ and for $1\le j\le n$, 
let 
$$K_{n}^m (j):=\left \{i\,|\, 1\le i\le n \,\, {\rm and} \,\, d_n(i,j) > m\right \}.$$ 
We have 
$$\|A_n - A_n^{(m)}\|_1 =\max_{1\le j\le n} \sum_{i\in K_{n}^m(j)} 
\left |[A_n]_{ij} \right |
\le c \max_{1\le j\le n}  \sum_{i\in K_{n}^m(j)} {\rm e}^{-\alpha d_n(i,j)}.
$$
Letting $\lambda = {\rm e}^{-\alpha}$, we obtain
$$\|A_n - A_n^{(m)}\|_1  \le c \max_{1\le j\le n} \sum_{i\in K_{n}^m(j)} \lambda ^{d_n(i,j)}
\le c \sum_{k=m+1}^n \lambda ^k < c \sum_{k=m+1}^\infty 
\lambda ^k = c\,\frac{\lambda ^{m+1}}{1 - \lambda}.
$$
Since $0<\lambda < 1$, for any given $\epsilon > 0$ we can always find $\bar m$
such that 
$$c\,\frac{\lambda ^{m+1}}{1 - \lambda} \le \epsilon \quad {\rm for \,\, all}\quad m\ge \bar m.$$
If $A_n=A_n^*$, then $\|A_n - A_n^{(m)}\|_2 \le \|A_n - A_n^{(m)}\|_1 < \epsilon$ for all
$m\ge \bar m$. The last assertion follows from the bounded maximal degree assumption.
\end{proof}

Hence, when forming the overlap matrices and discrete 
Hamiltonians, only matrix elements corresponding
to `sufficiently nearby' basis functions (i.e., basis
functions having sufficient overlap)  need to be
computed, the others being negligibly small.
The resulting matrices are therefore sparse, and
indeed banded for 1D problems, with a 
number of non-zeros that grows linearly in the matrix
dimension. 
The actual bandwidth, or sparsity pattern, may depend
on the choice and numbering (ordering) of basis functions and 
(for the discrete Hamiltonians) on the strength
of the interactions, i.e., on the form of the potential 
function $V$ in the Hamiltonian operator. 

It should be kept in mind
that while the number of non-zeros in the 
Hamiltonians discretized using (say) Gaussian-type orbitals is $O(n)$, 
the actual number of non-zeros per
row can be quite high, indeed much higher than when finite
differences or finite elements are used to discretize the
same operators. It is not unusual to have hundreds or even
thousands of non-zeros
per row. On the other hand, the matrices are very often not huge
in size. As already mentioned, the size $n$ of the matrix is
the total number of basis functions, which is a small or moderate
multiple (between 2 and 25, say) of the number $n_e$
of electrons. For example, if $n_b\approx 10$ and $n_e \approx 2000$,
the size of $H$ will be $n\approx 20,000$ and $H$ could easily contain
several millions of non-zeros.
This should be compared
with `real space' discretizations based on finite elements
or high-order finite difference schemes \cite{saad}. The resulting Hamiltonians are 
usually very sparse, 
with a number of non-zero entries per row averaging a few tens 
at most \cite{bekas}.
However, these matrices are of much larger dimension than the matrices
obtained using basis sets consisting of atom-centered orbitals.  
In this case, methodologies based on approximating the density matrix
are currently not feasible, except for 1D problems.
The same remark applies to discretizations based on plane waves, which
tend to produce matrices of an intermediate size between those obtained using localized
basis sets and those resulting from the use of real space discretizations. These
matrices are actually dense and are never formed explicitly. Instead, they
are only used in the form of matrix-vector products, which can be
implemented efficiently by means of FFTs; see, e.g., \cite{saad}. 

The possibility of developing linear scaling methods 
for electronic structure largely depends on the localization
properties of the density matrix $P$. 
It is therefore critical to understand the decay behavior
of the density matrix. Since at zero temperature the density matrix  
is just a particular spectral projector, we consider next some
general properties of such projectors.

\section{General properties of orthogonal projectors}\label{proj}

While our main goal in this paper is to study decay properties in
orthogonal projectors associated with certain sequences of
sparse matrices of increasing size, it is useful
to first establish some {\em a priori} estimates for the entries
of general projectors. Indeed, the intrinsic properties of a projector
like idempotency, positive semidefiniteness, and the relations
between their trace, rank, and Frobenius norm tend to impose  
rather severe constraints on the magnitude of its entries,
particularly for increasing dimension and rank.

We begin by observing that in an orthogonal projector $P$, 
all entries $P_{ij}$ satisfy $|P_{ij}|\leq 1$ and
since $P$ is positive semidefinite, its largest entry
is on the main diagonal.
 Also, the
trace and rank coincide: ${\rm Tr}(P) = {\rm rank}(P)$.
Moreover, $\|P\|_2 = 1$
and $\|P\|_F = \sqrt{{\rm Tr}(P)}$. 

In the context of electronic structure computations, we deal with a sequence of
$n\times n$ orthogonal projectors $\{P_n\}$ of rank $n_e$, where
$n=n_b\cdot n_e$ with $n_e$ increasing and $n_b$ fixed. Hence, 
\begin{equation} \label{dens_mat_seq}
 {\rm Tr}(P_n) = {\rm rank}(P_n) = n_e, \quad {\rm and}\quad
\|P_n\|_F = \sqrt{n_e}. 
\end{equation}
For convenience, we will call
a sequence of orthogonal projectors $\{P_n\}$ satisfying
(\ref{dens_mat_seq}) a {\em density matrix sequence}; the
entries of $P_n$ will be denoted by $[P_n]_{ij}$.
We have the following lemma.

\begin{lemma}\label{lemma1}
Let $\{P_n\}$ be a density matrix sequence. Then
$$\frac{\sum_{i\neq j}\left |[P_n]_{ij}\right |^2}
       {\|P_n\|_F^2} \leq 1 - \frac{1}{n_b}.$$
\end{lemma}
\begin{proof}
Just observe that 
${\rm Tr}(P_n)=\sum_{i=1}^n [P_n]_{ii} = n_e$ together with
$\left |[P_n]_{ii}\right |\leq 1$ for all $i$ imply that the minimum
of the sum $\sum_{i=1}^n \left |[P_n]_{ii}\right |^2 $
is achieved when $[P_n]_{ii} = \frac{n_e}{n} = \frac{1}{n_b}$ for all $i$.
Hence, $\sum_{i=1}^n \left |[P_n]_{ii}\right |^2 \geq \frac{n}{n_b^2}=
\frac{n_e}{n_b}$. Therefore,
\begin{equation}\label{off}
\sum_{i\neq j}\left |[P_n]_{ij}\right |^2 = \|P_n\|_F^2 - 
\sum_{i=1}^n \left |[P_n]_{ii}\right |^2 \leq \left (1-\frac{1}{n_b} \right )n_e
\end{equation}
and the result follows dividing through by $\|P_n\|_F^2 = n_e$.
\end{proof}

\begin{remark}
From the proof one can trivially see that the bound (\ref{off}) is sharp.
In section \ref{metals} we shall see a non-trivial 
example where the bound is attained.
\end{remark}

\begin{theorem}\label{thm}
Let $\{P_n\}$ be a density matrix sequence. Then, for any $\epsilon > 0$,
the number of entries of $P_n$ greater than or equal to 
$\epsilon$ in magnitude grows
at most linearly with $n$.
\end{theorem}
\begin{proof}
Clearly, it suffices to show that the number of off-diagonal entries
$[P_n]_{ij}$ with $|[P_n]_{ij}|\geq \epsilon$ can grow at most
linearly with $n$. Let
$${\cal I} =\left \{ (i,j)\, |\, 1\leq i,j\leq n\,\, {\rm and} \,\, i\ne j\right \}
\quad {\rm and} \quad
{\cal I}_{\epsilon} = \left \{ (i,j)\in {\cal I}\, |\, 
|[P_n]_{ij}|\geq \epsilon \right \}.$$
Then obviously
$$\sum_{i\neq j}\left |[P_n]_{ij}\right |^2 =
\sum_{(i,j)\in {\cal I}_{\epsilon}}\left |[P_n]_{ij}\right |^2
+
\sum_{(i,j)\in {\cal I}\setminus {\cal I}_{\epsilon}}\left |[P_n]_{ij}\right |^2
$$
and if $|{\cal I}_{\epsilon}| = K$, then
$$\sum_{i\neq j}\left |[P_n]_{ij}\right |^2 \geq K\epsilon^2
\quad \Rightarrow  \quad
\frac{\sum_{i\neq j}\left |[P_n]_{ij}\right |^2}
{\|P_n\|_F^2} \geq \frac{K\epsilon^2}{n_e}=\frac{K\epsilon^2n_b}{n}.$$
Hence, by Lemma \ref{lemma1}, 
$$\frac{K\epsilon^2n_b}{n} \leq 
\frac{\sum_{i\neq j}\left |[P_n]_{ij}\right |^2}
{\|P_n\|_F^2} 
\leq 1 - \frac{1}{n_b},$$
from which we obtain the bound
\begin{equation}\label{bound1}
K \leq \frac{n}{\epsilon^2n_b}\left (1 - \frac{1}{n_b}\right ),
\end{equation}
which shows that the number $K$ of entries of $P_n$ with
$|[P_n]_{ij}|\geq \epsilon$ can grow at most as $O(n)$ for
$n\to \infty$.
\end{proof}

\begin{remark}
Due to the presence of the factor $\epsilon^2$ in the denominator
of the bound (\ref{bound1}), for small
$\epsilon$ the proportion of entries of $P_n$ that are
not smaller than
$\epsilon$ can actually be quite large unless $n$
is huge. Nevertheless,  the result is interesting because
it shows that in any density matrix sequence, the
proportion of entries larger than a prescribed
threshold must vanish as $n\to \infty$. 
In practice, for density matrices corresponding to sparse
Hamiltonians with gap, localization occurs already for moderate
values of $n$.
\end{remark}

We already pointed out in the previous section that if the
entries in a matrix sequence $\{A_n\}$ decay at least algebraically with
exponent $p>1$ away from the main diagonal, with rates independent of $n$,
then for any prescribed $\epsilon >0$ it is possible to find a 
sequence of approximants $\left \{A_n^{(m)} \right \}$ 
with a fixed bandwidth $m$ (or sparsity pattern) 
such that $\|A_n - A_n^{(m)}\| < \epsilon$. This applies
in particular to density matrix sequences. The next result
shows that in principle, a {\em linear} rate of decay is enough to 
allow for banded (or sparse) approximation to within any
prescribed {\em relative} error in the Frobenius norm.

\begin{theorem}\label{linear}
Let $\{P_n\}$ be a density matrix sequence and assume that there
exists $c>0$ independent of $n$ such that
$\left |[P_n]_{ij} \right | \leq c/(|i-j| + 1)$ for all $i,j=1,\ldots ,n$.
Then, for all $\epsilon > 0$, there exists a positive integer $\bar m$ 
independent of $n$, such that
$$\frac{\|P_n - P_n^{(m)}\|_F}{\| P_n \|_F} \leq \epsilon \quad {\rm for \,\, all}\,\, 
m\geq \bar m,$$
where $P_n^{(m)}$ is the $m$-banded approximation obtained by
setting to zero all the entries of $P_n$ outside the band.
\end{theorem}   
\begin{proof}
We subtract $P_n^{(m)}$ from $P_n$ and compute $\|P_n - P_n^{(m)}\|_F^2$
by adding the squares of the non-zeros entries in the upper triangular
part of $P_n - P_n^{(m)}$ diagonal by diagonal and multiplying the
result by 2 (since the matrices are Hermitian). Using the
decay assumption we obtain
$$\|P_n - P_n^{(m)}\|_F^2 \leq 2c^2 \sum_{k=1}^{n-m-1}\frac{k}{(n-k+1)^2}
=2c^2 \sum_{k=1}^{n-m-1}\frac{k}{[k-(n+1)]^2}.$$
To obtain an upper bound for the right-hand side, we observe
that the function
$$f(x) = \frac{x}{(x-a)^2}, \quad a = n+1,$$
is strictly increasing and convex on the interval $[1, n-m]$.
Hence, the sum can be bounded above by the integral of $f(x)$ taken
over the same interval:
$$\sum_{k=1}^{n-m-1}\frac{k}{(n-k+1)^2} <
\int_{1}^{n-m} \frac{x}{(x-a)^2}\, dx, \quad a = n + 1.$$
Evaluating the integral and substituting $a=n+1$ in the result we obtain
$$\|P_n - P_n^{(m)}\|_F^2 < 
2c^2 \left [ \ln \left (\frac{m+1}{n}\right ) + 
(n+1)\left (\frac{1}{m+1} - \frac{1}{n}\right ) \right ].$$
Dividing by $\|P_n\|_F^2=n_e$ we find 
$$\frac{\|P_n - P_n^{(m)}\|_F^2}{\| P_n \|_F^2} < 
\frac{2c^2}{n_e} \left [ \ln \left (\frac{m+1}{n}\right ) + (n+1)\left 
(\frac{1}{m+1} - \frac{1}{n}\right )\right ] < \frac{2c^2}{n_e}\frac{n+1}{m+1}.$$
Recalling that $n=n_b\cdot n_e$, we can rewrite the last inequality as
$$\frac{\|P_n - P_n^{(m)}\|_F^2}{\| P_n \|_F^2} < \frac{2c^2}{m+1}\frac{n+1}{n_e}
= \frac{2c^2}{m+1}\left (n_b + \frac{1}{n_e}\right ) \leq \frac{2c^2}{m+1}(n_b+1),$$
a quantity which can be made arbitrarily small by taking $m$ sufficiently large.
\end{proof} 

\begin{remark}
In practice, linear decay (or even algebraic decay with a small exponent $p\ge 1$)
is too slow to be useful in the development of practical $O(n)$ algorithms.
For example, from the above estimates we obtain $\bar m = O(\epsilon ^{-2})$
which is clearly not a very encouraging result, even allowing for the fact that
the above bound may be pessimistic in general.
To date, practical linear scaling algorithms have been developed only for
density matrix sequences exhibiting exponential off-diagonal decay.
\end{remark}

In the case of exponential decay, 
one can prove the following result.

\begin{theorem}\label{expo}
Let $\{P_n\}$ be a density matrix sequence with
$\left |[P_n]_{ij}\right |\leq c\, {\rm e}^{-\alpha |i-j|}$,
where $c>0$ and $\alpha >0$ are independent of $n$.
Let $\left \{P_n^{(m)}\right \}$ be the corresponding
sequence of $m$-banded approximations.
Then there exists $k_0 >0$ independent of $n$ and $m$
such that
$$\frac{\|P_n - P_n^{(m)}\|_F^2}{\| P_n \|_F^2} \leq k_0\, 
{\rm e}^{-2\alpha m}.$$
\end{theorem}
\begin{proof}
Similar to that of Theorem \ref{linear}, except that it is now easy
to evaluate the upper bound and the constants exactly. We omit the details.
\end{proof}

\begin{remark}
It is immediate to see that the foregoing bound implies the
much more favorable estimate $\bar m = O(\ln \epsilon^{-1})$.
\end{remark}

Again, similar results holds for arbitrary sparsity patterns,
replacing $|i-j|$ with the graph distance. More precisely,
the following result holds.

\begin{theorem}
Let $\{P_n\}$ be a density matrix sequence with the exponential
decay property with respect to a sequence of graphs $\{G_n\}$
having uniformly bounded maximal degree. Then, for all $\epsilon  > 0$, 
there exists a positive integer $\bar m$ independent of $n$ 
such that
$$\frac{\|P_n - P_n^{(m)}\|_F}{\|P_n\|_F} \leq \epsilon
\quad {\rm for \,\, all} \quad m\ge \bar m,$$
where $P_n^{(m)}$ is sparse, i.e., it contains only $O(n)$ non-zeros. 
\end{theorem}

We consider now some of the consequences
of approximating full, but localized matrices with sparse ones.
The following quantity plays an important role in many electronic
structure codes:
$$ \langle E \rangle = {\rm Tr} (PH) = 
\varepsilon_1 + \varepsilon_2 + \cdots + \varepsilon_{n_e} ,$$
where $\varepsilon_i$ denotes the $i$th eigenvalue of the
discrete Hamiltonian $H$. Minimization of ${\rm Tr} (PH)$,
subject to the constraints $P=P^*=P^2$ and ${\rm Tr} (P) =n_e$,
is the basis of
several linear scaling algorithms; see, e.g.,
\cite{challacombe,goedecker1,LeBris,LNV,MS,NS00,niklasson}. 
Note that in the tight-binding model, and also within the
independent electron approximation, the quantity $\langle E \rangle$
represents the single-particle energy 
\cite{bates,goedecker1,niklasson,tar2}.
Now, assume that 
$\hat H \approx H$ and $\hat P\approx P$,
and define the corresponding approximation of $ \langle E \rangle$ as
$ \langle \hat E \rangle = {\rm Tr} (\hat P \hat H)$.
(We note in passing that in order to compute $\langle \hat E\rangle 
= {\rm Tr} (\hat P\hat H)$, only the entries of $\hat P$ corresponding to
non-zero entries in $\hat H$ need to be computed.)
Let $\Delta_P = \hat P - P$ and $\Delta_H = \hat H - H$.
We have
$$\langle \hat E \rangle  = {\rm Tr}[(P+\Delta_P)(H+\Delta_H)]
={\rm Tr}(PH) + {\rm Tr}(P \Delta_H ) + {\rm Tr}(\Delta_P H) + 
{\rm Tr}(\Delta_P \Delta_H).$$
Neglecting the last term, we obtain for $\delta_E = 
| \langle E \rangle -  \langle \hat E \rangle |$ the bound
$$\delta_E \le |{\rm Tr}(P \Delta_H)| + |{\rm Tr}(\Delta_P H)|.$$
Recalling that the Frobenius norm is
the matrix norm induced by the inner product $\langle A, B\rangle = {\rm Tr} (B^*A)$,
using the Cauchy--Schwarz inequality and $\|P\|_F = \sqrt{{n_e}}$ we find
$$ \delta_E \le \sqrt{{n_e}}\, \|\Delta_H\|_F + \|\Delta_P\|_F \|H\|_F.$$ 
Now, since the orthogonal projector $P$ is invariant with respect to 
scalings of the Hamiltonian, we can assume $\|H\|_F=1$, so that
$\delta_E \le \sqrt{{n_e}}\, \|\Delta_H\|_F +  \|\Delta_P\|_F$ holds.
In practice, a bound on the {\em relative} error would be more
meaningful. Unfortunately, it is not easy to obtain a rigorous bound in terms
of the relative error in the approximate projector $\hat P$. If, however, we
replace the relative error in $\langle \hat E \rangle$ with the normalized error
obtained by dividing the absolute error by the number $n_e$ of electrons, we obtain
$$\frac{\delta_E}{n_e} \le
\frac{\|\Delta_H\|_F}{\sqrt{{n_e}}}  +  \frac{\|\Delta_P\|_F} {n_e}.$$
A similar bound for $\delta_E/n_e$ that involves matrix 2-norms 
can be obtained as follows. Recall that 
$n = n_b\cdot n_e$, and that $\|A\|_F \le \sqrt{n}\|A\|_2$ for 
any $n\times n$ matrix $A$. Observing that the von Neumann trace
inequality \cite[pages 182--183]{HJTMA} implies
$|\textnormal{Tr}(P\Delta_H)|\leq\textnormal{Tr}
( P )\|\Delta_H\|_2=n_e\|\Delta_H\|_2$, we obtain 
\begin{equation}\label{err_est_F}
\frac{\delta_E}{n_e} \le  
\|\Delta_H\|_2 + \sqrt{\frac{n_b}{n_e}}\,\, \|\Delta_P\|_2.
\end{equation}
Since $n_b$ is constant, an interesting consequence of (\ref{err_est_F})
is that for large system sizes (i.e., in the limit as $n_e\to \infty$),
the normalized error in $\langle \hat E \rangle$ is essentially
determined by the truncation error in the Hamiltonian $H$ rather than
by the error in the density matrix $P$. 

On the other hand, scaling $H$ so that $\|H\|_F=1$ may not be
advisable in practice. Indeed, since the Frobenius norm of the Hamiltonian grows
unboundedly for $n_e \to \infty$, rescaling $H$ so that $\|H\|_F=1$
would lead to a loss of significant information when truncation is applied 
in the case of large systems. A more sensible scaling, which is often
used in algorithms for electronic structure computations, is to divide
$\|H\|$ by its largest eigenvalue in magnitude, so that $\|H\|_2 = 1$. This
is consistent with the assumption, usually satisfied in practice, that
the spectra of the Hamiltonians remain bounded as $n_e \to \infty$. (Note
this is the same normalization used to establish the decay bounds in 
section \ref{sec_main}.) With
this scaling we readily obtain, to first order, the bound
\begin{equation}\label{err_est_2}
\frac{\delta_E}{n_e} \le
 \|\Delta_H\|_2 + n_b\|\Delta_P\|_2 ,
\end{equation}
showing that errors in $\Delta_H$ and $\Delta_P$ enter the estimate for the 
normalized error in the objective function ${\rm Tr}(PH)$ with approximately the
same weight, since $n_b$ is a moderate constant. We also note that since
both error matrices $\Delta_H$ and $\Delta_P$ are Hermitian, 
(\ref{norm_ineq}) implies that the bounds 
(\ref{err_est_F}) and (\ref{err_est_2}) remain true if the $2$-norm is 
replaced by the $1$-norm.
We mention that the problem of the choice of norm in the
measurement of truncation errors has been discussed in \cite{RRS08,RS05}. These
authors emphasize the use of the $2$-norm, which is related to
the distance between the exact and inexact (perturbed) occupied subspaces 
${\cal X} := {\rm Range}(P)$
and $\hat{\cal X}:={\rm Range}(\hat P)$ as measured by the sine of the principal 
angle between $\cal X$ and $\hat{\cal X}$; see \cite{RRS08}.

One important practical aspect, which we do not address here, is that 
in many quantum chemistry codes the matrices have a natural block structure
(where each block corresponds, for instance, to the basis functions centered
at a given atom); hence,
dropping is usually applied to submatrices rather than to 
individual entries. Exploitation of the block structure is also
desirable in order to achieve high performance in matrix-matrix products
and other operations,
see, e.g., \cite{challacombe,challacombe2,RRS09}.

We conclude this section with a few remarks on the infinite-dimensional
case. Recall that any separable, complex Hilbert space 
$\mathscr{H}$ is isometrically 
isomorphic
to the sequence space
$${\ell }^2:= \Big \{(\xi_n) \, | \, \xi_n \in \complex \,\, 
\forall n\in \naturals \,\,\, {\rm and} \,\,\, \sum_{n=1}^\infty 
|\xi_n|^2 < \infty \Big \} \,.$$
Moreover, if $\{e_n\}$ is an orthonormal basis in $\mathscr{H}$,
to any bounded linear operator $\cal A$ on $\mathscr{H}$  
there corresponds the infinite
matrix $A=(A_{ij})$ acting on ${\ell }^2$, uniquely defined by
$A_{ij} = \langle e_j, {\cal A}e_i\rangle$. Note that
each column of $A$ must be in ${\ell }^2$, hence the entries $A_{ij}$
in each column of $A$ must go to zero for $i\to \infty$. The
same is true for the entries in each row (for $j\to \infty)$
since $A^*=(A_{ji}^*)$, the adjoint of $A$, is also a (bounded)
operator defined everywhere on ${\ell }^2$. More precisely,
for any bounded linear operator $A=(A_{ij})$ on ${\ell }^2$ 
the following bounds hold:
\begin{equation}\label{cook}
\sum_{j=1}^\infty |A_{ij}|^2 \leq \|A\|_2^2  \quad {\rm for \,\, all}
\,\,\, i \quad {\rm and}
\quad \sum_{i=1}^\infty |A_{ij}|^2 \leq \|A\|_2^2  \quad {\rm for \,\, all}
\,\,\, j\,,
\end{equation}
since $\|A\|_2 = \|A^*\|_2$.

An orthogonal projector $\cal P$ on $\mathscr{H}$ is a self-adjoint 
(${\cal P} = {\cal P}^*$), idempotent (${\cal P} = {\cal P}^2$)
linear operator. Such an operator is necessarily bounded, with
norm $\|{\cal P}\| = 1$. Hence, (\ref{cook}) implies 
\begin{equation}\label{p_cook}
\sum_{j=1}^\infty |P_{ij}|^2 \leq 1,
\end{equation}  
where $P=(P_{ij})$ denotes the matrix
representation of $\cal P$. The idempotency condition implies
$$ P_{ij} = \sum_{k=1}^{\infty} P_{ik}P_{kj}, \quad {\rm for\,\, all}
\quad i,j = 1,2,\ldots$$
In particular, for $i=j$ we get, using the hermiticity property
$P_{ij} = P_{ji}^*$:  
\begin{equation}\label{p_ii}
P_{ii} = \sum_{k=1}^{\infty} P_{ik}P_{ki} = \sum_{k=1}^{\infty}|P_{ik}|^2,
\quad {\rm for\,\, all} \quad i= 1,2,\ldots
\end{equation}
Now, since $P$ is a projector its entries satisfy
$|P_{ij}|\le 1$, therefore (\ref{p_ii}) 
is a strengthening of inequality (\ref{p_cook}).
 Note in particular that the off-diagonal
entries in the first row (or column) of $P$ must satisfy
$$\sum_{j>1}|P_{1j}|^2 \le 1 - |P_{11}|^2\,,$$
those in the second row (or column) must satisfy
$$\sum_{j>2}|P_{2j}|^2 \le 1 - |P_{22}|^2 - |P_{12}|^2\,,$$
and in general the entries $P_{ij}$ with $j>i$ must satisfy
\begin{equation}\label{proj_j}
\sum_{j>i}|P_{ij}|^2 \le 1 - \sum_{k=1}^{i}|P_{ki}|^2\quad {\rm for\,\, all}\,\,\,
i=1,2,\ldots
\end{equation}
Hence, decay in the off-diagonal entries in the $i$th row of $P$
must be fast enough for the bounds (\ref{proj_j}) to hold. In general,
however, it is not easy to quantify the asymptotic rate of decay to zero
of the off-diagonal entries in an arbitrary orthogonal projector
on ${\ell }^2$. In general, the rate of decay can be rather slow.
 In section 
\ref{metals} we will see an example of spectral projector associated with 
a very simple tridiagonal Hamiltonian for which the off-diagonal entries
decay linearly to zero.

\section{Decay results}\label{sec_main}
In this section
we present and discuss some results on the decay 
of entries for the Fermi--Dirac function applied to
Hamiltonians and for the density matrix (spectral
projector corresponding to occupied states). 
We consider both the banded case and the case of
more general sparsity patterns.
The proofs, which 
require some basic tools
from polynomial approximation theory, will be given in 
subsection \ref{proofs}.

\subsection{Bounds for the Fermi--Dirac function}\label{sec_fd}
We begin with the following result for the banded case.
As usual in this paper, in the following one should think of the positive
integer $n$ as being of the form $n=n_b\cdot n_e$ with $n_b$
constant and $n_e\to \infty$. 

\vspace{0.1in}

\begin{theorem}\label{thm1}
Let $m$ be a fixed positive integer and consider a sequence of 
matrices $\{H_n\}$ such that:
\begin{itemize}
\item[{\rm (i)}] $H_n$ is an $n\times n$ Hermitian, $m$-banded matrix 
for all $n$;
\item[{\rm (ii)}] For every $n$, all the eigenvalues of $H_n$ lie
in the interval $[-1,1]$.
\end{itemize}
For a  given Fermi level $\mu$ and inverse temperature $\beta$, 
define for each $n$ the $n\times n$ Hermitian
matrix $F_n:= f_{FD}(H_n)=\left[I_n+{\rm e}^{\beta(H_n-\mu I_n)}\right]^{-1}$.
Then there exist constants 
$c>0$ and $\alpha>0$,
independent of $n$, such that the following decay bound holds:
\begin{equation}
|[F_n]_{ij}|\leq c\, {\rm e}^{-\alpha |i-j|},\qquad i\neq j.\label{bound}
\end{equation}
The constants $c$ and $\alpha$ can be chosen as
\begin{eqnarray}
&&c=\frac{2\chi M(\chi)}{\chi-1},\qquad M(\chi)=
\max_{z\in\mathcal{E}_\chi}|f_{FD}(z)|, \label{const_c}\\
&&\alpha=\frac{1}{m}\ln\chi,\label{const_alpha}
\end{eqnarray}
for any $1<\chi<\overline{\chi}$, where
\begin{eqnarray}
\overline{\chi}=
\frac{\sqrt{\sqrt{(\beta^2(1-\mu^2)-\pi^2)^2+
4\pi^2\beta^2}-\beta^2(1-\mu^2)+\pi^2}}{\sqrt{2}\beta}+ \nonumber\\
+\frac{\sqrt{\sqrt{(\beta^2(1-\mu^2)-\pi^2)^2+
4\pi^2\beta^2}+\beta^2(1+\mu^2)+\pi^2}}
{\sqrt{2}\beta}\,, \label{horribleformula}
\end{eqnarray}
and ${\cal E}_{\chi}$ is the unique ellipse with foci in
$-1$ and $1$ with semi-axes $\kappa_1>1$ and $\kappa_2>0$, 
and $\chi= \kappa_1+\kappa_2$.
\end{theorem}

\vspace{0.1in}

\begin{remark}
The ellipse ${\cal E}_{\chi}$ in the previous theorem is unique
because the identity $\sqrt{\kappa_1^2 - \kappa_2^2} = 1$, valid
for any ellipse with foci in $1$ and $-1$, 
implies $\kappa_1 - \kappa_2 = 1/(\kappa_1+\kappa_2)$,
hence the parameter $\chi = \kappa_1+\kappa_2$ alone completely
characterizes the ellipse.
\end{remark}

\vspace{0.1in}

\begin{remark}
Theorem \ref{thm1} can be immediately generalized 
to the case where the spectra of the sequence $\{H_n\}$ are 
contained in an interval $[a,b]$,
for any $a<b\in\mathbb{R}$. It suffices to shift 
and scale each Hamiltonian:
$$
\widehat{H}_n=\frac{2}{b-a}H_n-\frac{a+b}{b-a}\, I_n,
$$
so that $\widehat{H}_n$ has spectrum in $[-1,1]$.
For the decay bounds to be independent of $n$,
however, $a$ and $b$ must be independent of $n$.
\end{remark}

\vspace{0.1in}

\begin{figure}[t!]
\begin{center}
\vspace{-1.1in}
\begin{center}
\includegraphics[width=1.0\textwidth]{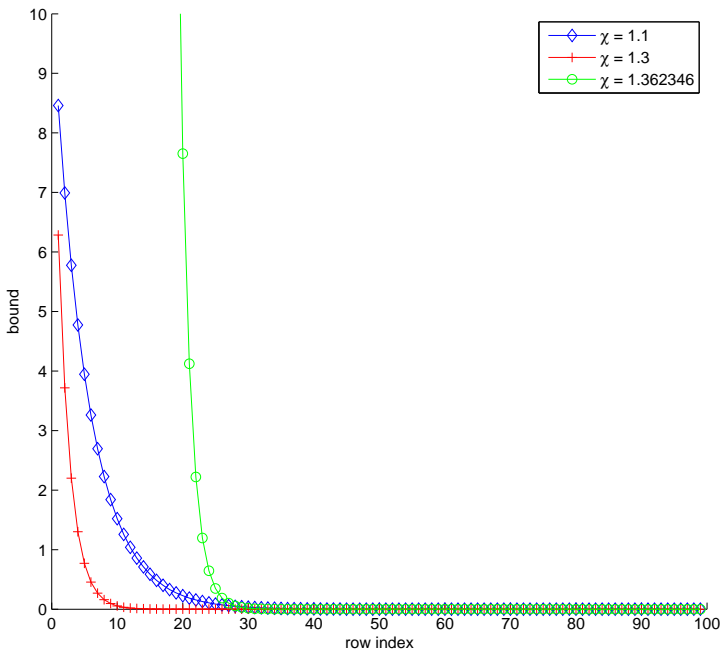}
\end{center}
\vspace{-1.9in}
\caption{Bounds \eqref{bound} with $\mu=0$ and $\beta=10$, 
for three different values of $\chi$.}\label{fig1} 
\end{center}
\end{figure}

\begin{figure}
\begin{center}
\vspace{-1.7in}
\includegraphics[width=1.0\textwidth]{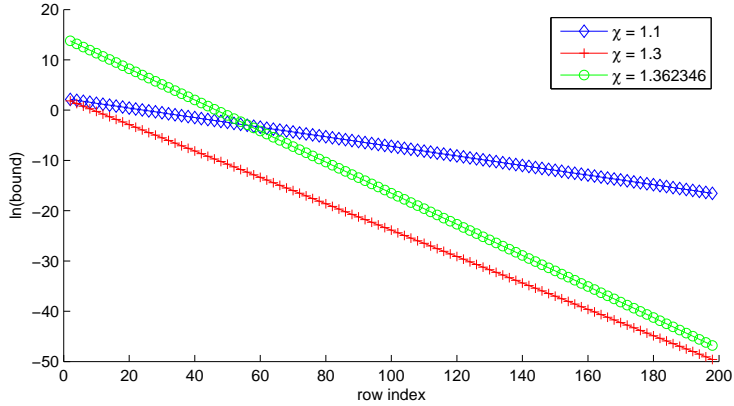}
\vspace{-2.6in}
\caption{Logarithmic plot of the bounds \eqref{bound} with 
$\mu=0$ and $\beta=10$.}\label{fig2}
\end{center}
\end{figure}

It is important to note that there is a certain 
amount of arbitrariness in the choice of $\chi$, 
and therefore of $c$ and $\alpha$.
If one is mainly interested in a fast asymptotic 
decay behavior (i.e., for sufficiently large $|i-j|$), 
it is desirable to choose $\chi$ as 
large as possible. On the other hand,
if $\chi$ is very close to $\overline{\chi}$ then 
the constant $c$ is likely to be quite large and the 
bounds might be too pessimistic.
Let us look at an example. Take $\mu=0$; in this case we 
have 
$$
\overline{\chi}=\left (\pi+\sqrt{\beta^2+\pi^2}\right )/\beta \quad {\rm and}
\quad M(\chi)=\left |1/\left (1+{\rm e}^{\beta \zeta}\right )\right |\,, \quad 
{\rm where} \quad  
\zeta=\iu\,\frac{\chi^2-1}{2\chi}\,.
$$
Note that, in agreement with experience, decay is faster for smaller $\beta$
(i.e., higher electronic temperatures); see sections \ref{proofs}  
and \ref{dep_T} for additional details and discussion.
Figures \ref{fig1} and 
\ref{fig2} show the behavior of the bound
given by \eqref{bound} on the first row of a $200\times 200$ 
tridiagonal matrix ($m=1$) 
for $\beta=10$ and for three values of $\chi$. 
It is easy to see from the plots
that the asymptotic behavior of the bounds improves as 
$\chi$ increases; however, the bound given by $\chi=1.362346$ 
is less useful 
than the bound given by $\chi = 1.3$. Figure \ref{fig3} 
is a plot of $c$ as a function of $\chi$ and it shows that 
$c$ grows very large
 when $\chi$ is close to $\overline{\chi}$. This is expected, 
since $f_{FD}(z)$ has two poles, given by $z=\pm \iu\,\pi/\beta$
on the regularity
ellipse ${\cal E}_{\overline{\chi}}$.
 It is clear from Figures \ref{fig1} and \ref{fig2} that 
$\chi=1.3$ is the best choice among the three proposed values, if
one is interested in determining a bandwidth outside of which
the entries of $F_n$ can be safely neglected.  
As already observed in \cite{benzirazouk,nadersthesis}, 
improved bounds can be obtained
by adaptively choosing
different (typically increasing) values of $\chi$ as $|i-j|$ 
grows, and by using as a bound 
the (lower) envelope of the
curves plotted in Figure \ref{fig5}, which 
shows the behavior of the decay bounds for several values 
of $\chi\in(1.1,\overline{\chi})$, with
$\overline{\chi}\approx 1.3623463$.

\begin{figure}[t!]
\begin{center}
\vspace{-1.75in}
\includegraphics[width=1.0\textwidth]{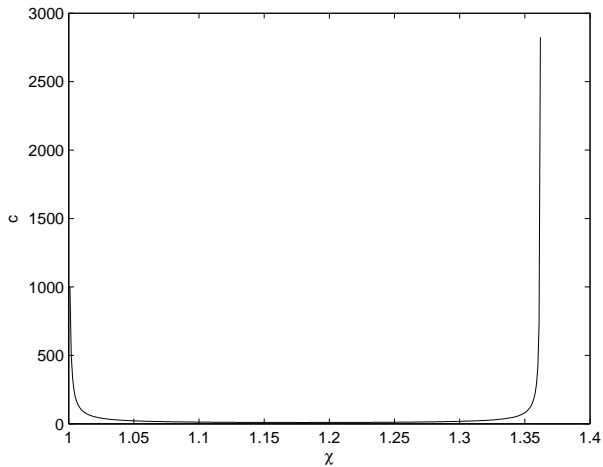}
\vspace{-2.6in}
\caption{Plot of $c$ as a function of $\chi$ with $\mu=0$ and 
$\beta=10$.}\label{fig3}
\end{center}
\end{figure}


\begin{figure}
\vspace{-2.05in}
\begin{center}
\includegraphics[width=1.\textwidth]{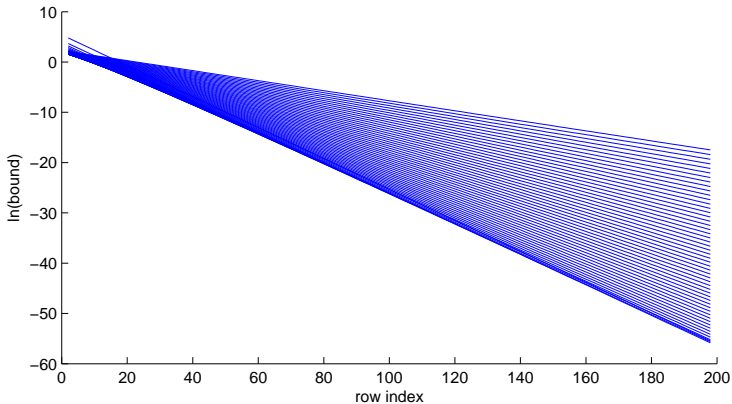}
\vspace{-2.65in}
\caption{Logarithmic plot of the bounds \eqref{bound} with 
$\mu=0$ and $\beta=10$, for several values of $\chi$.}\label{fig5}
\end{center}
\end{figure}

The results of Theorem \ref{thm1} can be generalized to 
the case of Hamiltonians with rather general sparsity patterns; 
see \cite{benzirazouk,cramer,nadersthesis}.
To this end, we make use of the notion of geodetic 
distance in a graph already used in section \ref{trunc}.
The following result holds.


\begin{theorem}\label{thm2}
Consider a sequence of matrices $\{H_n\}$ such that:
\begin{itemize}
\item[{\rm (i)}] $H_n$ is an 
$n\times n$ Hermitian matrix for all $n$;
\item[{\rm (ii)}] the spectra $\sigma(H_n)$ are uniformly bounded 
and contained in $[-1,1]$ for all $n$.
\end{itemize}
Let $d_n(i,j)$ be the graph distance associated with $H_n$. Then
 the following decay bound holds:
\begin{equation}
|[F_n]_{ij}|\leq c\, {\rm e}^{-\theta d_n(i,j)},\qquad 
i\neq j,\label{generalbound}
\end{equation}
where $\theta=\ln \chi$ and the remaining notation and 
choice of constants are as in Theorem \ref{thm1}.
\end{theorem}

We remark that in order for the bound (\ref{generalbound}) to be meaningful
from the point of view of linear scaling, we need to impose some restrictions
on the asymptotic sparsity of the graph sequence $\{G_n\}$. As discussed
in section \ref{trunc}, $O(n)$ approximations of $F_n$ are possible
if the graphs $G_n$ have maximum degree uniformly bounded with respect to $n$.
This guarantees that 
the distance $d_n(i,j)$ grows
unboundedly as $|i-j|$ does, at a rate independent of $n$ for $n\to \infty$.

\subsection{Density matrix decay for systems with gap}\label{sec_density}
The previous results establish exponential decay bounds for the Fermi--Dirac
function of general localized Hamiltonians
and thus for density matrices of arbitrary systems
at positive electronic temperature. In this subsection we consider 
the case of gapped systems (like insulators) at zero temperature.
In this case, as we know, the density matrix is the spectral
projector onto the occupied subspace. As an example, we consider
the density matrix corresponding to 
the linear alkane n-Dopentacontane ${\rm C}_{52}{\rm H}_{106}$
composed of 52 Carbon and 106 Hydrogen atoms,
discretized in a Gaussian-type orbital basis. 
The number of
occupied states is 209, or half the total number of
electrons in the system.\footnote{Here
spin is being taken into account, so that the density
kernel is given by
$\rho({\bf r},{\bf r}') = 2\sum_{i=1}^{n_e/2}\psi_i({\bf r})
\psi_i({\bf r}')^*$; see, e.g., \cite[page 10]{march}.}
The corresponding Hamiltonian in the original non-orthogonal basis
is displayed in Fig.~\ref{fig_H_jacek} (top) and the
`orthogonalized' Hamiltonian $\tilde H$
is shown in Fig.~\ref{fig_H_jacek} (bottom).
Fig.~\ref{fig_projector_jacek} displays the zero temperature
density matrix, which is seen to decay exponentially away
from the main diagonal. Comparing 
Fig.~\ref{fig_projector_jacek} and Fig.~\ref{fig_H_jacek},
we can see that for a truncation level of $10^{-8}$, the bandwidth
of the density matrix is only slightly larger than that of the
Hamiltonian. The eigenvalue spectrum of the Hamiltonian, scaled
and shifted so that its spectrum is contained in the interval
$[-1,1]$, is shown
in Fig.~\ref{spectrum_jacek}.  One can clearly see a large gap
($\approx 1.4$) between the 52 low-lying eigenvalues corresponding
to the core electrons in the system, as well as the smaller HOMO-LUMO
gap ($\approx 0.1$) separating the 209 occupied states from the virtual
(unoccupied) ones.  It is worth emphasizing that the exponential decay of
the density matrix is independent of the size of the system; that is,
if the alkane chain was made arbitrarily long by adding C and H
atoms to it, the density matrix would be of course much larger in size
but the bandwidth would remain virtually unchanged for the same
truncation level, due to the fact that the bandwidth and
the HOMO-LUMO gap of the Hamiltonian do not appreciably
change as the number of particles increases.
It is precisely this independence of the rate of decay 
(hence, of the bandwidth)
on system size that makes $O(n)$ approximations possible 
(and competitive) for large $n$.

Let us now see how Theorem \ref{thm1} can be used to prove 
decay bounds on the entries of density matrices. Let $H$ 
be the discrete Hamiltonian
associated with a certain physical system and let 
$\mu$ be the Fermi level of interest for this system. We assume that the
spectrum of $H$ has a gap $\gamma$ around $\mu$, that is, 
we have $\gamma=\varepsilon^+-{\varepsilon}^->0$, 
where ${\varepsilon}^+$ is the
smallest eigenvalue of $H$ to the right of $\mu$ and $
{\varepsilon}^-$ is the largest eigenvalue of $H$ to 
the left of $\mu$. In the particular case of the HOMO-LUMO gap,
we have ${\varepsilon}^-={\varepsilon}_{n_e}$ and
${\varepsilon}^+={\varepsilon}_{n_e+1}$.

\begin{figure}[t!]
\vspace{-1.1in}
\begin{center}
\includegraphics[width=0.65\textwidth]{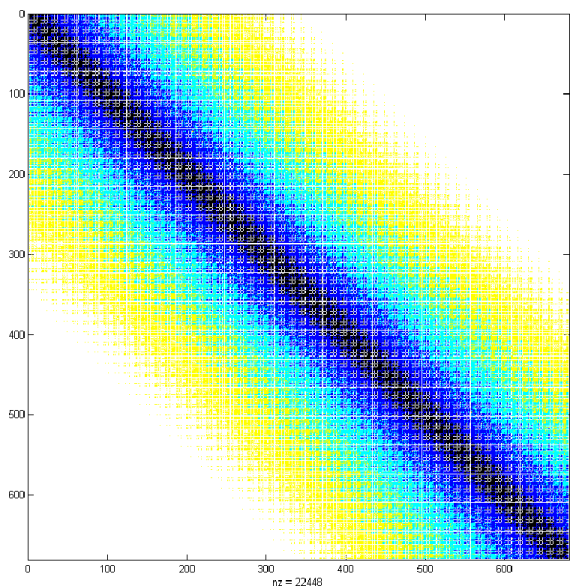}
\end{center}
\vspace{-0.1in}
\caption{Magnitude of the entries in the density matrix
for the linear alkane $C_{52}H_{106}$ chain, with 209 occupied states. 
 White: $<10^{-8}$; yellow:
$10^{-8}-10^{-6}$; green: $10^{-6}-10^{-4}$;
blue: $10^{-4}-10^{-2}$; black: $>10^{-2}$.
 Note: {\tt nz} refers to the number of
`black' entries.}\label{fig_projector_jacek}
\end{figure}

\begin{figure}[t!]
\begin{center}
\includegraphics[width=0.95\textwidth]{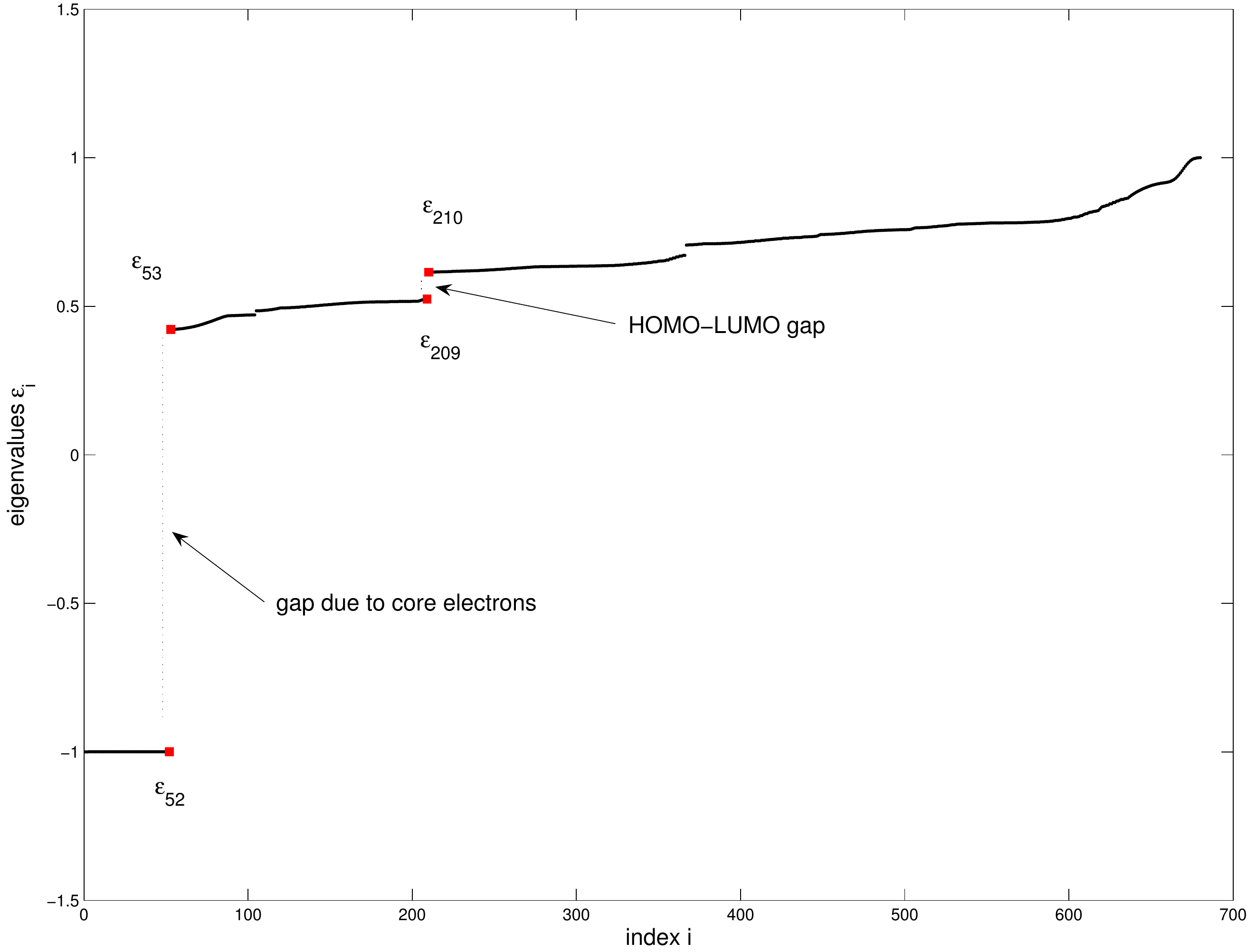}
\end{center}
\vspace{-0.1in}
\caption{Spectrum of the Hamiltonian for $C_{52}H_{106}$.}\label{spectrum_jacek}
\end{figure}

The Fermi--Dirac function can be used to approximate the 
Heaviside function; 
the larger is $\beta$, the better the approximation.
More precisely, the following result is easy to prove (see \cite{nadersthesis}):

\vspace{0.1in}

\begin{prop}\label{prop1}
Let $\delta>0$ be given. If $\beta$ is such that
\begin{equation}
\beta\geq \frac{2}{\gamma} \ln\Big(\frac{1-\delta}{\delta}\Big), \label{approx}
\end{equation}
then $1-f_{FD}({\varepsilon}^-)\leq\delta$ and $
f_{FD}({\varepsilon}^+)\leq\delta$.
\end{prop}

\begin{figure}
\begin{center}
\includegraphics[width=0.95\textwidth]{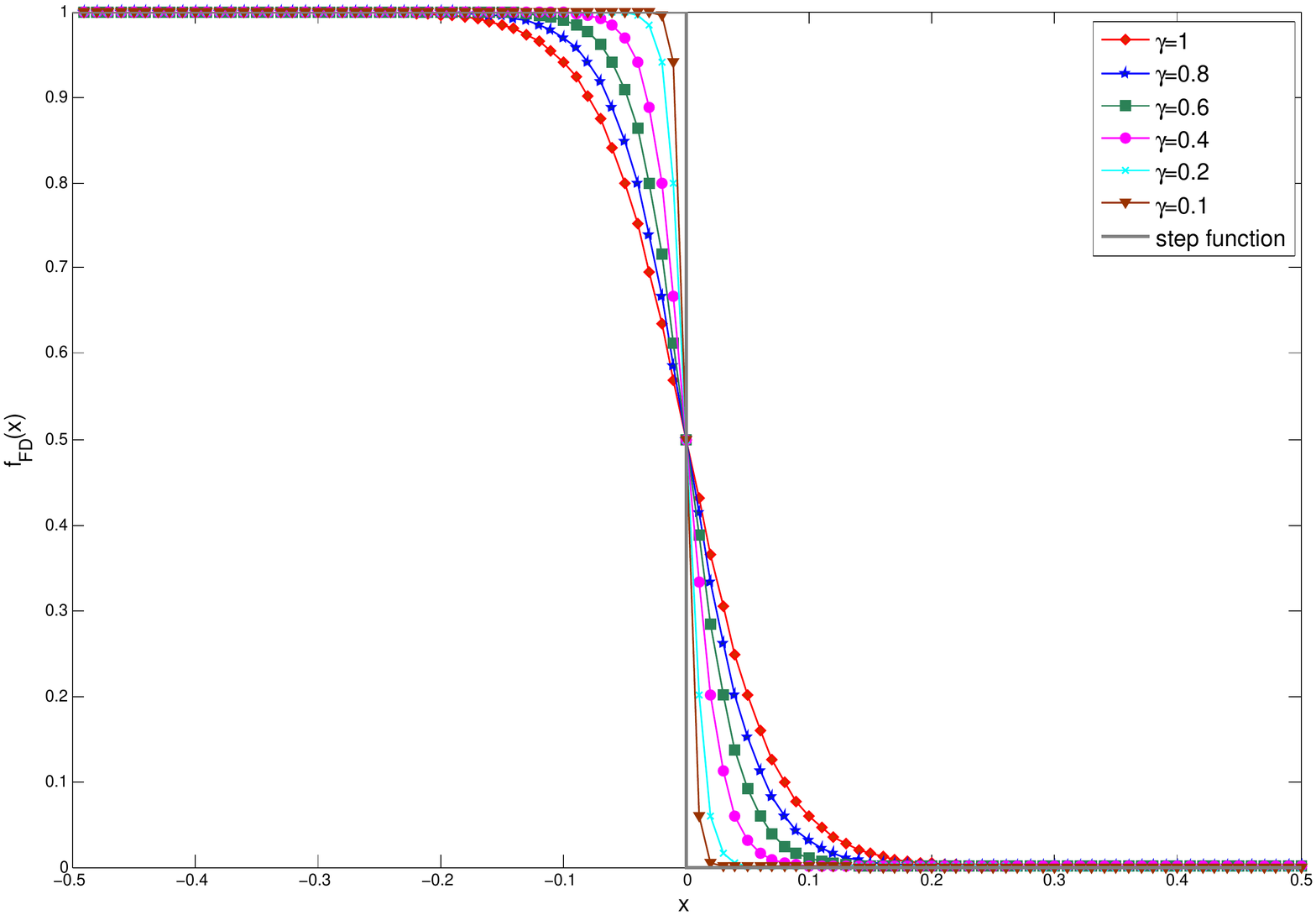}
\vspace{-0.15in}
\caption{Approximations of Heaviside function
by Fermi--Dirac function ($\mu=0$) for different values of
$\gamma$ and $\delta = 10^{-6}$.}\label{FD}
\end{center}
\end{figure}

\begin{figure}
\begin{center}
\includegraphics[width=0.95\textwidth]{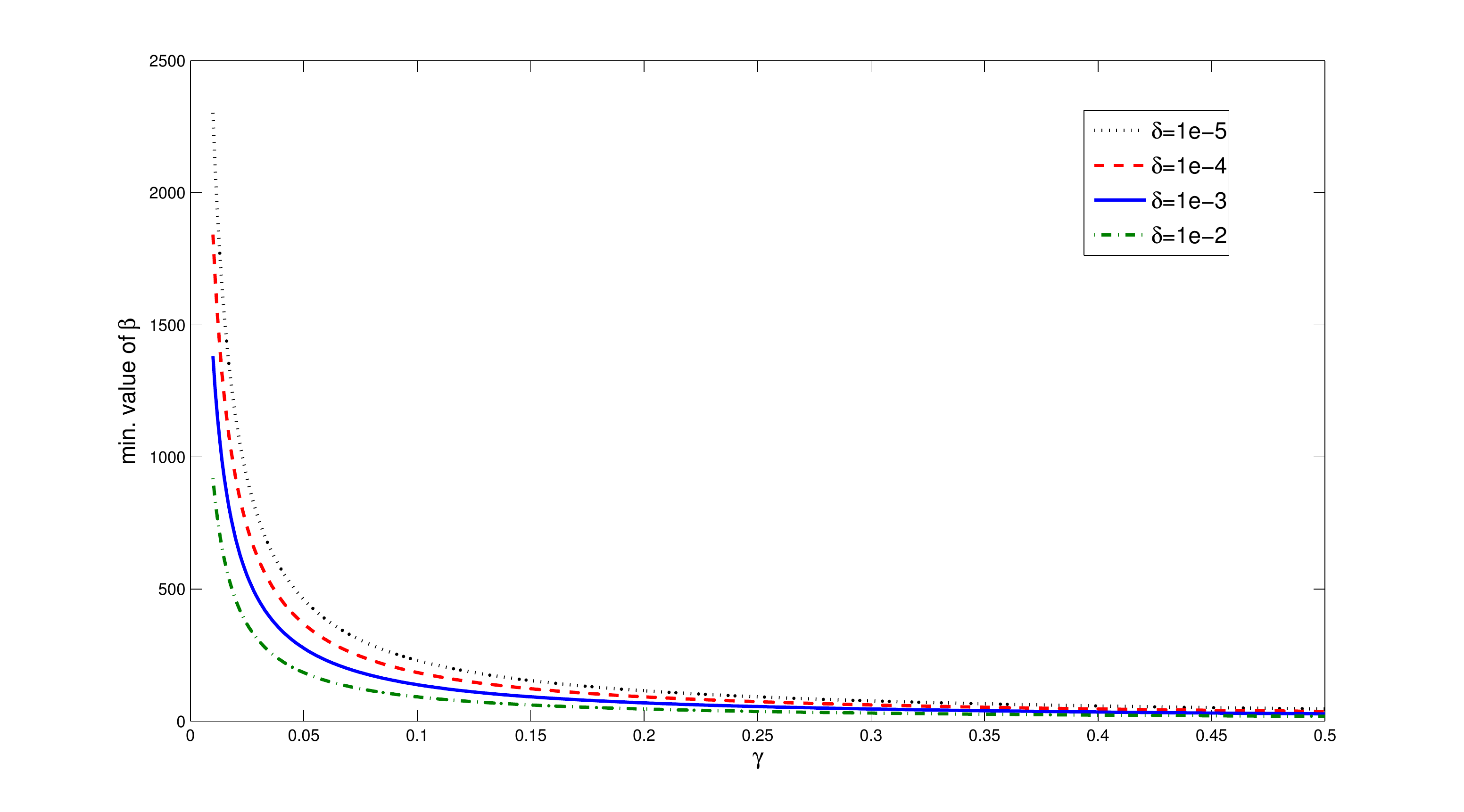}
\vspace{-0.2in}
\caption{Behavior of the minimum acceptable value of $\beta$ as a function of $\gamma$, for different values of $\delta$.}\label{betagamma}
\end{center}
\end{figure}

\vspace{0.1in}

In Fig.~\ref{FD} we show Fermi--Dirac approximations
to the Heaviside function (with a jump at $\mu=0$) for different
values of $\gamma$ between $0.1$ and $1$, where $\beta$ has
been chosen so as to reduce the error in Proposition \ref{prop1}
above the value $\delta = 10^{-6}$. The behavior of $\beta$ as
a function of $\gamma$ according to \eqref{approx} is plotted in Fig.~\ref{betagamma}. 

As a consequence of Theorem \ref{thm1} and Proposition \ref{prop1} we have:

\vspace{0.1in}

\begin{cor}\label{cor1}
Let $n_b$ be a fixed positive integer and $n=n_b\cdot n_e$, 
where the integers $n_e$ form a monotonically
increasing sequence. Let $\{H_n\}$ be a sequence of 
Hermitian $n\times n$ matrices with the following properties:
\begin{enumerate}
\item Each $H_n$ has bandwidth $m$ independent of $n$;
\item There exist two fixed intervals $I_1=[-1,a], 
I_2=[b,1]\subset\mathbb{R}$ with $\gamma=b-a>0$, such that 
for all $n=n_b\cdot n_e$,
$I_1$ contains the smallest $n_e$ eigenvalues of $H_n$ 
(counted with their multiplicities) and $I_2$ contains 
the remaining $n-n_e$ eigenvalues.
\end{enumerate}
Let $P_n$ denote the 
$n\times n$ spectral projector 
onto the subspace spanned by the eigenvectors associated 
with the $n_e$ smallest
eigenvalues of $H_n$, for each $n$. Let $\delta>0$ be 
arbitrary. Then there exist constants $c>0, \alpha>0$ 
independent of $n$ such that
\begin{equation}
|[P_n]_{ij}|\leq 
\min\left \{1,c\, {\rm e}^{-\alpha |i-j|}\right \}+\delta , 
\qquad {\rm for \,\, all} \quad i\neq j.\label{bound_band}
\end{equation}
The constants $c$ and $\alpha$ can be computed from 
\eqref{const_c} and \eqref{const_alpha}, where $\chi$ is 
chosen in the interval
$(1,\overline{\chi})$, with $\overline{\chi}$ given by 
\eqref{horribleformula} and $\beta$ such that \eqref{approx} holds.
\end{cor}

\vspace{0.1in}

Corollary \ref{cor1} allows us to determine \emph{a priori} 
a bandwidth $\bar m$ independent of $n$ outside of which the 
entries of $P_n$ are
smaller than a prescribed tolerance $\tau>0$.
Observe that it is not possible to incorporate $\delta$ 
in the exponential bound, but, at least in principle, 
one may always choose
$\delta$ smaller than a certain threshold. For instance, 
one may take $\delta<\tau/2$ and define $\bar{m}$ as 
the smallest integer
value of $m$ such that the relation 
$c\, {\rm e}^{-\alpha m}\leq\tau/2$ holds.

In the case of Hamiltonians with a general sparsity 
pattern one may apply Theorem \ref{thm2} to obtain a more 
general version of Corollary
\ref{cor1}. If the fixed bandwidth hypothesis is removed, 
the following bound holds:
\begin{equation}
|[P_n]_{ij}|\leq 
\min\left \{1, c\, {\rm e}^{-\theta d_n(i,j)}\right \}+\delta , 
\qquad {\rm for \,\, all}\quad  i\neq j,\label{bound_general}
\end{equation}
with $\theta=\ln \chi$.
Once again, for the result to be meaningful some restriction
on the sparsity patterns, like the uniformly bounded maximum
degree assumption already discussed, must be imposed.

\subsection{Proof of decay bounds}\label{proofs}
Theorem \ref{thm1} is a consequence of results proved in 
\cite{benzigolub} (Thm.~2.2) and \cite{nadersthesis} (Thm.~2.2); 
its proof
relies on a fundamental result in polynomial approximation theory known 
as Bernstein's Theorem \cite{meinardus}. Given a function $f$ continuous 
on $[-1,1]$ and
a positive integer $k$, the $k$th best
approximation error for $f$ is the quantity
$$
E_k(f)=\inf\left \{\max_{-1\leq x\leq 1}|f(x)-p(x)|:p\in P_k\right \},
$$
where $P_k$ is the set of all polynomials with real 
coefficients and degree less than or equal to $k$.
Bernstein's theorem describes the asymptotic behavior of the
best approximation error for a function $f$ analytic on 
a domain containing the interval $[-1,1]$.

Consider the family of ellipses in the complex 
plane with foci in $-1$
and $1$. As already mentioned, an ellipse in this family is completely determined 
by the sum $\chi>1$ of its half-axes and will be denoted 
as $\mathcal{E}_{\chi}$.

\vspace{0.05in}

\begin{theorem}\label{bern} [Bernstein]
Let the function $f$ be analytic in the interior of the 
ellipse $\mathcal{E}_{\chi}$ and continuous on 
$\mathcal{E}_{\chi}$. Moreover, assume that
$f(z)$ is real for real $z$. Then
$$
E_k(f)\leq\frac{2M(\chi)}{\chi^k(\chi-1)},
$$
where $M(\chi)=\max_{z\in\mathcal{E}_{\chi}}|f(z)|$.
\end{theorem}

\vspace{0.05in}

Let us now consider the special case where 
$f(z):=f_{FD}(z)=1/(1+{\rm e}^{\beta(z-\mu)})$ is the 
Fermi--Dirac function of parameters
$\beta$ and $\mu$. Observe that $f_{FD}(z)$ has poles 
in $\mu\pm\iu\frac{\pi}{\beta}$, so the admissible values 
for $\chi$ with respect to $f_{FD}(z)$
are given by $1<\chi<\overline{\chi}$, where the parameter 
$\overline{\chi}$ is such that 
$\mu\pm\iu \frac{\pi}{\beta}\in\mathcal{E}_{\overline{\chi}}$
(the {\em regularity ellipse} for $f = f_{FD}$).
Also observe that smaller values of $\beta$ correspond to a
greater distance between the poles of $f_{FD}(z)$ and the real
axis, which in turn yields a larger value of $\bar{\chi}$. In other words, 
the smaller $\beta$, the faster the decay in Theorem \ref{thm1}. 
Explicit computation of $\overline{\chi}$ yields 
\eqref{horribleformula}.

Now, let $H_n$ be as in Theorem \ref{thm1}. We have
$$
\|f_{FD}(H_n)-p_k(H_n)\|_2=
\max_{x\in\sigma(H_n)}|f_{FD}(x)-p_k(x)|\leq E_k(f_{FD})\leq cq^{k+1},
$$
where $c=2\chi M(\chi)/(\chi -1)$ and $q=1/\chi$. 
The Bernstein approximation
of degree $k$ gives a bound on $|[f_{FD}(H_n)]_{ij}|$ 
when $[p_k(H_n)]_{ij}=0$, that is, when $|i-j|>mk$. 
We may also assume $|i-j|\leq m(k+1)$.
Therefore, we have
$$
|[f_{FD}(H_n)]_{ij}|\leq c\, {\rm e}^{m(k+1)\ln(q^{1/m})}=
c\, {\rm e}^{-\alpha m(k+1)}\leq c\, {\rm e}^{-\alpha |i-j|}.
$$

As for Theorem \ref{thm2}, note that for a 
general sparsity pattern we have $[(H_n)^k]_{ij}=0$, 
and therefore $[p_k(H_n)]_{ij}=0$,
whenever $d_n(i,j)>k$. Writing
$d_n(i,j)=k+1$ we obtain
$$
|[f_{FD}(H_n)]_{ij}|\leq c\, (1/\chi)^{k+1}=c\, {\rm e}^{-\theta d_n(ij)}.
$$

Let us now prove Corollary \ref{cor1}. Assume that 
$\beta$ satisfies the inequality \eqref{approx} for 
given values of $\delta$ and $\gamma$.
If we approximate the Heaviside function with 
step at $\mu$ by means of the Fermi--Dirac function 
$f_{FD}(x)=1/(1+{\rm e}^{\beta(x-\mu)})$,
the pointwise approximation error is given by 
$g(x)={\rm e}^{\beta(x-\mu)}/(1+{\rm e}^{\beta (x-\mu)})$ 
for $x<\mu$ and by $f_{FD}(x)$ for $x>\mu$. It
is easily seen that $g(x)$ is a monotonically 
increasing function, whereas $f_{FD}$ is monotonically 
decreasing. As a consequence,
for each Hamiltonian $H_n$ we have that 
$1-f_{FD}(\lambda)\leq\delta$ for all eigenvalues 
$\lambda\in I_1$ and $f_{FD}(\lambda)\leq\delta$
for all $\lambda\in I_2$. In other words, the 
pointwise approximation error on the spectrum 
of $H_n$ is always bounded by $\delta$. Therefore, we have
$$
|[P_n-f_{FD}(H_n)]_{ij}|\leq \|P_n-f_{FD}(H_n)\|_2\leq\delta .
$$
We may then conclude using Theorem \ref{thm1}:
$$
 |[P_n]_{ij}|\leq |[f_{FD}(H_n)]_{ij}|+
\delta\leq c\, {\rm e}^{-\alpha|i-j|}+\delta .
$$
Finally, recall that in an orthogonal projector
no entry can exceed unity in absolute value. With this in mind,
(\ref{bound_band}) and (\ref{bound_general}) readily follow.

\subsection{Additional bounds}

Theorems \ref{thm1} and \ref{thm2} rely on Bernstein's 
result on best polynomial approximation. Following the same 
argument, one may derive decay bounds for the density
matrix from any other estimate on the best 
polynomial approximation error for classes of functions that include the
Fermi--Dirac function. 
For instance, consider the following result of 
Achieser (see \cite[Thm.~78]{meinardus}, and \cite{achieser}):
\vspace{0.05in}

\begin{theorem}\label{thm_achieser}
Let the function $f$ be analytic in the interior of the 
ellipse $\mathcal{E}_{\chi}$. Suppose that 
$|{\rm Re}\, f(z)|<1$ holds in $\mathcal{E}_{\chi}$
and that $f(z)$ is real for real $z$. Then the following bound holds:
\begin{equation}
E_k(f)\leq \frac{4}{\pi}
\sum_{\nu=0}^\infty \frac{(-1)^{\nu}}
{(2\nu+1)\cosh ((2\nu+1)(k+1)\ln\chi)}.\label{infinite_sum}
\end{equation}
\end{theorem}  

\vspace{0.05in}

The series in \eqref{infinite_sum} converges quite fast; 
therefore, it suffices to compute a few terms explicitly 
to obtain a good approximation
of the bound. A rough estimate shows that, in 
order to approximate the right hand side of 
\eqref{infinite_sum} within a tolerance $\tau$, one may
truncate the series after $\nu_0$ terms, where 
$r^{\nu_0}<\tau(1-r)$ and $r=\chi^{-{\frac{k+1}{2}}}$.

Observe that, as in Bernstein's results, there 
is again a degree of arbitrariness in the choice of 
$\chi$. However, the
admissible range for $\chi$ is smaller here because 
of the hypothesis $|{\rm Re}\, f(z) |<1$.

The resulting matrix decay bounds have the form
\begin{equation}\label{bound_achies}
|[f_{FD}(H_n)]_{ij}|\leq \frac{4}{\pi}
\sum_{\nu=0}^\infty \frac{(-1)^{\nu}}
{(2\nu+1)\cosh ((2\nu+1)(d(i,j)+1)\ln\chi)} 
\end{equation}
for the case of general sparsity patterns. While these bounds
are less transparent than those derived from Bernstein's
Theorem, they are computable. We found that the bounds (\ref{bound_achies})
improve on \eqref{bound} for entries close to the 
main diagonal, but do not seem to have a better asymptotic 
behavior. A possibility would be to combine the two bounds
by taking the smaller between the two values.



So far we have only considered bounds based on best approximation
of analytic functions defined on a single interval. In \cite{hasson},
Hasson has obtained an interesting result on 
polynomial approximation of a step function defined on the union of two 
symmetric intervals. Let $a,b\in\mathbb{R}$ 
with $0<a<b$ and let ${\rm sgn}(x)$ be the sign function 
defined on $[-b,-a]\cup [a,b]$, i.e., 
${\rm sgn}(x)=-1$ on $[-b,-a]$ and ${\rm sgn}(x)=1$ on $[a,b]$. Notice that
the sign function is closely related to the 
Heaviside function $h(x)$, since we have 
$h(x)=\frac{1}{2}(1+{\rm sgn}(x))$. 

\vspace{0.2in}

\begin{prop}
There exists a positive constant $K$ such that
\begin{equation}
E_k({\rm sgn};[-b,-a]\cup [a,b])\leq 
K\frac{\Big(\sqrt{\frac{b-a}{b+a}}\,\Big)^k}{\sqrt{k}}.\label{hasson_bound}
\end{equation}
\end{prop}

Given a sequence of Hamiltonians $\{H_n\}$ with gapped spectra, 
one may choose $a$ and $b$ and shift $H_n$, 
if necessary, so that the spectrum of each $H_n$
is contained in $[-b,-a]\cup [a,b]$ and the 
eigenvalues corresponding to occupied states 
belong to $[-b,-a]$. Then we obtain the following
decay bound for the density matrix:
\begin{equation}\label{hasson_decay}
\left |[P_n]_{ij}\right |\leq 
K\frac{{\rm e}^{-\xi d(i,j)}}{2\sqrt{d(i,j)}},\qquad {\rm where} \quad
\xi=\frac{1}{2}\ln\frac{b+a}{b-a}\,.
\end{equation}  
Under the bounded maximal degree condition, the rate
of decay is independent of $n$.

A few remarks on \eqref{hasson_decay} are in order:
\begin{itemize}
\item Since \eqref{hasson_decay} relies directly 
on a polynomial approximation of the step function, we do not need here
the extra term $\delta$ found in \eqref{bound_general}.
\item Unfortunately, it is not possible to assess 
whether \eqref{hasson_decay} may be useful in 
practice without an explicit formula -- or
at least an estimate -- for the constant $K$. 
The asymptotic decay rate, however, is faster than
exponential and indeed faster
than for other bounds; a comparison
is shown in Fig.~\ref{fig_hasson} (top). 
Notice that this logarithmic plot is only meant to 
show the slope of the bound (which is computed
for $K=1$). 
\item A disadvantage of \eqref{hasson_decay} is 
the requirement that the intervals containing 
the spectra $\sigma(H_n)$ should be 
symmetric with respect to $0$. Of course one 
may always choose $a$ and $b$ so that this 
hypothesis is satisfied, but the quality of the
decay bound deteriorates if $b$ (or $-b$) is 
not close to the maximum (resp., minimum) eigenvalue; 
see Fig.~\ref{fig_hasson} (bottom). 
The blue curve shows the slope of the decay 
bound for $a=0.25$ and $b=1$, in a logarithmic scale. 
In green we display the behavior
of the first row of the density matrix associated 
with a tridiagonal $100\times 100$ matrix 
with spectrum in $[-1,-0.25]\cup [0.25,1]$. 
The red plot refers to the first row of the 
density matrix associated with a matrix 
with spectrum in $[-0.4375,-0.25]\cup [0.25,1]$. The
first matrix is clearly better approximated 
by the decay bound than the second one.
\end{itemize}


\begin{figure}[t!]
\begin{center}
\includegraphics[width=0.85\textwidth]{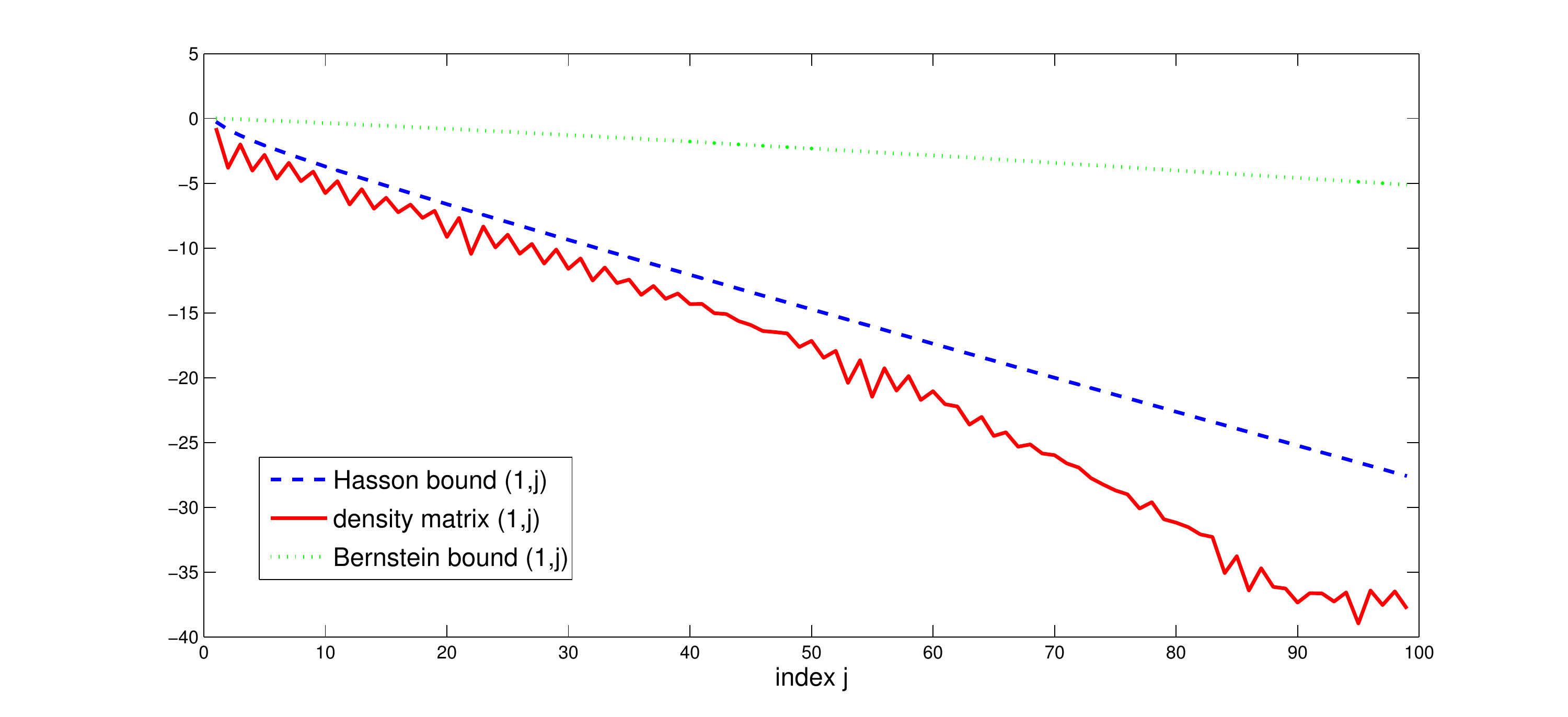}
\end{center}
\begin{center}
\includegraphics[width=0.85\textwidth]{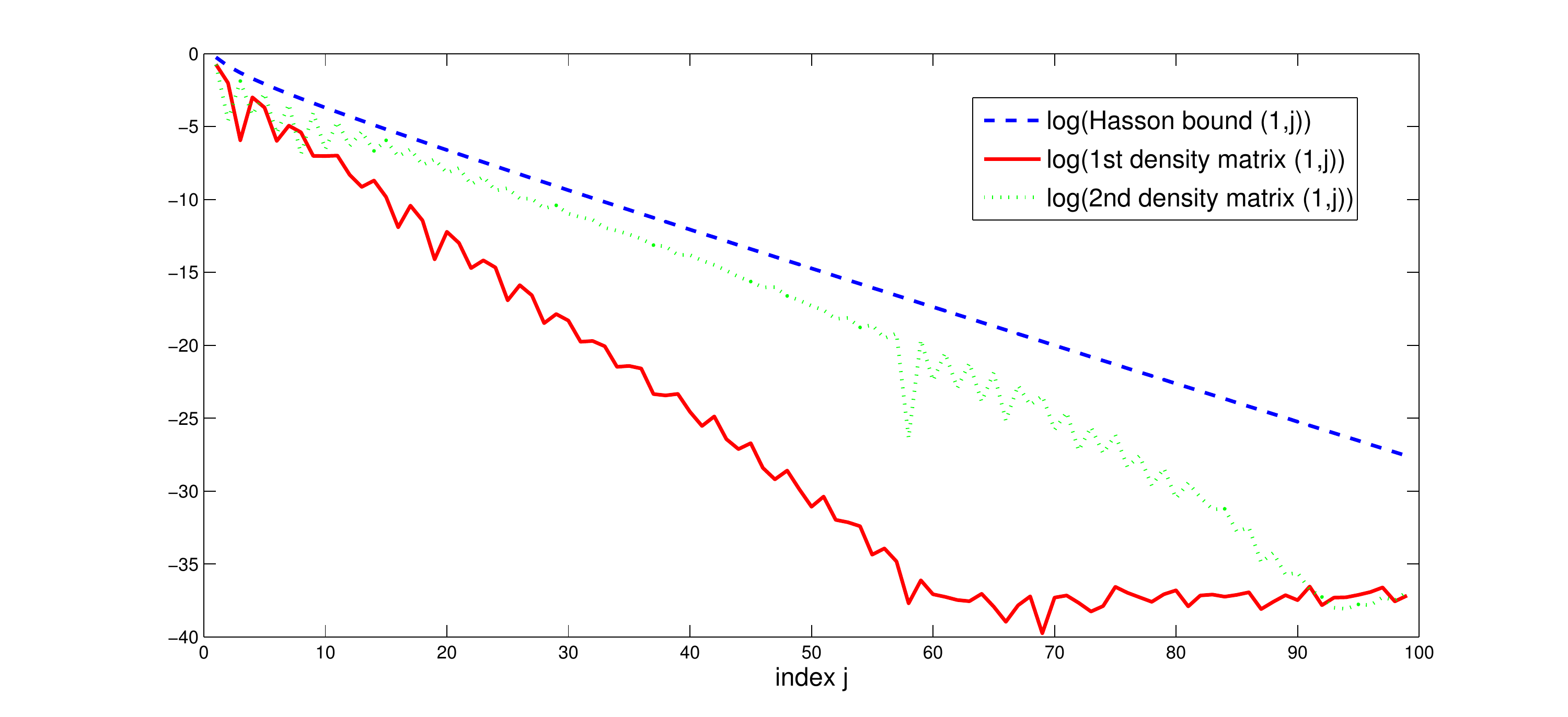}
\end{center}
\caption{Top: logarithmic plot of Hasson (blue) 
and Bernstein-type (green) decay bounds, for 
a $100\times 100$ tridiagonal matrix
with spectrum in $[-1,-0.25]\cup [0.25,1]$. 
The first row of the \lq\lq exact\rq\rq\  
density matrix is plotted in red. Bottom:
logarithmic plot of
Hasson decay bounds (blue) and first rows of 
density matrices associated with matrices 
with different eigenvalue distributions
(red and green).}\label{fig_hasson}
\end{figure}

As one can see from the two plots in Fig.~\ref{fig_hasson},
even for $c=K=1$
both types of decay bounds are rather conservative, and estimating
the truncation bandwidth $\bar m$ needed to achieve a prescribed
error from these bounds would lead to an overly large
band. Hence, the bounds may not be very useful in practice. 
For further discussion of these issues, see section \ref{pract}.

\subsection{Further results}
Let us assume again that we have a sequence $\{H_n\}$ of Hermitian $n\times n$
Hamiltonians (with $n=n_b\cdot n_e$, $n_b$ fixed, $n_e \to \infty$) such that
\begin{itemize} 
\item The matrices $H_n$ are banded with uniformly bounded bandwidth, or
sparse with graphs having uniformly bounded maximum degree;
\item the spectra $\sigma(H_n)$ are uniformly bounded;
\item
the sequence $\{H_n\}$ has a \lq\lq stable\rq\rq\ spectral gap, i.e., there
exist real numbers $g_1<g_2$
such that $[g_1,g_2]\cap\sigma(H_n)=\emptyset$ for sufficiently large $n$.
\end{itemize}
In this subsection we let
\begin{itemize}
\item
$\mu:=(g_2+g_1)/2$ (Fermi level),
\item
$\gamma:=g_2-\mu=\mu-g_1$ (absolute spectral
gap). 
\end{itemize}
Note that because of the uniformly bounded spectra assumption,
the absolute spectral gap is within a constant of the relative gap
previously defined.

Chui and Hasson study in \cite{chuihasson} the asymptotic
behavior of the error of best polynomial
approximation for a sufficiently smooth function $f$
defined on the set $I=[-b,-a]\cup [a,b]$,
with $0<a<b$. Denote as $\mathcal{C}(I)$ the space of
real-valued continuous functions on $I$,
with the uniform convergence norm. Then we have (see
\cite[Thm.~1]{chuihasson} and \cite{lorentz}):

\begin{theorem}\label{chtheorem}
Let $f\in\mathcal{C}(I)$ be such that $f|_{[-b,-a]}$ is
the restriction of a function $f_1$
analytic on the left half plane ${\rm Re}\, z<0$ and $f|_{[a,b]}$
is the restriction of a function $f_2$
analytic on the right half plane ${\rm Re}\, z>0$. Then
$$
\limsup_{k\to \infty}\left [ E_k(f,I)\right ]^{1/k}\leq\sqrt{\frac{b-a}{b+a}},
$$
where $E_k(f,I)$ is the error of best polynomial
approximation for $f$ on $I$.
\end{theorem}

The authors of \cite{chuihasson} observe that the above result cannot be
obtained by extending $f(x)$ to a continuous
function on $[-b,b]$ and applying known bounds for polynomial
approximation over a single interval.
Theorem \ref{chtheorem} looks potentially useful for our
purposes, except that it provides an asymptotic result,
rather than an explicit bound for each value of $k$.
Therefore, we need to reformulate the argument
in \cite{chuihasson}. To this end, we prove a variant
of Bernstein's Theorem (cf.~Theorem \ref{bern}) adapted
to our goals. Instead of working on the interval
$[-1,1]$ we want to bound the approximation error
on the interval $[a^2,b^2]$.

\begin{theorem}\label{cor2}
Let $f\in\mathcal{C}([a^2,b^2])$ be the restriction of
a function $f$ analytic in the interior of the ellipse
$\mathcal{E}_{a^2,b^2}$ with foci in $a^2, b^2$ and a
vertex at the origin. Then, for all $\xi$ with
$$
1<\xi<\overline{\xi}:=\frac{a+b}{a-b}, 
$$
there exists a constant $K$ such that
$$
E_k(f,[a^2,b^2])\leq K\left (\frac{1}{\xi} \right)^k.
$$
\end{theorem}
\begin{proof}
The proof closely parallels the argument given in \cite{meinardus}
for the proof of Theorem \ref{bern}.
First of all, observe that 
the ellipse $\mathcal{E}_{\chi}$ in Bernstein's Theorem has foci in $\pm 1$ and
vertices in $\pm (\chi +1/\chi)/2$ and
$\pm (\chi -1/\chi)/2$. The parameter $\chi$ is the sum
of the lengths of the semiaxes. Similarly,
the ellipse $\mathcal{E}_{a^2,b^2}$ has foci in $a^2, b^2$
and vertices in $0$, $a^2+b^2$
 and $(a^2+b^2)/2\pm \iu ab$.
 Also observe that $\overline{\xi}$ is the sum of the lengths
of the semiaxes of $\mathcal{E}_{a^2,b^2}$,
normalized w.r.t. the semifocal length, so that it plays
exactly the same role as $\chi$ for
$\mathcal{E}_{\chi}$.
Now we look for a
conformal map that sends an annulus in the complex plane
to the ellipse
where $f$ is analytic. When this ellipse is $\mathcal{E}_{\chi}$,
a suitable map is
$u=c(v)=(v+1/v)/2$, which sends the annulus
$\chi ^{-1}<|v|<\chi$ to $\mathcal{E}_{\chi}$.
When the desired ellipse has foci in $a^2,b^2$,
we compose $c(v)$ with the
change of variable
$$
x=\psi (u)=\left ( u+\frac{a^2+b^2}{b^2-a^2}\right )\frac{b^2-a^2}{2},
$$
thus obtaining a function that maps the annulus
$\mathcal{A}=\left\{\xi^{-1} <|v|<\xi\right\}$ to an
ellipse. Denote this ellipse as $\mathcal{E}_{a^2,b^2,\xi}$ and observe that it is
contained in the interior of
$\mathcal{E}_{a^2,b^2}$.
Therefore we have that the function
$$
f(\psi(c(v)))=
f\left (\left [\frac{1}{2}\left ( v+\frac{1}{v}\right )
+\frac{a^2+b^2}{b^2-a^2}\right ]\frac{b^2-a^2}{2} \right )
$$
is analytic on $\mathcal{A}$ and continuous on $|v|=\xi $.
The proof now proceeds as in the original
Bernstein Theorem. The Laurent expansion
$$
f(\psi(c(v)))=\sum_{\nu =-\infty}^{\infty}\alpha_{\nu}v^{\nu}
$$
converges in $\mathcal{A}$ with $\alpha_{-\nu}=\alpha_{\nu}$.
Moreover, we have the bound
$$
|\alpha_{\nu}|=\left | \frac{1}{2\pi \iu}\int_{|v|=\xi}
\frac{f(\psi(c(v)))}{v^{\nu +1}}dv \right |\leq\frac{M(\xi)}{\xi ^{\nu}},
$$
where $M(\xi)$ is the maximum value (in modulus) taken
by $f$ on the ellipse $\mathcal{E}_{a^2,b^2,\xi}$.

Now observe that $u=c(v)$ describes the real interval
$[-1,1]$ for $|v|=1$, so for $u\in [-1,1]$ we have
$$
f(\psi (u))=\alpha _0+2\sum_{\nu =1}^{\infty}\alpha _{\nu}T_{\nu}(u),
$$
where $T_{\nu}(u)$ is the $\nu$-th Chebyshev polynomial.
Since $\psi (u)$ is a linear transformation, we have
$E_k(f(z),[a^2,b^2])=E_k(f(u),[-1,1])$, so from the theory
of Chebyshev approximation \cite{meinardus} we obtain
$$
E_k(f,[a^2,b^2])=E_k(f(u),[-1,1])\leq 2M(\xi)
\sum_{\nu=k+1}^{\infty}\xi ^{-\nu}=\frac{2M(\xi)}{\xi-1}\xi ^{-k},
$$
hence the thesis.
Note that the explicit value of $K$ is computable.
\end{proof}

The following result is based on \cite[Thm.~1]{chuihasson}.

\begin{theorem}\label{thm3}
Let $f\in\mathcal{C}(I)$ be as in Theorem \ref{chtheorem}.
Then, for all $\xi$ with
$$
1<\xi<\overline{\xi}:=\frac{a+b}{a-b}, 
$$
there exists $C>0$ independent of $k$ such that
$$
E_k(f,I)\leq C\xi^{-\frac{k}{2}}.
$$
\end{theorem}

\begin{proof} 
Let $P_k$ and $Q_k$ be polynomials of best uniform approximation
of degree $k$ on the interval $[a^2,b^2]$ for
the functions $f_2(\sqrt{x})$ and $f_2(\sqrt{x})/\sqrt{x}$,
respectively. Then by Theorem \ref{cor2} there are
constants $K_1$ and $K_2$ such that
\begin{equation}
\max_{x\in [a^2,b^2]} |P_k(x)-f_2(\sqrt{x})| \leq K_1\xi^{-k}\label{approx1}
\end{equation}
and
\begin{equation}
\max_{x\in [a^2,b^2]}|Q_k(x)-f_2(\sqrt{x})/\sqrt{x}|
\leq K_2\xi^{-k}.\label{approx2}
\end{equation}
We use the polynomials $P_k$ and $Q_k$ to define a third
polynomial $R_{2k+1}(x):=[P_k(x^2)+xQ_k(x^2)]/2$, of degree
$\leq 2k+1$, which approximates $f(x)$ on $[a,b]$
and has small norm on $[-b,-a]$. Indeed, from \eqref{approx1}
and \eqref{approx2} we have:
\begin{equation}\label{firstdis}
\begin{split}
\max_{x\in [a,b]}|R_{2k+1}(x)-f(x)| & \leq\frac{1}{2}\max_{x\in [a,b]}|P_k(x^2)-f(x)|+
\frac{1}{2}\max_{x\in [a,b]}|xQ_k(x^2)-f(x)|\\
& \leq\frac{1}{2}K_1\xi^{-k} + \frac{1}{2}bK_2\xi^{-k} =
\frac{K_1 + bK_2}{2}\xi^{-k}
\end{split}
\end{equation}
and
\begin{eqnarray}
\hspace{-.4in} \max_{x\in [-b,-a]} |R_{2k+1}(x)|\leq\frac{1}{2}
\max_{x\in [a,b]}|P_k(x^2)-f(x)+
f(x)-xQ_k(x^2)|\\
\hspace{-.2in} \leq\frac{1}{2}\max_{x\in [a,b]}|P_k(x^2)-f(x)|+
\frac{1}{2}\max_{x\in [a,b]}|xQ_k(x^2)-f(x)|
\leq\frac{K_1 + bK_2}{2}\xi^{-k}.
\end{eqnarray}
Similarly, we can find another polynomial $S_{2k+1}(x)$ such that
\begin{equation}
\max_{x\in [-b,-a]} |S_{2k+1}(x)-f(x)|\leq\frac{K_3 + bK_4}{2}\xi^{-k}
\end{equation}
and
\begin{equation}
\max_{x\in [a,b]}|S_{2k+1}(x)|\leq\frac{K_3 + bK_4}{2}\xi^{-k}.\label{lastdis}
\end{equation}
Then, from the inequalities \eqref{firstdis}-\eqref{lastdis} we have
\begin{eqnarray*}
\max_{x\in I}|R_{2k+1}(x)+S_{2k+1}(x)-f(x)|\leq \max_{x\in [a,b]}|R_{2k+1}(x)-f(x)|
+\max_{x\in [a,b]}|S_{2k+1}(x)|\\
+\max_{x\in [-b,-a]}|S_{2k+1}(x)-f(x)|+\max_{x\in [-b,-a]}|R_{2k+1}(x)|\\
\leq (K_1+K_3+b(K_2+K_4))\xi^{-k}\,,
\end{eqnarray*}
and therefore
$$
E_k(f,I)\leq \sqrt{\xi}\,(K_1+K_3+b\, (K_2+K_4))\,\xi^{-\frac{k}{2}},
$$
for odd values of $k$, and
$$
E_k(f,I)\leq \xi (K_1+K_3+b\, (K_2+K_4)) \xi^{-\frac{k}{2}}
$$
for even values of $k$.
This completes the proof. \end{proof}

In the following we assume, without loss of generality, that
$k$ is odd.
In order to obtain bounds on the density matrix,
we apply  Theorem \ref{thm3} to the step
function $f$ defined on $I$ as follows:
$$
f(x)=\left\{\begin{array}{ccc}
1&{\rm for}&-b\leq x\leq -a\\
0&{\rm for}&a\leq x\leq b\,,\end{array}\right.
$$
i.e., $f$ is the restriction of $f_1(z)\equiv 1$ on $[-b,-a]$
and the restriction of $f_2(z)\equiv 0$
on $[a,b]$.
 Here the polynomial
approximation of $f_2(\sqrt{x})$,
$f_2(\sqrt{x})/\sqrt{x}$ and $f_1(\sqrt{-x})$ is exact,
so we have $K_1=K_2=K_3=0$.
As for $K_4$, observe that $|1/\sqrt{z}|$ achieves its maximum
on the vertex of $\mathcal{E}_{a^2,b^2,\xi}$
with smallest abscissa; therefore we have
$$
K_4=\frac{2M(\xi)}{\xi-1},
$$
where
$$
M(\xi)=\frac{1}{\sqrt{z_0}}\quad\textnormal{with}\quad
z_0=\left[-\frac{1}{2}\left(\xi+\frac{1}{\xi}\right)+
\frac{a^2+b^2}{b^2-a^2}\right]\frac{b^2-a^2}{2}.
$$
Moreover, we find $R_{2k+1}(x)\equiv 0$
and $S_{2k+1}(x)=(1+xV_k(x^2))/2$,
where $V_k(x)$ is the polynomial of best uniform
approximation for $1/\sqrt{x}$ on $[a^2,b^2]$.
Thus, we obtain the bound
$$
E_k(f,I)\leq C \xi ^{-\frac{k}{2}},
$$
where
$C$ is given by
$$
C=\sqrt{\xi}\,K_4\,b.
$$

Let us now apply this result to our sequence of Hamiltonians.
We will assume that
the matrices are shifted so that $\mu=0$, that is,
we replace each $H_n$
by $H_n-\mu I_n$. Under this hypothesis, the natural
choice for $a$ is $a=\gamma$,
whereas $b$ is the smallest number such that
$\sigma(H_n)\subset [-b,-a]\cup [a,b]$
for every $n$.

Using the same argument used in section \ref{proofs}
for the derivation of matrix decay bounds
(see also \cite{benzigolub}
and \cite{benzirazouk}), we can obtain
bounds on the off-diagonal entries
of $f(H_n)$. If $H_n$ is banded with bandwidth $m$
independent of $n$, we have
\begin{equation}
\left |[P_n]_{ij} \right |=|[f(H_n)]_{ij}|\leq\sqrt{\xi}\,
\frac{2M(\xi)}{\xi-1}\,b\, \xi ^{-\frac{|i-j|}{2m}},\label{bound3}
\end{equation}
whereas if $H_n$ has a more general sparsity pattern we obtain
\begin{equation}
\left |[P_n]_{ij} \right |=|[f(H_n)]_{ij}|\leq\sqrt{\xi}\,
\frac{2M(\xi)}{\xi-1}\,b\, \xi ^{-\frac{d_n(i,j)}{2}},\label{bound2}
\end{equation}
where $d_n(i,j)$ is the distance between nodes $i$ and $j$
in the graph $G_n$ associated with $H_n$.

Next, we compare the bounds derived in this section with those
for the Fermi--Dirac approximation of the step function obtained
in section \ref{sec_fd}, using a suitable choice of the inverse
temperature $\beta$. Recall that if 
$\mathcal{E}_{\chi}$ denotes the
regularity ellipse for the Fermi--Dirac function,
the earlier bounds for the banded case are:
\begin{equation}
\left |[P_n]_{ij} \right |\leq\frac{2M(\chi)}{\chi-1}
\left(\frac{1}{\chi}\right)^{\frac{|i-j|}{m}}.\label{oldbound}
\end{equation}
For ease of computation, we assume in this section that
$\mu=0$ and that the spectrum of each matrix $H_n$ is contained
in $[-1,1]$. As explained in section \ref{sec_fd},
once $\gamma$ is known, we pick a tolerance $\delta$
and compute $\beta$ so that the Fermi--Dirac
function provides a uniform approximation of the step function
with error $\le \delta$ outside the gap:
$$
\beta\geq \frac{2}{\gamma} \ln\Big(\frac{1-\delta}{\delta}\Big) .
$$
Then the supremum of the set of admissible values of $\chi$, which ensures
optimal asymptotic decay in this framework, is
$$
\overline{\chi}=\left (\pi+\sqrt{\beta^2+\pi^2}\right )/\beta.
$$

\begin{figure}[t!]
\begin{center}
\includegraphics[width=0.70\textwidth]{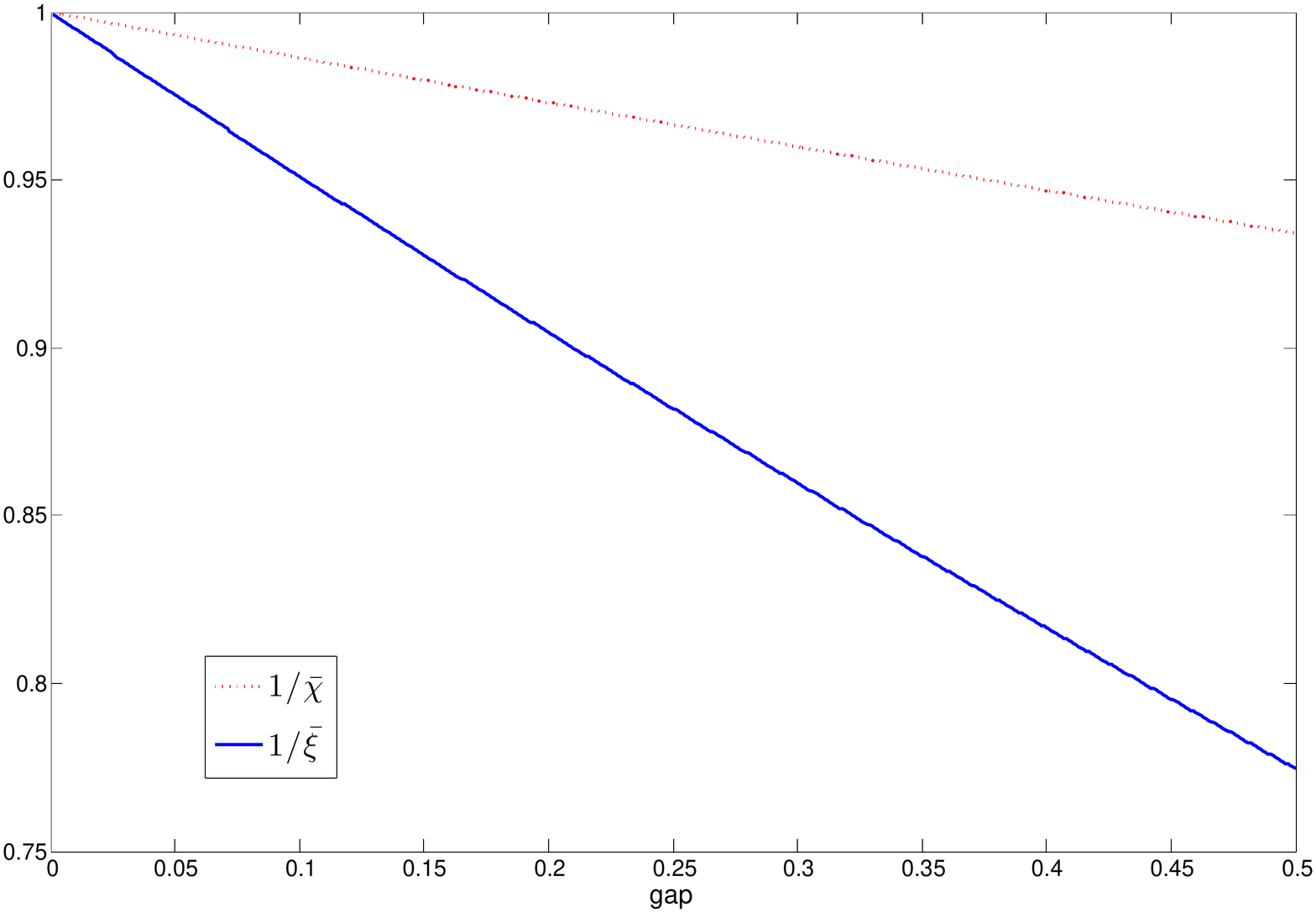}
\vspace{-0.15in}
\caption{Comparison of parameters
$1/\,\overline{\xi}$ and $1/\,\overline{\chi}$ for several values of the spectral gap.
Here $\delta=10^{-5}$.}
\label{chi_xi_1}
\end{center}
\end{figure}

\begin{figure}[h!]
\begin{center}
\includegraphics[width=0.70\textwidth]{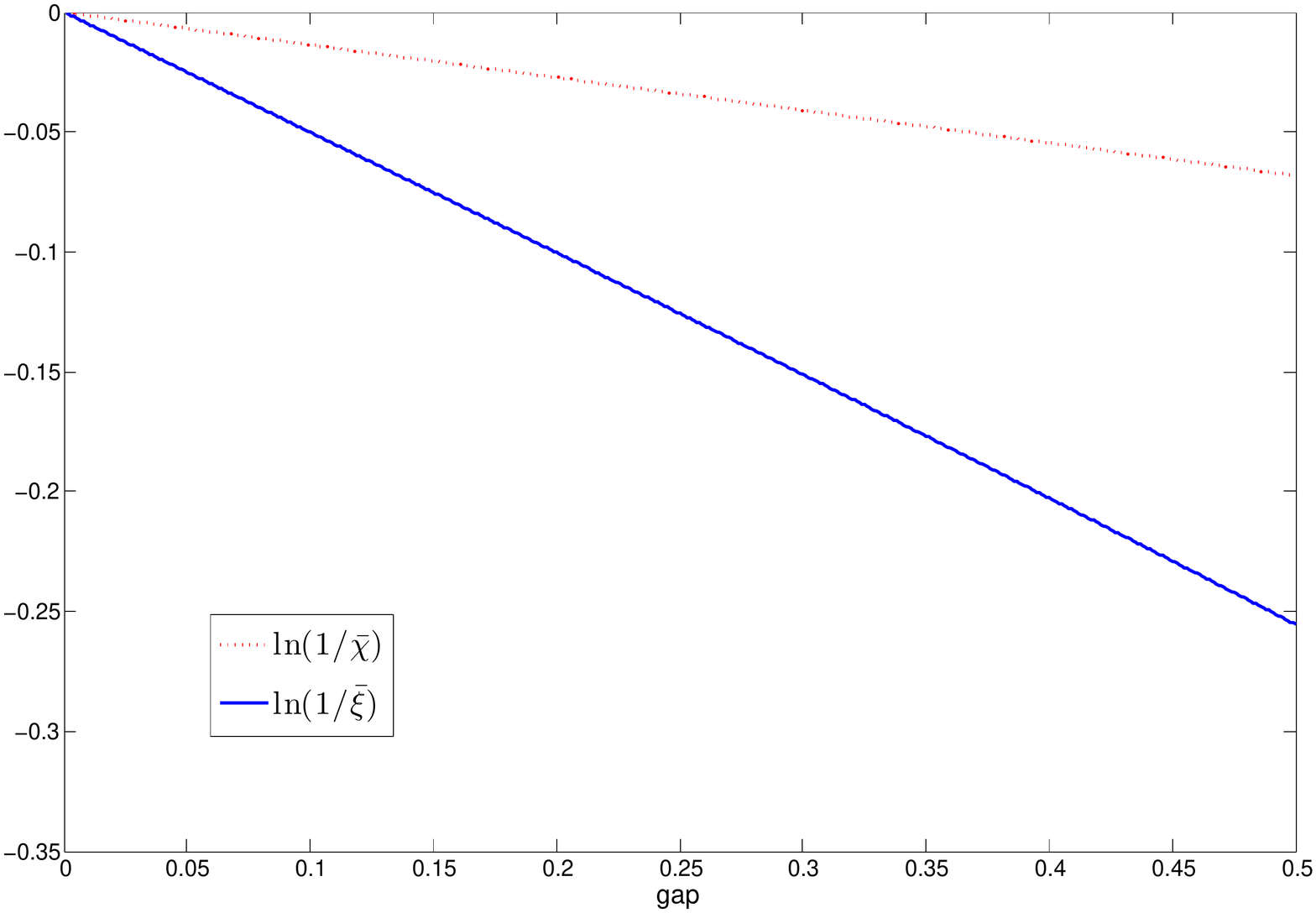}
\vspace{-0.15in}
\caption{Logarithmic plot of parameters $1/\,\overline{\xi}$ and
$1/\,\overline{\chi}$ w.r.t.~several values of the spectral gap.
Here $\delta=10^{-5}$.}
\label{chi_xi_2}
\end{center}
\end{figure}

Figures \ref{chi_xi_1} and \ref{chi_xi_2} compare the
values of $1/\,\overline{\xi}$ and $1/\,\overline{\chi}$ (which
characterize the behavior of the bounds \eqref{bound3}
and \eqref{oldbound}, respectively). Note that
in general we find $1/\,\overline{\xi}<1/\,\overline{\chi}$; this means
that the asymptotic decay rate is higher for the bound based
on disjoint interval approximation. Moreover, the disjoint
interval method directly approximates
the step function and therefore does not require one to choose
a tolerance for \lq\lq intermediate\rq\rq\  approximation.
As a result, the bounds based on disjoint interval approximation
prescribe a smaller truncation bandwidth $\bar m$ in the approximation
to the spectral projector in order to achieve a given level
of error. For instance, in the tridiagonal case ($m=1$) we
observed a factor of three reduction in $\bar m$ compared to the
previous bounds, independent of the size of the gap.

\subsection{Dependence of the rate of decay 
on the spectral gap}\label{dep_gap}

As already mentioned in section \ref{related}, the functional
dependence of the decay length (governing
the rate of decay in the density matrix) 
on the spectral gap has been the subject of some
discussion; see, for instance, \cite{baer,IBA,JK,prodan2,tar1,zhang}. 
Some of these authors have argued that the decay length
decreases like the square root of the gap if the Fermi
level is located near one of the gap edges (i.e., 
close to either $\varepsilon_{n_e}$ or to
$\varepsilon_{n_e+1}$), and like the gap itself if
the Fermi level falls in the middle of the gap.
These estimates hold for the small gap limit.

In this section we address this problem by studying
how the decay described by the bounds
\eqref{bound3} and \eqref{bound2} behaves asymptotically with 
respect to $\gamma$ or, equivalently, with respect to $a$ (see the notation
introduced in the previous section). Note that we are
assuming here that the Fermi level falls exactly in the middle
of the gap.

Let us rewrite
\eqref{bound3} in the form
$$
|[P_n]_{ij}|\leq C{\rm e}^{-\alpha |i-j|/m},
$$
where
$$
\alpha=\frac{1}{2}\ln \xi=\frac{1}{2}\ln\left(\frac{a+b}{b-a}\right).
$$
For a fixed $m$, the decay behavior is essentially
described by the parameter $\alpha$.
Let us assume for simplicity of notation that $b=1$, so that
the spectral gap is normalized
and the expression for $\alpha$ becomes
$$
\alpha=\frac{1}{2}\ln\left(\frac{1+a}{1-a}\right).
$$
The Taylor expansion of $\alpha$ for $a$ small yields
$$
\alpha=a+\frac{a^3}{3}+o\,(a^3)\,.
$$
Therefore, for small values of $\gamma$, the
decay behavior is described at first order by the gap itself,
rather than by a more complicated function of $\gamma$.
This result is consistent with similar ones
found in the literature \cite{IBA,JK,zhang}. The fact that some systems
exhibit density matrix decay lengths proportional to the square 
root of the gap (see, e.g., \cite{JK}) does not contradict
our result: since we are dealing here with upper bounds
a square root-dependence, which corresponds to
faster decay for small $a$, is still consistent with our bounds.
Given that our bounds are completely general, it does
not come as a surprise that we obtain the more conservative
estimate among the alternatives discussed in the literature.

\subsection{Dependence of the rate of decay on the temperature}\label{dep_T}

Another issue that has stirred some controversy in the literature 
concerns the precise rate of decay in the density matrix
in metals at positive temperature; see, e.g., the results and
discussion in \cite{baer,goe,IBA}. Recall that in metals at positive
temperatures $T$, the density matrix $F_n = f_{FD}(H_n)$ decays
exponentially.   The question is whether
the decay length  
is proportional to $T$ or to $\sqrt{T}$, for small $T$.
Our approach shows
that the decay length
is proportional to $T$.

Indeed, from the analysis in section \ref{sec_fd},
in particular Theorems \ref{thm1} and \ref{thm2}, we find that
the decay length $\alpha$ in the exponential 
decay bound (\ref{bound}) (or, more generally,
the decay length $\theta$ in the bound (\ref{generalbound}))  
behaves like ${\rm ln}\,\chi$, 
where -- assuming for simplicity that $\mu = 0$, as before -- the
parameter $\chi$ is any number satisfying 
$$1 < \chi < \overline{\chi}, \quad \overline{\chi} = 
\left (\pi + \sqrt{\beta^2 + \pi^2}\, \right )/\beta\,.$$

Letting $x = \pi/\beta = \pi k_B T$ and observing that for small $x$
$${\rm ln}\left (x + \sqrt{1 + x^2}\,\right) = x + o\,(x^2), $$
we conclude that, at low temperatures,
the decay length is proportional to $k_B T$.
This conclusion is in complete agreement
with the results in \cite{goe,IBA}.
To the best of our knowledge,
this is the first time this result has been established in a
fully rigorous and completely general manner.

\subsection{Other approaches}
Decay bounds on the entries of spectral projectors can also be
obtained from the contour integral representation
\begin{equation}\label{cauchy2}
 P_n = \frac{1}{2\pi \iu}\int_\Gamma (zI_n - H_n)^{-1}dz\,,
\end{equation}
where $\Gamma$ is a simple closed curve (counterclockwise oriented)
in $\complex$ surrounding a portion of the real axis containing
the eigenvalues of $H_n$ which correspond to the occupied states
and only those. Componentwise, (\ref{cauchy2}) becomes
$$[P_n]_{ij}=\frac{1}{2\pi \iu}\int_\Gamma \left [(zI_n - H_n)^{-1}
\right ]_{ij} dz\,, \quad 1\le i,j\le n,$$
from which we obtain
$$\left |[P_n]_{ij}\right | \le 
\frac{1}{2\pi}\int_\Gamma | \left [(zI_n - H_n)^{-1}
\right ]_{ij} |\, dz\,, \quad 1\le i,j\le n.$$
Assume the matrices $H_n$ are banded, with uniformly bounded spectra
and bandwidths as $n\to \infty$.
 By \cite[Prop.~2.3]{demkomosssmith} there exist, 
for all $z \in \Gamma$, explicitly computable constants
$c(z) \geq 0$ and $0 < \lambda(z) < 1$ (independent of $n$) such that
\begin{align}\label{res_est}
\big|\left[(zI_n-H_n)^{-1}\right]_{ij}\big| \leq c(z) [\lambda(z)]^{|i-j|},
\end{align}for all $i,j =1,\dots,n$. Moreover, $c$ and $\lambda$
depend continuously on $z \in \Gamma$. Since $\Gamma$ is compact we
can set
\begin{align}\label{step2}
c = \max_{z\in \Gamma} c(z) \quad \text{and} \quad \lambda =  \max_{z\in \Gamma} \lambda(z).
\end{align} Now let us assume that the matrices $H_n$ have spectral
gaps $\gamma_n$ satisfying $\inf_n \gamma_n > 0$. It is then clear that $c$ is finite 
and that $\lambda \in (0,1)$.
Hence, we obtain the following bound:
\begin{align}\label{cc_bound}
 \left |[P_n]_{ij}\right | \leq \left(c \cdot \frac{\ell(\Gamma)}{2\pi} \right) \lambda^{|i-j|},
\end{align}for all $i,j =1,\dots,n$, where $\ell(\Gamma)$ denotes the length of $\Gamma$.
Finally, letting $C = c\cdot \frac{\ell(\Gamma)}{2\pi}$ and $\alpha = -\ln \lambda$
we obtain the exponential decay bounds
\begin{equation}\label{cauchy_bd}
\left |[P_n]_{ij}\right | \le C\cdot {\rm e}^{-\alpha |i-j|}, \quad 1\le i,j\le n,
\end{equation}
with both $C>0$ and $\alpha > 0$ independent of $n$. As usual, the bounds
can be easily extended to the case of general sparsity patterns.
One disadvantage of this
approach is that explicit evaluation of the constants $C$ and $\alpha$ is 
rather complicated.
  
The integral representation (\ref{cauchy2}) is useful not only as
a theoretical tool, but also increasingly as a computational tool.
Indeed, quadrature rules with suitably chosen nodes
$z_1, \ldots ,z_k\in \Gamma$ can be used to approximate the integral in
(\ref{cauchy2}), leading to 
\begin{equation}\label{rational}
P_n\approx \sum_{i=1}^k w_i (z_i I_n - H_n)^{-1}
\end{equation}
for suitable quadrature weights $w_1,\ldots w_k$.
Note that this amounts to a rational approximation of $P_n=h(H_n)$.
In practice, using the trapezoidal rule with a small number of nodes 
suffices to achieve high
accuracy, due to the exponential convergence of this quadrature
rule for analytic functions \cite{DR}. 
Note that if $P_n$ is real then it is sufficient to use just the $z_i$
in the upper half-plane and then take the real part of the result 
\cite[page 307]{higham}.
If the spectral gap $\gamma_n$
for $H_n$ is not too small, all the resolvents $(z_i I_n - H_n)^{-1}$
decay rapidly away from the main diagonal, with exponential rate
independent of $n_e$. Hence, $O(n)$ approximation is possible, at least
in principle. Rational approximations of the type (\ref{rational}) are
especially useful in those situations where only selected entries
of $P_n$ are required. Then only the corresponding entries of the resolvents
$(z_i I_n - H_n)^{-1}$ need to be computed. For instance, in some
cases only the diagonal entries of $P_n$ are needed \cite{saad}. In others, 
only entries in positions corresponding to the non-zero entries
in the Hamiltonian $H_n$ must be computed; this is the case, 
for instance, when computing 
the objective function $\langle E\rangle = {\rm Tr} (P_nH_n)$ in
density matrix minimization algorithms.
Computing selected entries of a resolvent is not an
easy problem. However, progress has been made on this front in several recent
papers; see, e.g., \cite{li,lin1,lin3,ss09,tang}.

\subsection{Computational considerations}\label{pract}

In the preceding sections we have rigorously established exponential
decay bounds for zero-temperature density matrices corresponding to 
finite-range Hamiltonians with non-vanishing spectral gap (`insulators'), 
as well as for density matrices corresponding to arbitrary
finite-range Hamiltonians at positive electronic temperatures.
Our results are very general and apply to a wide variety
of physical systems and discretizations. Hence, a mathematical
justification of the physical phenomenon of `nearsightedness'
has been obtained, and the possibility of $O(n)$ methods firmly 
established.\footnote{Heuristics 
relating the \lq\lq nearsightedness range of electronic
matter\rq\rq\  and the linear complexity of the
Divide-and-Conquer method of Yang \cite{DC}, essentially
a domain decomposition approach to DFT,
were already given by Kohn himself; see, e.g., \cite{kohn,PK05}.}

Having thus achieved our main purpose, the question remains whether
our estimates can be of practical use in the design of $O(n)$
algorithms. As shown in section \ref{trunc}, having estimated 
the rate of decay in
the density matrix $P$ allows one to prescribe {\em a priori} a
sparsity pattern for the computed approximation $\tilde P$ to $P$.  
Having estimated an `envelope' for the non-negligible entries
in $P$ means that one can estimate beforehand
the storage requirements and set up static data structures 
for the computation of the approximate density matrix 
$\tilde P$. An added advantage is the possibility of
using the prescribed sparsity pattern to develop efficient parallel
algorithms; it is well known that adaptive computations,
in which the sparsity pattern is determined `on the fly', may
lead to load imbalances and loss of parallel efficiency
due to the need for large amounts of communication and
unpredictable memory accesses.  This is completely analogous
to prescribing a sparsity pattern vs.~using an adaptive one
when computing sparse approximate inverses
for use as preconditioners when solving linear systems, see \cite{Benzi02}.

Most of the $O(n)$ algorithms currently in use consist of 
iterative schemes producing increasingly
accurate approximations to the density matrix. These approximations  
may correspond to successive terms in an expansion of $P$
with respect to a prescribed basis
\cite{goedeckercolombo,baerhg3,lin2}, or they may be the
result of a gradient or descent method in density matrix
minimization approaches \cite{CKP08,challacombe,LNV,MS}. Closely related 
methods include {\em purification} and algorithms
based on approximating the sign function \cite{NS00};
we refer again to \cite{bowler2,niklasson,RRS11} for 
recent surveys on the state of the
art of linear scaling methods for electronic structure.
Most of these algorithms construct a sequence of approximations 
$$P^{(0)}, P^{(1)},\ldots ,P^{(k)},\ldots$$
which, under appropriate conditions, converge to $P$. Each iterate
is obtained from the preceding one by some matrix-matrix
multiplication, or powering, scheme; each step introduces new
nonzeros (fill-in), and the matrices $P^{(k)}$ become increasingly
dense. The exponential decay property, however, implies that
most of these nonzeros will be negligible, with only $O(n)$
of them being above any prescribed threshold $\delta >0$. Clearly,
knowing {\em a priori} the location of the non-negligible
entries in $P$ can be used to drastically reduce the computational
burden and to achieve linear scaling, since only those entries
need to be computed. Negligible entries that fall {\em within}
the prescribed sparsity pattern may be removed using a drop
tolerance; this strategy further decreases storage and arithmetic
complexity, but 
its implementation demands the use of dynamic data structures.

An illustration of this use of the decay estimates can be found
for instance
in \cite{benzirazouk}, where a Chebyshev expansion of the Fermi--Dirac
function $f_{FD}(H)$ was used to approximate the density matrix at
finite temperatures. Given a prescribed error tolerance, 
exponential decay bounds were applied to the Fermi--Dirac
function to determine the truncation bandwidth needed to satisfy
the required approximation error. When computing the polynomial
$p_k(H)\approx f_{FD}(H)$ using the Chebyshev expansion, only entries within the
prescribed bandwidth were retained. Combined with an estimate of
the approximation error obtained by monitoring the magnitude of
the coefficients in the Chebyshev expansion, this approach worked
well for some simple 1D model problems resulting
in linear scaling computations. A related approach, 
based on qualitative decay estimates for the density 
matrix, was already used in \cite{baerhg}. These authors  
present computational
results for a variety of 1D and 2D systems 
including insulators at zero
temperature and metals at finite temperature; 
see further \cite{baerhg3}. 

Unfortunately, the practical usefulness of our bounds 
for more realistic calculations is limited.
The bounds are generally pessimistic and tend to be
overly conservative, especially for the case of zero or low temperatures.
This is to be expected, since
the bounds were obtained by
estimating the degree of a polynomial approximation to the Fermi--Dirac 
matrix function needed to satisfy a prescribed error tolerance. These bounds
tend to be rather pessimistic because they do not take into account the
possibility of numerical cancellation when evaluating the matrix polynomial.
For instance, the bounds must apply in the worst-case scenario where the
Hamiltonian has non-negative entries and the approximating polynomial has
nonnegative coefficients. Moreover, the bounds do not take into account the
size of the entries in the Hamiltonian, particularly the fact that the nonzeros
within the band (or sparsity pattern) are not of uniform size but may be
spread out over several orders of magnitude. 
It should be emphasized that the presence of a gap is only a
sufficient condition for localization of the density matrix, not a
necessary one: it has been pointed out, for example in \cite{maslen}, that
disordered systems may exhibit strong localization
even in the absence of a well-defined gap. 
This is the case, for instance, of the 
{\em Anderson model of localization} in condensed
matter physics \cite{And58}.
Obviously, our approach is unable to account for such phenomena 
in the zero temperature case. 
The theory reviewed in this paper is primarily a qualitative one;
nevertheless, it captures many of the features of actual physical systems, like
the asymptotic dependence of the decay rate on the gap size or on
the electronic temperature.  

 A natural question is whether the bounds can be improved to the point
where they can be used to obtain practical estimates of the entries in 
the density matrix. In order to achieve this, 
additional assumptions on the 
Hamiltonians would be needed, making the theory less general. In other
words, the price we pay for the generality of our theory is 
that we get pessimistic bounds. 
Recall that for a given sparsity pattern in the
normalized Hamiltonians $H_n$ our decay bounds depend on just 
one essential parameter, the 
gap $\gamma$. Our bounds are the same no matter 
what the eigenvalue distribution is to the left of the
highest occupied level, $\varepsilon_{n_e}$, 
and to the right to the lowest unoccupied one, $\varepsilon_{n_e+1}$.
 If more spectral information were at hand,
the bounds could be improved. 
 The situation is very similar to 
that arising in the derivation of error bounds for the convergence
of Krylov methods, such as the CG method for solving symmetric positive
definite linear systems $Ax=b$; see, e.g., \cite[Theorem 10.2.6]{GVL}.
 Bounds based on the spectral condition
number $\kappa_2 (A)$ alone, while sharp, do not in general capture the actual
convergence behavior of CG. They represent the
worst-case behavior, which is rarley observed in practice. 
Much more accurate bounds can
be obtained by making assumptions on the distribution of
the eigenvalues of $A$. For instance, if $A$ has only $k$ distinct
 eigenvalues, then the CG method converges (in exact arithmetic)
to the solution $x_*=A^{-1}b$ in at most $k$ steps. Similarly, suppose
the Hamiltonian $H_n$ has only $k < n$ distinct eigenvalues (with $\mu$ 
not one of them), and that the multiplicities of the eigenvalues 
to the left of $\mu$ add up
to $n_e$, the number of electrons. Then there
is a polynomial $p_k(\lambda)$ of degree at most $k-1$ such that 
$p_k(H_n)=P_n$, the density matrix. This is just the interpolation
polynomial that takes the value 1 on the eigenvalues to the
left of $\mu$, and zero on the eigenvalues to the right of $\mu$.
This polynomial \lq\lq approximation\rq\rq\ is actually exact.
 If $k\ll n$, and is independent of $n$,
then $P_n$ will be a matrix with $O(n)$ nonzero entries; moreover,
the sparsity pattern of $P_n$ can be determined {\em a priori} from the
graph structure of $H_n$. Another situation is that in which the
eigenvalues of $H_n$ fall in a small number $k$ of narrow bands,
or tight clusters,
with the right-most band to the left of $\mu$ well-separated
from the left-most band to the right of $\mu$. In this case we
can find again a low-degree polynomial $p_k(\lambda)$ with $p_k(H_n)\approx P_n$,
and improved bounds can be obtained. 

The problem, of course, is that these are rather special eigenvalue
distributions, and it is difficult to know 
{\em a priori} whether such conditions
hold or not. 

Another practical issue that should be at least briefly mentioned is
the fact that our bounds assume knowledge of lower and upper bounds on
the spectra of the Hamiltonians $H_n$, as well as estimates for
the size and location of the spectral gap (this is also needed in order
to determine the Fermi level $\mu$). These issues have received a great
deal of attention in the literature, and here we limit ourselves to
observe that $O(n)$ procedures exists to obtain sufficiently accurate
estimates of these quantities; see, e.g., \cite{goedecker1}.

\section{Transformation to an orthonormal basis}\label{overlap}

In this section we discuss the transformation of an Hamiltonian from
a non-orthogonal to an orthogonal basis. The main point is that while
this transformation results in matrices with less sparsity, the transformed
matrices retain the decay properties of the original matrices, only with
(possibly) different constants. What is important, from the point of view of
asymptotic complexity, is that the rate of decay remains
independent of system size.

We begin with a discussion of decay in the inverse of the overlap matrix. 
To this end, consider a sequence $\{S_n\}$ of overlap matrices of
size $n = n_b\cdot n_e$, with $n_b$ constant and $n_e$ increasing to infinity. 
We make the following assumptions:

\begin{enumerate}
\item Each $S_n$ is a banded symmetric positive definite (SPD) matrix 
with unit diagonal entries and
with bandwidth uniformly bounded with respect to $n$;
\item The spectral condition number (ratio of the largest to the smallest
eigenvalue) of each $S_n$, $\kappa_2(S_n)$, is uniformly bounded with respect
to $n$. Because of assumption 1, this is equivalent to 
requiring that the smallest eigenvalue of
$S_n$ remains bounded away from zero, for all $n$.
\end{enumerate}

As always in this paper, the
bandedness assumption in item 1 is not essential and can be replaced by
the weaker hypothesis that each $S_n$ is sparse and that the corresponding
graphs $\{G_n\}$ have bounded maximal degree with respect
to $n$. Actually, it would be enough to require that the sequence $\{S_n\}$ 
has the exponential decay property relative to a sequence of graphs $\{G_n\}$
of bounded maximal degree. In order to simplify the discussion, and also in
view of the fact that overlap matrices usually exhibit 
exponential or even super-exponential
decay, we assume from the outset that each $S_n$ has already been truncated
to a sparse (or banded) matrix. Again, this is for notational convenience only,
and it is straightforward to modify the following arguments to account for the
more general case.  
On the other hand, the assumption on
condition numbers in item 2 is essential and cannot be weakened. 

\begin{remark}
We note that assumption 2 above
is analogous to the condition that the sequence
of Hamiltonians $\{H_n\}$ has spectral gap bounded below uniformly in $n$; while
this condition ensures (as we have shown) the exponential decay property in the
associated spectral projectors $P_n$, assumption 2 above insures exponential decay
in the inverses (or inverse factors) of the overlap matrices. Both conditions  
amount to asking that the corresponding problems be uniformly well-conditioned
in $n$. The difference is that the decay on the spectral projectors
depends on the spectral gap of the Hamiltonians and therefore on the
nature of the system under study (i.e., insulator vs.~metallic
system),
whereas the sparsity and spectral properties of the overlap matrices 
depend on other features of the system, mainly the inter-atomic distances. 
\end{remark}

In the following we shall need some basic results  
on the decay of the inverses \cite{demkomosssmith}, 
inverse Cholesky factors \cite{benzituma} and inverse
square roots (L\"owdin factors) \cite{benzigolub}
of banded SPD matrices; see also \cite{jaffard}.

Let $A$ be SPD and $m$-banded, and let $a$ and $b$ denote the
smallest and largest eigenvalue of $A$, respectively. Write
$\kappa$ for the spectral condition number $\kappa_2(A)$ of $A$  
(hence, $\kappa = b/a$).
Define
$$q := \frac{\sqrt{\kappa} - 1}{\sqrt{\kappa} + 1} \quad
{\rm and} \quad  \lambda :=q^{1/m} \,.$$ 
Furthermore, let $K_0 :=(1+\sqrt{\kappa})^2/(2b)$. In \cite{demkomosssmith},
Demko et al.~obtained the following bound on the entries of $A^{-1}$:
\begin{equation}\label{dms}
|[A^{-1}]_{ij}| \le K\, \lambda^{|i-j|}, \quad 1\leq i,j \leq n,
\end{equation}
where $K:=\max \{a^{-1},K_0\}$. Note that the bound (\ref{dms}) `blows up'
as $\kappa \to \infty$, as one would expect.

As shown in \cite{benzituma}, the decay bound (\ref{dms}) and the bandedness
assumption on $A$ imply a similar decay bound on the inverse Cholesky
factor $Z=R^{-1} = L^{-T}$, where $A= R^TR = LL^T$ with $R$ upper triangular
($L$ lower triangular). Assuming that $A$ has been scaled so that 
$\max_{1\le i \le n} A_{ii} = 1$ (which is automatically true if $A$ is 
an overlap matrix corresponding to a set of normalized basis functions), we have
\begin{equation}\label{z_bound}
|Z_{ij}| \le K_1\,  \lambda^{j-i}, \quad  j \ge i \,,
\end{equation}
with $K_1 = K\frac{1 - \lambda^m}{1 - \lambda}$; here $K$, $\lambda$ are
the same as before. We further note that while $K_1 > K$, for some classes
of matrices it is possible to show that the actual 
magnitude of the $(i,j)$ entry of $Z$
(as opposed to the bound (\ref{z_bound})) is actually less than the magnitude
of the corresponding entry of $A^{-1}$. This is true, for instance, for 
an irreducible $M$-matrix; see \cite{benzituma}.

Finally, let us consider the inverse square root, $A^{-1/2}$. In \cite{benzigolub}
the following bound is established:
\begin{equation}\label{l_bound}
\left |[A^{-1/2}]_{ij} \right | \le K_2\, \lambda^{|i-j|}, \quad 1\le i,j\le n \,.
\end{equation}
Here $K_2$ depends again on the extreme eigenvalues $a$ and $b$ of $A$,
whereas $\lambda = q^{1/m}$, where now $q$ is any number satisfying the 
inequalities
$$\frac{\sqrt{\kappa} - 1}{\sqrt{\kappa} + 1} < q < 1 \,.$$
As before, the bound (\ref{l_bound}) blows up as $\kappa \to \infty$, as one would expect.

Introducing the positive scalar $\alpha = - {\rm ln}\, \lambda$, we can rewrite
all these bounds in the form
$$|B_{ij}| \le K\, {\rm e}^{-\alpha |i-j|}, \quad 1\le i,j \le n $$
for the appropriate matrix $B$ and suitable constants $K$ and $\alpha > 0$. 

Let now $\{S_n\}$ be a sequence of $n\times n$ overlap
matrices, where $n=n_b\cdot n_e$ with $n_b$ fixed and $n_e\to \infty$. 
Assuming that the matrices $S_n$ satisfy the above assumptions 1-2,
then their inverses
satisfy the uniform exponential decay bounds (\ref{dms}), with 
$K$ and $\lambda$ 
constant and independent of $n$. Hence, as discussed in section
\ref{trunc}, for any given $\epsilon > 0$
there exists an integer $\bar m$ independent of $n$ such that
each matrix $S_n$ in the sequence can be approximated, in norm,
by an $\bar m$-banded matrix with an error less than $\epsilon$.
As usual, this result can be extended from the banded case to the
sparse case, assuming that the corresponding graphs $G_n$ have
bounded maximal degree as $n\to \infty$.
Moreover, under assumptions 1-2 above, 
the inverse Cholesky factors $Z_n$ satisfy a uniform (in $n$) exponential decay
bound of the type (\ref{z_bound}), and therefore uniform approximation 
with banded triangular matrices is possible. Again, generalization to
more general sparsity patterns is possible, provided the usual assumption
on the maximum degree of the corresponding graphs $G_n$ holds.
Similarly, under the same conditions we obtain a uniform rate
of exponential decay for the entries of the inverse 
square roots $S_n^{-1/2}$, with a corresponding
result on the existence of a banded (or sparse) approximation. 

Let us now consider the sequence of transformed Hamiltonians, $\tilde H_n
= Z_n^TH_nZ_n$. Here $Z_n$ denotes either the inverse Cholesky factor or
the inverse square root of the corresponding overlap matrix $S_n$. Assuming
that the sequence $\{H_n\}$ satisfies the off-diagonal exponential decay property and
that $\{S_n\}$ satisfies assumptions 1-2 above, it follows from the decay
properties of the matrix sequence $\{Z_n\}$ that the sequence $\{\tilde H_n\}$
also enjoys off-diagonal exponential decay. This is a straightforward
consequence of the following result, which is adapted from a similar
one for infinite matrices due to Jaffard \cite[Proposition 1]{jaffard}. 

\begin{theorem}
Consider two sequences $\{A_n\}$ and $\{B_n\}$
of $n\times n$ matrices (where $n\to \infty$) whose entries satisfy
$$\left |[A_n]_{ij}\right | \le c_1\, {\rm e}^{-\alpha |i-j|} \quad {\rm and} 
  \quad \left |[B_n]_{ij} \right | \le c_2\, {\rm e}^{-\alpha |i-j|}\,, 
\quad 1\le i,j\le n\,,$$
where $c_1$, $c_2$ and $\alpha >0$ are independent of $n$. Then
the sequence $\{C_n\}$, where $C_n = A_nB_n$, satisfies a similar bound:
\begin{equation}\label{c_bound}
\left |[C_n]_{ij}\right | \le c\, {\rm e}^{-\alpha' |i-j|}, \quad 1\le i,j\le n\,,
\end{equation}
for any $0<\alpha'<\alpha$, with $c$ independent of $n$.
\end{theorem} 
\begin{proof}
First note that the entries of each $A_n$ clearly satisfy 
$$\left |[A_n]_{ij}\right | \le c_1\, {\rm e}^{-\alpha' |i-j|} \quad {\rm for\,\, any}
\quad  \alpha' < \alpha\,.$$ 
Let $\omega = \alpha - \alpha'$. Then $\omega > 0$ and the entries
$[C_n]_{ij}$ of $C_n= A_nB_n$ satisfy
$$\left |[C_n]_{ij}\right | \le \sum_{k=1}^n 
\left |[A_n]_{ik}\right | \left |[B_n]_{kj} \right |
\le c_1 c_2 \, \left (\sum_{k=1}^n {\rm e}^{-\omega |k-j|}\right ){\rm e}^{-\alpha' |i-j|}\,.$$
To complete the proof just observe that for any $j$, 
$$\sum_{k=1}^n {\rm e}^{-\omega |k-j|} 
= \sum_{k=0}^{j-1}{\rm e}^{-\omega k} + 
  \sum_{k=1}^{n-j} {\rm e}^{-\omega k}
< \sum_{k=0}^{\infty} {\rm e}^{-\omega k} +
  \sum_{k=1}^{\infty} {\rm e}^{-\omega k}
= \frac{1+ {\rm e}^{-\omega}}
{1- {\rm e}^{-\omega}}\,.$$
Since the last term is independent of $n$, the entries of $C_n$ satisfy
(\ref{c_bound}) with a constant $c$ that is also independent of $n$.
\end{proof}

The foregoing result can obviously be extended to the product of three 
matrices. Thus, the entries of the matrix sequence $\{\tilde H_n\}$,
where $\tilde H_n = Z_n^TH_nZ_n$, enjoy the exponential off-diagonal decay
property:
$$\left |[\tilde  H_n]_{ij} \right | \le c \, {\rm e}^{-\alpha |i-j|},
\quad 1 \le i,j \le n\,,$$
for suitable constants $c$ and $\alpha > 0$. 

Alternatively, one could first approximate $H_n$ and $Z_n$ with banded matrices
$\bar{H}_n$ and $\bar{Z}_n$ and then define the 
(approximate) transformed Hamiltonian
as $\tilde H_n := \bar{Z}_n^T \bar{H}_n \bar{Z}_n$, 
possibly subject to further 
truncation. Using the fact that both $H_n$ and $Z_n$ have 2-norm bounded
independently of $n$, 
it is easy to show that the final approximation error can be reduced below
any prescribed tolerance by reducing the error in $\bar{H}_n$ and 
$\bar{Z}_n$.
Hence, with either approach,
the transformed Hamiltonians $\tilde H_n$ can be approximated uniformly
in $n$ within a prescribed error by banded matrices of constant bandwidth,
just like the original (\lq\lq non-orthogonal\rq\rq) Hamiltonians. While the
bandwidth of the approximations will be larger than for the original
Hamiltonians, the truncated matrices retain a good deal of sparsity and
asymptotically contain $O(n)$  non-zeros.
Hence, we have a justification of the statement 
(see section \ref{sec:Introduction}) that in our theory we
can assume from the outset that the basis set $\{\phi_i\}_{i=1}^n$ is
orthonormal.

In Fig.~\ref{fig_H_jacek} we show the Hamiltonian $H$ for the
already mentioned linear alkane ${\rm C}_{52}{\rm H}_{106}$
(see section \ref{sec_density}) 
discretized in a Gaussian-type orbital basis (top) and
the `orthogonalized' Hamiltonian $\tilde H = \bar{Z}^T \bar{H} 
\bar{Z}$
(bottom). This figure shows that while the transformation to orthogonal
basis alters the magnitude of the entries in the Hamiltonian,
the bandwidth of $\tilde H$ (truncated to a tolerance of $10^{-8}$) is
only slightly wider than that of $H$. In this case the overlap
matrix $S$ is well-conditioned, hence the entries of $Z$ exhibit fast
decay. An ill-conditioned overlap matrix would lead to a less sparse
transformed Hamiltonian $\tilde H$.   

As usual, the bandedness assumption was made for simplicity of
exposition only; similar bounds can be obtained for more general sparsity
patterns, assuming the matrices $H_n$ and $S_n$ have the exponential
decay property relative to a sequence $\{G_n\}$ of graphs having 
maximal degree uniformly bounded with respect to $n$.

It is important to emphasize that in practice, the explicit formation
of $\tilde H_n$ from $H_n$ and $Z_n$ is not needed and is never carried
out. Indeed, in all algorithms for electronic structure computation  
the basic matrix operations are matrix-matrix and matrix-vector products,
which can be performed without explicit transformation of the
Hamiltonian to an orthonormal basis. On the other hand, for the
study of the decay properties it is convenient to assume that all the
relevant matrices are explicitly given in an orthogonal representation.  

\begin{figure}[t!]
\vspace{-1.1in}
\begin{center}
\includegraphics[width=0.65\textwidth]{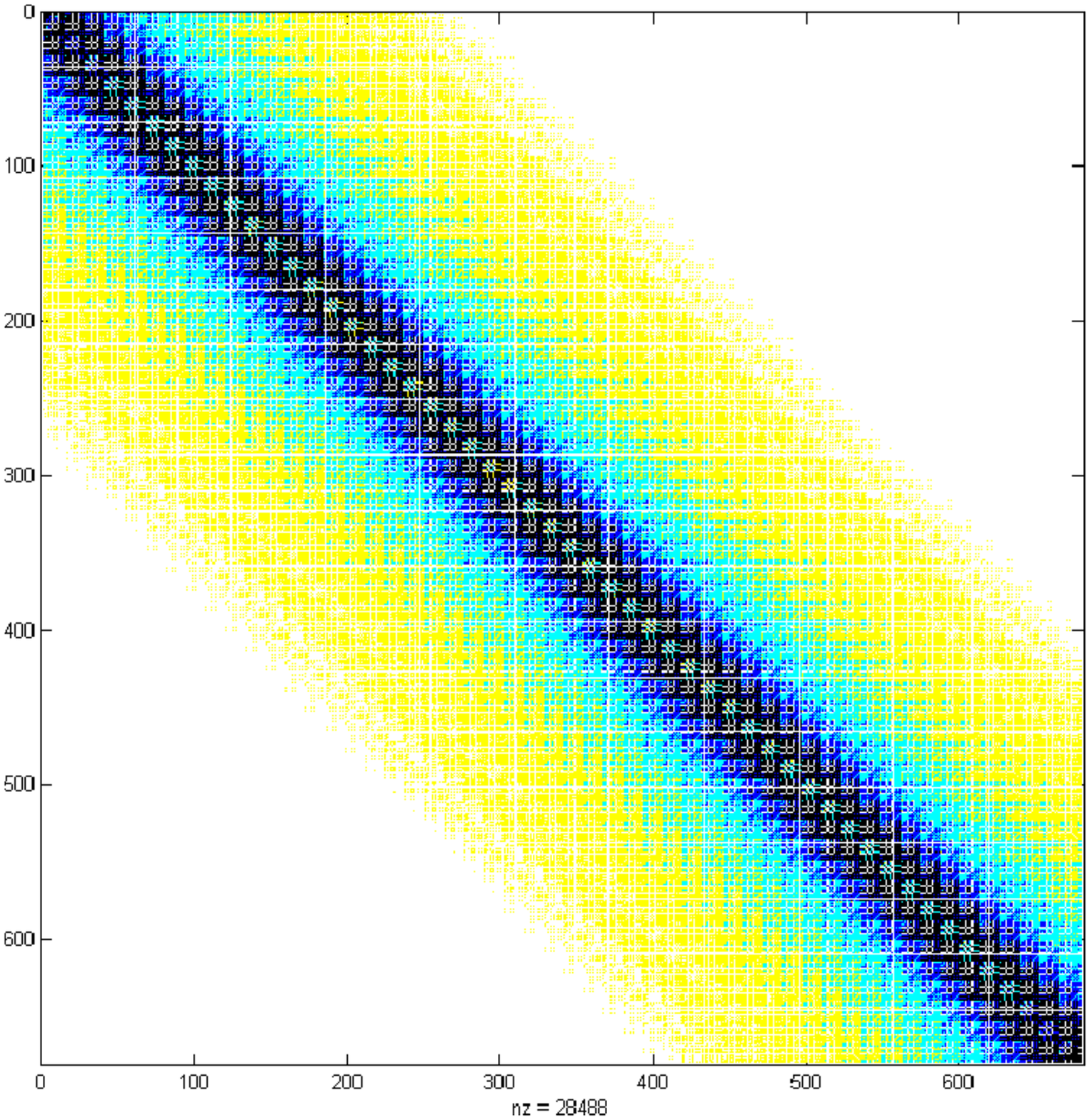}
\end{center}
\vspace{-1.1in}
\begin{center}
\includegraphics[width=0.65\textwidth]{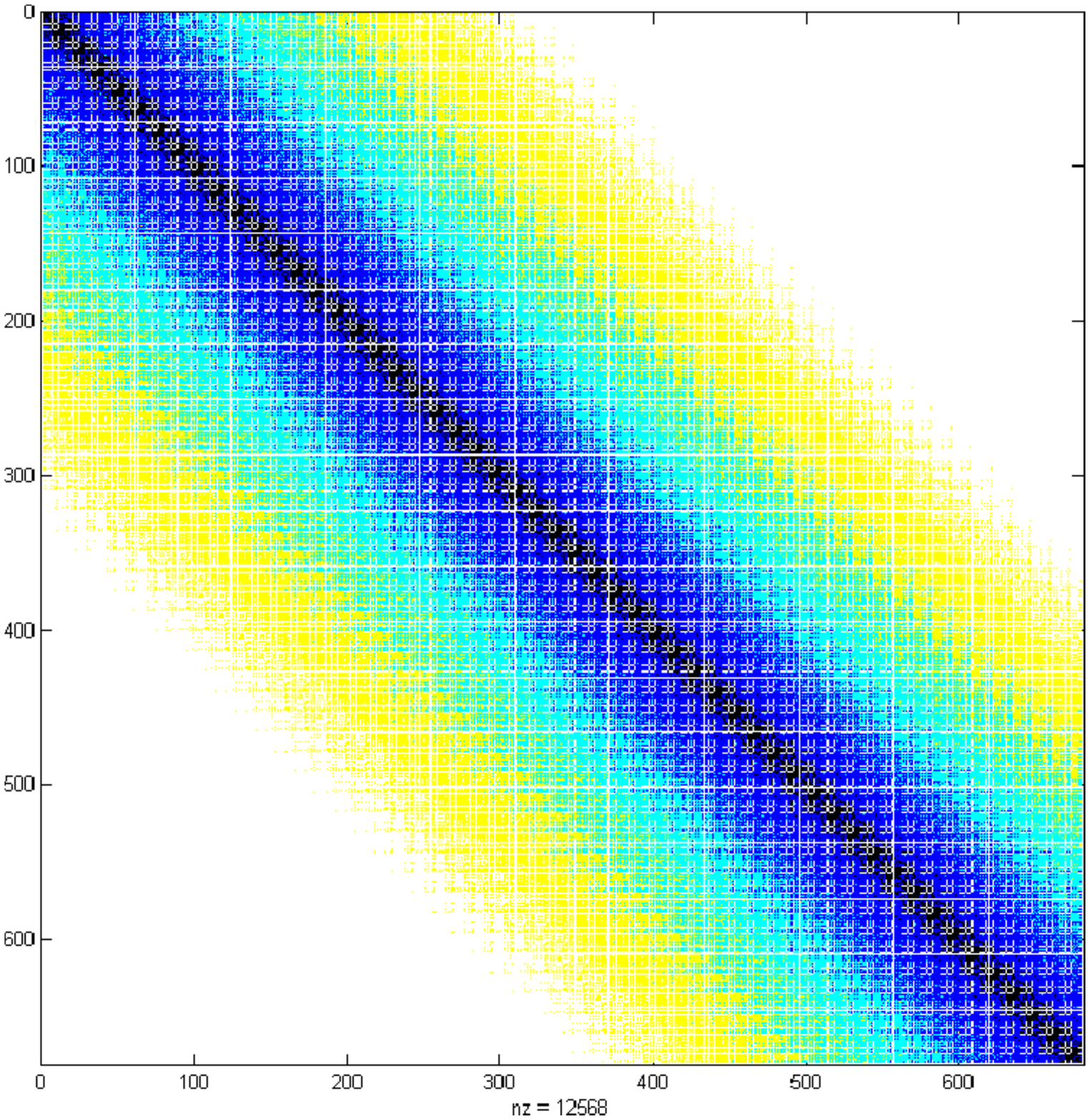}
\end{center}
\vspace{-0.1in}
\caption{Magnitude of the entries in the Hamiltonian
for the $C_{52}H_{106}$ linear alkane. Top: non-orthogonal (GTO) basis.
Bottom: orthogonal basis. White: $<10^{-8}$; yellow:
$10^{-8}-10^{-6}$; green: $10^{-6}-10^{-4}$;
blue: $10^{-4}-10^{-2}$; black: $>10^{-2}$.
 Note: {\tt nz} refers to the number of
`black' entries.}\label{fig_H_jacek}
\end{figure}
 
One last issue to be addressed is whether the transformation to an orthonormal
basis should be effected via the inverse Cholesky factor or via the L\"owdin (inverse
square root) factor of the overlap matrix. 
Comparing the decay bounds for the two factors suggests that the inverse 
Cholesky factor should be preferred (smaller $\alpha$). Also note
that the inverse Cholesky factor is triangular, and its sparsity
can be increased by suitable reorderings of the overlap matrix. The
choice of ordering may also be influenced by the
computer architecture used. We refer to \cite{CS96} for the use
of bandwidth-reducing orderings like reverse Cuthill--McKee, and
to \cite{challacombe2} for the use of space-filling curve orderings
like the 3D Hilbert curve
to improve load balancing and data locality on parallel architectures.
In contrast, the L\"owdin 
factor is a full symmetric matrix, regardless of the ordering. On the other hand,
the multiplicative constant $c$ is generally smaller for the L\"owdin
factor. Closer examination of a few examples suggests that 
in practice there is no great difference in the actual decay
behavior of these two factors. However, approximating $S_n^{-1/2}$ is
generally more expensive and considerably more involved than 
approximating the inverse
Cholesky factor. For the latter, the AINV algorithm \cite{bmt96} and its
variants \cite{challacombe,RBHN08,xiang} are quite efficient and have been successfully
used in various quantum chemistry codes. For other $O(n)$ algorithms for
transformation to an orthonormal basis, see \cite{jansik,ozaki,schweizer}.
In all these algorithms, sparsity is preserved by dropping small entries
in the course of the computation. Explicit decay bounds for the $Z_n$
factors could be used, in principle, to establish {\em a priori} which
matrix elements {\em not} to compute, thus reducing the amount
of overhead. Notice, however, that even if asymptotically bounded,
the condition numbers $\kappa_2(S_n)$ can be fairly large,
leading to rather pessimistic decay estimates. This is again
perfectly analogous to the
situation with the condition number-based error bounds for the conjugate gradient 
(CG) method
applied to a linear system $Ax=b$. And indeed,
both the CG error bounds and the estimates (\ref{dms}) are obtained 
using Chebyshev
polynomial approximation for the function $f(\lambda) = \lambda^{-1}$.

\section{The vanishing gap case}\label{metals}
In this section we discuss the case of a sequence $\{H_n\}$ of
bounded, finite range Hamiltonians for which the spectral gap
around the Fermi level $\mu$ vanishes as $n\to \infty$. Recall
that this means that $\inf_n \gamma_n = 0$, where 
$\gamma_n := \varepsilon_{n_e+1}^{(n)} - \varepsilon_{n_e}^{(n)}$ 
is the HOMO-LUMO gap for the $n$-th Hamiltonian; it is assumed here
that $\varepsilon_{n_e}^{(n)} < \mu < \varepsilon_{n_e+1}^{(n)}$ for all 
$n=n_b\cdot n_e$. The reciprocal $\gamma_n^{-1}$ of the gap
can be interpreted as the condition number of the problem \cite{ruben}, 
so a vanishing spectral gap means that the conditioning
deteriorates as $n_e\to \infty$ and the problem becomes
increasingly difficult.

As already mentioned, in the zero-temperature 
limit our decay bounds blow up, and therefore
lose all meaning as $\gamma_n \to 0$. On the other
hand, we know {\em a priori} that some type of decay should be
present, in view of the results in section \ref{proj}. 
A general treatment of the vanishing gap case appears to be
rather difficult. The main reason is that in the limit as
$\beta \to \infty$ the Fermi--Dirac approximation to the
Heaviside function becomes discontinuous, and therefore
we can no longer make use of tools from classical approximation theory 
for analytic functions. Similarly, in the vanishing gap case
the decay bounds (\ref{cauchy_bd}) based on the resolvent   
estimates (\ref{res_est}) break down since $c\to \infty$ and $\lambda \to 1$
in (\ref{cc_bound}).

Rather than attacking the problem in general, in
this section we give a complete
analysis of what is perhaps the simplest nontrivial example
of a sequence $\{H_n\}$ with vanishing gap. While this is 
only a special case, this example captures some of the 
essential features of the `metallic' case, such as the
rather slow off-diagonal decay of the entries of the density 
matrix. 
The simple model studied in this section may appear 
at first sight to be too simple
and unrealistic to yield any useful information about actual
physical systems. However, calculation of the
density matrix at zero temperature on a system composed of
500 Al atoms reported in \cite{zhang} reveals a decay behavior
which is essentially identical to that obtained analytically
for a free electron gas, a model very close to ours (which is
essentially a discrete variant of the one in \cite{zhang}). 
We believe that our analysis will shed some light on
more general situations in which a slowly decaying density matrix
occurs.

We begin by considering the infinite tridiagonal Toeplitz matrix
\begin{equation}\label{toep}
H=\left(\begin{array}{cccccc}
0 & \frac{1}{2} & & &&\\
\frac{1}{2}  & 0 & \frac{1}{2} &&&\\
& \ddots & \ddots & \ddots &&\\
&& \frac{1}{2} & 0 & \frac{1}{2}&\\
&&& \ddots  & \ddots & \ddots \end{array}\right)\,, 
\end{equation}
which defines a bounded, banded, self-adjoint operator
on ${\ell }^2$. The graph of this matrix is just a
(semi-infinite) path. The operator can be interpreted
as an averaging operator, or as a centered second-difference 
operator with a zero Dirichlet condition at one end, shifted
and scaled so as to have spectrum contained in $[-1,1]$.
From a physical standpoint, $H$ is the shifted and scaled
discrete one-electron Hamiltonian where the electron is constrained
to the half-line $[0,\infty)$. 

 For $n$ even ($n=2\cdot n_e$, with $n_e\in \naturals$)
consider the $n$-dimensional approximation 
\begin{equation}\label{toep_n}
H_n=\left(\begin{array}{ccccc}
0 & \frac{1}{2}  & &&\\
\frac{1}{2}  & 0 & \frac{1}{2} &&\\
& \ddots & \ddots & \ddots &\\
&& \frac{1}{2} & 0 & \frac{1}{2} \\
&&& \frac{1}{2} & 0 \end{array}\right)\,.
\end{equation}
This corresponds to truncating the semi-infinite path and imposing
zero Dirichlet conditions at both ends.
Let now $\{e_1, e_2, \ldots \}$ denote the standard basis of
${\ell }^2$, and let $\hat I$ denote the identity operator
restricted to the subspace of ${\ell }^2$ spanned by $e_{n+1}, e_{n+2},
\ldots$. Letting
$$H_{(n)}:= \left(\begin{array}{cc}
H_n & 0 \\
0   &\hat I\end{array}\right)\,,
$$
the sequence $\{H_{(n)}\}$ is now a sequence of bounded self-adjoint
linear operators on ${\ell }^2$ that converges strongly to $H$.
Note that $\sigma (H_n)\subset [-1,1]$ for all $n$; also,
$0\notin \sigma (H_n)$ for all even $n$. It is easy to see
that half of the eigenvalues of $H_n$ lie in $[-1,0)$ and the
other half in $(0,1]$. We set $\mu=0$ and we label as `occupied'
the states corresponding to negative eigenvalues.
The spectral gap of each $H_n$ is then $\varepsilon_{n/2 + 1}^{(n)}
-\varepsilon_{n/2}^{(n)}$. 

The eigenvalues and eigenvectors of $H_n$ are known
explicitly \cite[Lemma 6.1]{demmel}. Indeed, the eigenvalues, in
descending order, are given by
$\varepsilon_k^{(n)}=\cos \left (\frac{k\pi}{n+1}\right )$ (with $1\le k\le n$)
and the corresponding normalized eigenvectors are
given by $v_k^{(n)} = (v_k^{(n)}(j))$ with entries 
$$v_k^{(n)}(j) = \sqrt{\frac{2}{n+1}} \sin \left (\frac{jk\pi}{n+1}\right ),
\quad 1\le j\le n.$$ 
Note that the eigenvalues are symmetric with respect to the origin,
and that the spectral gap at 0 vanishes, since 
$\varepsilon_{n/2 + 1}^{(n)} = -\varepsilon_{n/2}^{(n)} \to 0$
as $n\to \infty$. We also point to the well known fact that the
eigenvectors of this operator are strongly delocalized. Nevertheless,
as we will see, some localization (decay) is present in the
density matrix, owing to cancellation (i.e., destructive interference).

Now let $P_n$ be the zero-temperature density matrix associated
with $H_n$, i.e., the spectral projector onto the subspace of $\complex^n$
spanned by the eigenvectors of $H_n$ associated with the lowest $n_e$
eigenvalues (the occupied subspace). We extend $P_n$ to a projector
acting on ${\ell }^2$ by embedding $P_n$ into an infinite
matrix $P_{(n)}$ as follows:
$$P_{(n)}:= \left(\begin{array}{cc}
P_n & 0 \\
0   & 0\end{array}\right)\,.
$$
Note that $P_{(n)}$ is just the orthogonal projector
onto the subspace of ${\ell }^2$ spanned by
the eigenvectors of $H_{(n)}$ associated with eigenvalues in the
interval $[-1,0)$. Moreover, ${\rm Tr}(P_{(n)})=
{\rm Tr}(P_n) = {\rm rank}(P_n) = \frac{n}{2} = n_e$.
The limiting behavior of the sequence
$\{P_{(n)}\}$ (hence, of $\{P_n\}$) is completely 
described by the following result.

\vspace{0.1in}

\begin{theorem}\label{no_gap}
Let $H$, $H_n$, and $P_{(n)}$ be as described above. Then
\begin{itemize}
\item[(i)] $H$ has purely absolutely continuous spectrum,\footnote{ 
The absolutely continuous spectrum of a self-adjoint linear operator
$H$ on a Hilbert space $\mathscr{H}$ is the spectrum of the restriction of $H$ to  
the subspace $\mathscr{H}_{\rm ac}\subseteq \mathscr{H}$ 
of vectors $\psi$ whose spectral measures $\mu_{\psi}$ are
absolutely continuous with respect to the Lebesgue measure.
For details, see 
\cite[pages 224--231]{RS}.}
given
by the interval $[-1,1]$. In particular, $H$ has no eigenvalues.
\item[(ii)] The union of the spectra of the $n$-dimensional sections
$H_n$ of $H$ is everywhere dense in $\sigma (H)=[-1,1]$. In other
words, every point in $[-1,1]$ is the limit of a sequence of the
form $\{\varepsilon_k^{(n)}\}$ for $n\to \infty$, where 
$\varepsilon_k^{(n)}\in \sigma (H_n)$ and $k=k(n)$.
\item[(iii)] The sequence $\{H_n\}$ has vanishing gap: $\inf_n \gamma_n=0$.
\item[(iv)] The spectral projectors $P_{(n)}$ converge strongly
to $P=h(H)$ where $h(x)=\chi_{[-1,0)} (x)$, the characteristic
function of the interval $[-1,0)$. 
\item[(v)] $P$ is the orthogonal projector onto an infinite-dimensional 
subspace of ${\ell}^2$.  
\end{itemize}
\end{theorem}
\begin{proof}
Statements (i)-(ii) are straightforward consequences of classical
results on the asymptotic eigenvalue distribution of Toeplitz matrices,
while (iv)-(v) follow from general results in spectral theory.
Statement (iii) was already noted (the eigenvalues of $H_n$ are
explicitly known) and it also follows from (i)-(ii).
 More in detail, statement (i) is a special case of
Rosenblum's Theorem on the spectra of infinite banded Toeplitz
matrices: see \cite{rosenblum} or \cite[Thm.~1.31]{BG}. For
the fact that the spectrum of $H$ coincides with the interval
$[-1,1]$ and that the finite section eigenvalues $\varepsilon_k^{(n)}$
are dense in $\sigma (H)=[-1,1]$ (statement (ii)), see 
the paper by Hartman and Wintner \cite{HW} or the book
by Grenander and Szeg\"o \cite[Chapter 5]{GS}.
Statement (iv) can be proved as follows. For a linear
operator $A$ on ${\ell}^2$, write $R_{\lambda}(A) =
(A - \lambda I)^{-1}$, with $\lambda \notin \sigma(A)$. A
sequence $\{A_n \}$ of self-adjoint (Hermitian) operators
 is said to converge in the {\em strong resolvent
sense} to $A$ if $R_{\lambda}(A_n)\longrightarrow R_{\lambda}(A)$ 
strongly for all $\lambda\in \complex$ with ${\rm Re}\, \lambda \ne 0$, that is:
$$\lim_{n\to \infty}\|R_{\lambda}(A_n)x - R_{\lambda}(A)x\| = 0 \quad
{\rm for \,\, all}\quad x\in {\ell}^2\,.$$
It is easy to check, using for instance the results in
\cite[Chapter 2]{BS98}, that the sequence $\{H_n\}$ converges
in the strong resolvent sense to $H$. Statement (iv) (as well as
(ii)) now follows from \cite[Thm.~VIII.24]{RS}. The fact (v) that
$P=h(H)$ is an orthogonal projector onto an infinite-dimensional subspace
of ${\ell}^2$ follows from the fact that $\mu=0$ is not an
eigenvalue of $H$ (because of (i)) and from the spectral theorem
for self-adjoint operators in Hilbert space; see, e.g., \cite[Chapter VII]{RS}
or \cite[Chapter 12]{Rudin}.
\end{proof}

The foregoing result implies that the Toeplitz matrix sequence
$\{H_n\}$ given by (\ref{toep_n}) exhibits some of the 
key features of the discrete Hamiltonians describing metallic systems,
in particular the vanishing gap property and the fact that the
eigenvalues tend to fill the entire energy spectrum. The 
sequence $\{H_n\}$ can be thought of as
a 1D `toy model' that can be solved analytically to gain some
insight into the decay properties of the density matrix 
such systems. 
Indeed,
from the knowledge of the eigenvectors of $H_n$ we can write down
the spectral projector corresponding to the lowest $n_e = n/2$
eigenvalues explicitly. Recalling that the eigenvalues $\varepsilon_k^{(n)}$
are given in descending order, it is convenient to compute
$P_n$ as the projector
onto the orthogonal complement of the subspace spanned by the
eigenvectors corresponding to the $n/2$ largest eigenvalues:
$$P_n = I_n - \sum_{k=1}^{n_e} v_k^{(n)} (v_k^{(n)})^T.$$
The $(i,j)$ entry of $P_n$ is therefore given by
$$[P_n]_{ij} = e_i^T P_n e_j =
\delta_{ij} - \frac{2}{n+1} \sum_{k=1}^{n_e} 
\sin \left (\frac{ik\pi}{n+1}\right )\sin \left (\frac{jk\pi}{n+1}\right ).$$
For $i=j$, we find
\begin{equation}\label{diag}
[P_n]_{ii} = 1- \frac{2}{n+1} \sum_{k=1}^{n_e}
\sin^2 \left (\frac{ik\pi}{n+1}\right ) = \frac{1}{2}, \quad 
{\rm for \,\, all\,\,} i=1,\ldots ,n \,\, {\rm and\,\, for \,\, all\,\,} n.
\end{equation}
Hence, for this system the charge density $P_{ii}$ is constant and the system
essentially behaves like a non-interacting electron gas, see for example \cite{GV}.
We note in passing that this example confirms that the bound
(\ref{off}) is sharp, since equality is attained for this
particular projector. 
Moreover, the trigonometric identity 
\begin{equation}\label{trigid}
\sin\theta \, \sin\phi = -\frac{1}{2}\left [\cos(\theta + \phi) -
\cos(\theta - \phi) \right ]
\end{equation}
implies for all $i,j=1,\ldots ,n$: 
\begin{equation}\label{p_ij}
[P_n]_{ij} = 
\frac{1}{n+1} \sum_{k=1}^{n_e} \left [
\cos \left (\frac{(i+j)k\pi}{n+1} \right ) - 
\cos \left (\frac{(i-j)k\pi}{n+1} \right )\right ]\,. 
\end{equation}
From (\ref{p_ij}) it immediately follows, for all $i$ and
for all $n$, that
\begin{equation}\label{jplus2}
[P_n]_{i,i+2l} = 0\,,\quad l=1,2,\ldots
\end{equation}
Since (\ref{diag}) and (\ref{jplus2}) hold for all $n$, they also hold
in the limit as $n\to \infty$. Hence, the strong limit $P$ of the
sequence of projectors $\{P_{(n)}\}$ satisfies $P_{ii} = 1/2$
and $P_{i,j}=0$ for all $j=i+2l$, where $i,l=1,2,\ldots$
To determine the remaining off-diagonal entries $P_{ij}$ (with
$j\ne i$ and $j\ne i+2l$) we directly compute the limit of
$[P_n]_{ij}$ as $n\to \infty$, as follows. Observe that  
using the substitution $x=k/(n+1)$ and taking the limit as $n\to
\infty$ in (\ref{p_ij}) we obtain for all $i\ge 1$ and for all
$j\ne i+2l$ ($l=0,1,\ldots$):
\begin{equation}\label{integral}
\begin{split}
P_{ij} & = \int_0^{\frac{1}{2}} \cos[(i+j)\pi x] \, dx -
         \int_0^{\frac{1}{2}} \cos[(i-j)\pi x] \, dx \\
       & = \frac{1}{\pi} \left [\frac{(-1)^{\frac{i+j-1}{2}}}{i+j} +
           \frac{(-1)^{\frac{i-j+1}{2}}}{i-j}\right ] \,.
\end{split}
\end{equation}
It follows from (\ref{integral}) that $|P_{ij}|$ is bounded
by a quantity that decays only linearly in the distance from the
main diagonal. As a result, $O(n)$ approximation of $P_n$ for
large $n$ involves a huge prefactor. Therefore, 
from this very simple example we can gain some insight
into the vanishing gap case. The analytical results obtained
show that the density matrix can exhibit rather slow decay,
confirming the well known fact that $O(n)$ approximations
pose a formidable challenge in the vanishing gap case.

The 2D case is easily handled as follows. We consider for 
simplicity the case of a square lattice consisting of $n^2$ points
in the plane. The 2D Hamiltonian is given by
$$ H_{n^2} = \frac{1}{2}(H_n\otimes I_n + I_n\otimes H_n)\,,$$
where the scaling factor $\frac{1}{2}$ is needed so as 
to have $\sigma(H_{n^2})\subset [-1,1]$. The eigenvalues and
eigenvectors of $H_{n^2}$ can be explicitly written in terms
of those of $H_n$; see, e.g., \cite{demmel}. Assuming again that
$n$ is even, exactly half of the $n^2$ eigenvalues of $H_{n^2}$
are negative (counting multiplicities), the other half positive.
As before, we are interested in finding the spectral projector associated
with the eigenvectors corresponding to negative eigenvalues. Note
again that the spectral gap tends to zero as $n\to \infty$. 
If $P_{n^2}$ denotes the spectral projector onto the occupied states,
it is not difficult to show that
\begin{equation}\label{2D_case}
P_{n^2} = P_n\otimes (I_n - P_n) + (I_n - P_n)\otimes P_n\,.
\end{equation}
It follows from (\ref{2D_case}) that the spectral projector $P_{n^2}$
has a natural $n\times n$ block structure, where

\begin{itemize}
\item Each diagonal block is equal to $\frac{1}{2}I_n$;
note that this gives the correct trace, ${\rm Tr}(P_{n^2}) = \frac{n^2}{2}$.
\item The $(k,l)$ off-diagonal block $\Pi_{kl}$ is given by
$\Pi_{kl}=[P_n]_{kl}(I_n - 2P_n)$. Hence, each off-diagonal block has a
`striped' structure, with the main diagonal as well as the third, fifth,
etc.~off-diagonal identically zero. Moreover, every block $\Pi_{kl}$
with $l=k+2m$ ($m\ge 1$) is zero. 
\end{itemize}

This shows that in the 2D case,
the rate of decay in the spectral projector is essentially 
the same as in the 1D case. The 3D case can be handled in a 
similar manner, leading to the same conclusion.

For this simple example we can also compute the entries of
the density matrix at positive electronic 
temperature $T>0$. Recalling that the density matrix in this case is
given by the Fermi--Dirac function with parameter $\beta =
1/(k_B T)$ we have in the 1D case (assuming $\mu=0$)
\begin{equation}\label{posTcase}
P_{ij}  = \frac{2}{n+1} \sum_{k=1}^n 
\frac{\sin \left (\frac{ik\pi}{n+1}\right )\sin \left (\frac{jk\pi}{n+1}\right )}
{1 + {\rm exp} \left [\beta \cos \left (\frac{k\pi}{n+1} \right ) \right ]}\,.
\end{equation}
Making use again of the trigonometric identity (\ref{trigid})
and using the same substitution $x=k/(n+1)$, we can reduce the
computation of the density matrix element $P_{ij}$ for $n\to \infty$
to the evaluation of the following integral:
\begin{equation}\label{int_posTcase}
P_{ij}=\int_0^1 \frac{\cos\, [(i-j)\pi x] - \cos\, [(i+j)\pi x]}
{1 + {\rm exp}\, ( \beta \cos \pi x )} \, dx \,.
\end{equation}
Unfortunately, this integral cannot be evaluated explicitly in terms
of elementary functions. Note, however, that the integral   
$$I_k = \int_0^1 \frac{\cos\, (k\pi x)}{1 + {\rm exp}\, ( \beta \cos \pi x )} \, dx$$
(where $k$ is an integer) becomes, under the change of variable
$\pi x= \arccos t$, 
$$I_k =
\frac{1}{\pi} \int_{-1}^1 \frac{\cos \,(k\,\arccos t)}{1+{\rm e}^{\beta t}}
\frac{dt}{\sqrt{1-t^2}}\,.
$$
Hence, up to a constant factor, $I_k$ is just the $k$th coefficient
in the Chebyshev expansion of the Fermi--Dirac function 
$1/(1+{\rm e}^{\beta t})$. Since the Fermi--Dirac function is analytic
on the interior of an ellipse containing the interval $[-1,1]$ and 
continuous on the boundary of such an ellipse, it follows from the
general theory of Chebyshev approximation that the coefficients
$I_k$ decay at least exponentially fast as $k\to \infty$; see, e.g, 
\cite{meinardus}. This in turn implies that the entries $P_{ij}$
given by (\ref{int_posTcase}) decay at least exponentially fast away from 
the main diagonal, the faster the larger the temperature is,
as already discussed in section \ref{dep_T}. 
Hence, for this special case
we have established in a more direct way the exponential decay behavior
already proved in general in section \ref{sec_fd}. In the present case, 
however, for any value of $\beta$ the decay rate of
the entries $P_{ij}$ given by (\ref{int_posTcase}) can be determined
to arbitrary accuracy by numerically computing the Chebyshev coefficients 
of the Fermi--Dirac
function.

We mention that a simple, one-dimensional model of a system with arbitrarily
small gap has been described in \cite{GLXE}. The
(continuous) Hamiltonian in \cite{GLXE} 
consists of the kinetic term plus a potential given by
a sum of Gaussian wells located at the
nuclei sites $X_i$: 
$$ {\mathcal H} = -\frac{1}{2}\frac{d^2}{dx^2} + V(x), \quad
V(x) = -\sum_{i=-\infty}^\infty \frac{a}{\sqrt{2\pi \sigma^2}}\, {\rm exp}\,(
-(x-X_i)^2/2\sigma^2)\,,$$
with $a>0$ and $\sigma>0$ tunable parameters. 
The spectra of this family of Hamiltonians present a band structure
with band gap proportional to $\sqrt{a}/\sigma$. Note that the model
reduces essentially to ours for $a\to 0$ and/or for $\sigma \to \infty$.
On the other hand, while the gap can be made arbitrarily small
by tuning the parameters in the model, for any choice of 
$\alpha >0$ and $\sigma >0$ the gap does not vanish; 
therefore, no approximation of
the infinite-size system with a sequence
of finite-size ones can lead to a vanishing gap
in the thermodynamic limit. This means that our bounds, when
applied to this model, will yield exponential decay, albeit 
very slow (since the correlation lengths will be quite large
for small $a\to 0$ and/or for large $\sigma$). 
The model in \cite{GLXE}, on the other hand, can be useful
for testing purposes when developing algorithms for metal-like
systems with slowly decaying density matrices.

\section{Other applications}\label{other}
In this section we sketch a few possible applications of our decay results to
areas other than electronic structure computations. 

\subsection{Density matrices for thermal states}
In quantum statistical mechanics, the equilibrium density matrix for a system of particles
subject to a heat bath at absolute temperature $T$ is defined as 
\begin{equation} \label{heath}
P = \frac {{\rm e}^{-\beta H}}{Z}, \quad {\rm where} \quad Z={\rm Tr}\,({\rm e}^{-\beta H})\,.
\end{equation} 
As usual, $\beta = (k_B T)^{-1}$ where $k_B$ denotes the Boltzmann constant;
see \cite{pathria}.
The matrix $P$ is the quantum analog of the canonical Gibbs state.
The Hamiltonian $H$ is usually assumed to have been shifted so that the smallest eigenvalue
is zero \cite[page 112]{Mac04}. 
Note that $P$ as defined in (\ref{heath}) is not an orthogonal projector.
It is, however, Hermitian 
and positive semidefinite. Normalization by the {\em partition
function} $Z$ ensures that $\sigma(P) \subset [0,1]$ and that ${\rm Tr}(P) = 1$.  

It is clear that for increasing temperature, i.e., for  $ T\to \infty$
(equivalently, for $\beta \to 0$) the canonical density matrix $P$ approaches the
identity matrix, normalized by the matrix size $n$. In particular, the 
off-diagonal entries tend to zero. The physical interpretation of this is that 
in the limit of large temperatures the system states become totally uncorrelated.
For temperatures approaching the absolute zero, on the other hand, the canonical
matrix $P$ tends to the orthogonal projector associated with the zero eigenvalue
(ground state). In this limit, the correlation between state $i$ and state $j$
is given by the $(i,j)$ entry of the orthogonal projector onto the eigenspace
corresponding to the zero eigenvalue, normalized by $n$. 

For finite, positive values of $T$, the canonical density matrix $P$ is full
but decays away from the main diagonal (or, more generally,
away from the sparsity pattern of $H$). The rate of decay depends on $\beta$:
the smaller it is, the faster the decay. Application of the bounds developed
in section \ref{sec_main} to the matrix exponential is straightforward. For instance,
the bounds based on Bernstein's Theorem take the form
\begin{equation}\label{exp_dec_bnds}
|[{\rm e}^{-\beta H}]_{ij}| \le C(\beta)\,{\rm e}^{-\alpha\, d(i,j)}, 
\quad i\ne j,
\end{equation}
where
$$ C(\beta) = \frac{2\, \chi}{\chi - 1}\,{\rm e}^{\beta\, \|H\|_2\, (\kappa_1 - 1)/2}
\quad {\rm and} \quad \alpha = 2 \ln \chi.$$
In these expressions, $\chi >1$ and $\kappa_1 >1$ are the parameters associated
with the Bernstein ellipse with foci in $-1$ and $1$ and major semi-axis 
$\kappa_1$, as described in section
\ref{sec_main}. Choosing $\chi$ large makes the exponential term
decay ${\rm e}^{-\alpha\, d(i,j)}$ very fast, but causes $C(\beta)$ to grow larger.
Clearly, a smaller $\beta$ makes the upper bound (\ref{exp_dec_bnds}) smaller. 
Bounds on the entries of the canonical density matrix $P$ can be obtained
dividing through the upper bounds by $Z$. Techniques for estimating $Z$
can be developed using the techniques described in \cite{MMQ}; see also
\cite{benziboito}.

Although the bound (\ref{exp_dec_bnds}) is an exponentially decaying one,
it can be shown that the decay in the entries of a banded or sparse matrix
is actually {\em super}-exponential. This can be shown by expanding the
exponential in a series of Chebyshev polynomials and using the fact that
the coefficients in the expansion, which can be expressed in terms of
Bessel functions, decay to zero super-exponentially; see \cite{meinardus}  
and also \cite{Ise00}. The decay bounds obtained in this way are, 
however, less transparent and more complicated to evaluate than 
(\ref{exp_dec_bnds}).

Finally, exponential decay bounds for spectral projectors
and other matrix functions might provide a
rigorous justification for $O(n)$ algorithms recently
developed for disordered systems; see, e.g., \cite{Sac04,Sac05}. 

\subsection{Quantum information theory}
A related area of research where our decay bounds
for matrix functions have proven useful is
the study of quantum many-body systems in
information theory; see, e.g., \cite{cramer,CEPD06,ECB10,Schuch07,SCW06}.  
In particular, relationships between spectral gaps and rates of
decay for functions of finite range Hamiltonians have been established in
\cite{cramer} using the techniques introduced in \cite{benzigolub}. 
The exponential decay of correlations and its relation to the spectral
gap have also been studied in \cite{hast1,hast2}.

As shown in \cite{CEPD06}, exponential decay bounds for matrix
functions play a crucial role
in establishing so-called {\em area laws} for the {\em entanglement
entropy} of ground states associated with bosonic systems. 
These area laws essentially state that the entanglement
entropy associated with a 3D bosonic lattice is proportional
to the surface area, rather than to the volume, of the lattice.
Intriguingly, such area laws are analogous to those governing
the Beckenstein--Hawking black hole entropy.
We refer the interested reader to
the recent, comprehensive survey paper \cite{ECB10} for 
additional information.

\subsection{Complex networks}
The study of complex networks is an emerging field of science currently
undergoing vigorous development. Researchers in this highly interdisciplinary 
field
include mathematicians, computer scientists, physicists, chemists, engineers,
neuroscientists, biologists, social scientists, etc. Among the
mathematical tools used in this field,  linear algebra
and graph theory, in particular
spectral graph theory, play a major role. 
Also, statistical mechanics concepts and
techniques have been found to be ideally suited to the study 
of large-scale networks.

 In recent years, quantitative methods of
network analysis have increasingly made use of matrix functions. 
This approach has been spearheaded in the works of Estrada,
Rodr\' iguez-Vel\' azquez,
D.~Higham and Hatano; see, e.g., \cite{E10,E12,EH08,EH09,EH10,ERV05},
as well as the recent surveys \cite{estrada,EHB12} 
and the references therein.
Functions naturally arising in the context of network analysis include
the exponential, the resolvent, and hyperbolic functions, among others.
Physics-based justifications for the use of these matrix functions
in the analysis of complex networks have been thoroughly
discussed in \cite{EHB12}.

For example, the exponential of the adjacency matrix $A$ associated
with a simple, undirected graph $G=(V,E)$ can be used to give natural
definitions of important measures associated with nodes in $G$, 
such as the {\em subgraph centrality} associated with node $i$,
defined as $C(i)=[{\rm e}^A]_{ii}$, and the {\em communicability} associated
with two distinct nodes $i$ and $j$, defined as $C(i,j)=[{\rm e}^A]_{ij}$.
Other network quantities that can be expressed in terms of the entries in
appropriate matrix functions of $A$ include betweenness, returnability,
vulnerability, and so forth. The graph Laplacian $L=D-A$, where 
$D={\rm diag}(d_1,\ldots ,d_n)$ with $d_i$ denoting the degree of
node $i$, is sometimes used instead of the adjacency matrix, as
well as weighted analogues of both $A$ and $L$. 

Most networks arising in real-world applications are sparse,
often with degree distributions closely approximated by power laws.
Because the maximum degree in such \lq\lq scale-free\rq\rq\  networks   
increases as the number of nodes tends to infinity, one cannot
expect uniform exponential decay rates to hold asymptotically for the matrix
functions associated with such graphs unless additional
structure is imposed, for instance in the
form of weights. Nevertheless, 
our bounds for the entries of functions of sparse matrices can be used
to obtain estimates on quantities such as the communicability between 
two nodes. 
A discussion of locality (or the lack thereof) in matrix 
functions used in the analysis of complex networks can be found 
in \cite{EHB12}.
We also refer the reader to \cite{benziboito} for a description
of quadrature rule-based
bounds for the entries of matrix functions associated with complex networks.

\subsection{Tridiagonal eigensolvers}

The solution of symmetric tridiagonal eigenvalue problems plays an important
role in many field of computational science. As noted for example in
\cite{VP}, solving such problems is key for most dense real symmetric
(and complex Hermitian) eigenvalue computations and therefore plays a
central role in standard linear algebra libraries such as LAPACK and
ScaLAPACK. Even in the sparse case, the symmetric tridiagonal
eigenvalue problem appears as a step in the Lanczos algorithm. 

The
efficiency of symmetric tridiagonal eigensolvers can be significantly
increased by exploiting localization in the eigenvectors (more 
generally, invariant subspaces) associated with an isolated cluster
of eigenvalues. It would be highly desirable to identify beforehand
any localization in the eigenspace in a cost-effective manner, as
this would lead to reduced computational costs \cite{parlett,VP}. It is clear that 
this problem is essentially the same as the one considered in this
paper, with the additional assumption that the matrix $H$ is
tridiagonal. Given estimates on the location of the cluster of
eigenvalues and on the size of the gaps separating it from the remainder of
the spectrum, the techniques described in this paper can be used to
bound the entries in the spectral projector associated to the cluster
of interest; in turn, the bounds can be used to identify banded approximations
to the spectral projectors with guaranteed prescribed error.
Whether the estimates obtained in this manner are accurate
enough to lead to practical algorithms with substantially improved
run times and storage demands over current ones remains an open
question for further research.
 
Finally, in the recent paper \cite{ye12} the exponential
decay results in \cite{benzigolub} are used to derive error 
bounds and stopping
criteria for the Lanczos method applied to
the computation of ${\rm e}^{-tA}v$, where $A$ is a large symmetric
positive definite matrix, $v$ is a vector, and $t>0$. The bounds are
applied to the exponential of the tridiagonal matrix $T_k$ generated
after $k$ steps by the Lanczos process  in order to obtain the
approximation error after $k$ steps.

\subsection{Non-Hermitian extensions}
Although the main focus of the paper has been the study of functions
of sparse Hermitian matrices, many of our results can be
extended, under appropriate conditions, to non-Hermitian matrices.  
The generalizations of our decay bounds to {\em normal} matrices,
including for example skew-Hermitian ones, is relatively straightforward;
see, e.g., the results in \cite{benzirazouk} and \cite{nadersthesis}.
Further generalizations to diagonalizable matrices have been
given in \cite{benzirazouk}, although the bounds now contain
additional terms taking into account the departure from
normality. 
These bounds may be difficult to use in practice, as
knowledge of the eigenvectors or of the field of values
of the matrix is needed.
Bounds for functions of general
sparse matrices can also be obtained using contour integration;
see, e.g., \cite{nadersthesis} and \cite{mastronardi}.
It is quite possible that
these bounds will prove useful in applications involving
functions of sparse, non-normal matrices. Examples include
functions of digraphs in network analysis, like returnability,
or functions of the Hamiltonians occurring in
the emerging field of non-Hermitian quantum mechanics;
see, respectively, \cite{EH09} and \cite{BBM99,BBJ03,NHQM}.

\section{Conclusions and open problems}\label{concl}
In this paper we have described a general theory of localization
for the density matrices associated with certain sequences of banded or
sparse discrete Hamiltonians of increasing size. We have obtained, under
very general conditions, exponential decay bounds for
the off-diagonal entries of zero-temperature density matrices for
gapped systems (`insulators') and for density matrices associated with 
systems at positive electronic temperature. The theory, while purely
mathematical, 
recovers well-known physical phenomena such as
the fact that the rate of decay is faster at higher
temperatures and for larger gaps, and even captures the
correct asymptotics for small gaps and low temperatures.
Thus, we have provided a theoretical
justification for the development of $O(n)$ methods 
for electronic structure computations. As an integral
part of this theory,
we have also surveyed the approximation of rapidly 
decaying matrices by banded or sparse ones, the effects of
transforming a Hamiltonian from a non-orthogonal to an 
orthogonal basis, and some general properties of orthogonal
projectors. 

In the case of zero-temperature and vanishing gaps, our bounds
deteriorate for increasing $n$. In the limit as $n\to \infty$ 
we no longer have exponentially decaying bounds, which is entirely
consistent with the physics. For metallic
systems at zero temperature the decay in the spectral
projector follows a power law, and we have exhibited
a simple model Hamiltonians for which the decay 
in the corresponding density matrix is only linear in the distance
from the main diagonal. 

Because of the slow decay, the development
of $O(n)$ methods in the metallic case at zero temperature is
problematic. We refer the reader to \cite{baerhg2,bowler1,lin2,watson} for some attempts
in this direction, but the problem remains essentially open. 
In the metallic case it may be preferable 
to keep $P$ in the factorized form 
$P= XX^*$, where $X\in \complex^{n\times n_e}$ 
is any matrix whose columns span the occupied subspace,
and to seek a
maximally localized $X$. Note that 
$$ P = XX^* = (XU)(XU)^*$$
 for any unitary $n_e\times n_e$
matrix $U$, so the question is whether the occupied subspace admits
a set of basis vectors that can be rotated 
so as to become as localized as possible. Another 
possibility is to research the use of rank-structured approximations
(such as hierarchical matrix techniques \cite{hackbusch}) to the spectral
projector. Combinations of tensor product
approximations and wavelets appear to be promising. We refer
here to \cite{goedeckerivanov} for a study of the decay properties
of density matrices in a wavelet basis (see also \cite{SW06}), and 
to \cite{BCM99} for an early attempt to exploit near low-rank
properties of spectral projectors. See  also
the more recent works by W.~Hackbusch and collaborators 
\cite{chinnam1,chinnam3,chinnam2,flad1,flad2,luo}. 

Besides the motivating application of electronic structure, our theory
is also applicable to other problems where localization plays a prominent
role. We hope that this paper will stimulate further research in this
fascinating and important area at the crossroads of mathematics,
physics, and computing.

\vspace{0.2in}

{\bf Acknowledgements.} 
We are indebted to three anonymous referees for carefully reading the
original manuscript and for suggesting a number of corrections
and improvements. Thanks also
to the handling Editor, Fadil Santosa, for useful feedback and
for his patience and understanding during the several months
it took us to revise the paper.
We would also like to acknowledge useful discussions with David Borthwick, 
Matt Challacombe, Jean-Luc Fattebert, Roberto Grena, Daniel 
Kressner and Maxim Olshanskii. Finally, we are
grateful to Jacek Jakowski for providing the
data for the linear alkane. 
\medspace


\end{document}